\documentclass[11pt, reqno]{amsart}
\pdfoutput=1

\usepackage{amsmath, amssymb, latexsym, mathtools, stmaryrd, color}
\usepackage{genyoungtabtikz}
\usepackage{mathrsfs}
\usepackage{enumerate}
\usepackage{color, xcolor}
\usepackage[all,knot]{xy}

\usepackage[pagebackref=true,bookmarks=true,colorlinks=true,linktoc=page,citecolor=darkgreen,linkcolor=blue,urlcolor=mediumblue]{hyperref}
\definecolor{mediumblue}{rgb}{0.0, 0.0, 0.8}
\colorlet{darkgreen}{green!50!black}

\renewcommand*{\backref}[1]{}
\renewcommand*{\backrefalt}[4]{%
 \ifcase #1 No citations.
 \or [Page #2.]
 \else [Pages #2.]
 \fi%
}
\usepackage{cleveref}

\usepackage[normalem]{ulem}  

\newtheorem{thm}{Theorem}[section]
\newtheorem{prop}[thm]{Proposition}
\newtheorem{cor}[thm]{Corollary}
\newtheorem{lem}[thm]{Lemma}
\newtheorem{conj}[thm]{Conjecture}

\theoremstyle{definition}
\newtheorem{defn}[thm]{Definition}

\newtheorem{ex}[thm]{Example}

\theoremstyle{remark}
\newtheorem{Rmk}[thm]{Remark}

\newenvironment{red}
{\relax\color{red}}
{\hspace*{.5ex}\relax}

\newcommand{\ber}{\begin{red}}
\newcommand{\er}{\end{red}}

\numberwithin{equation}{section}

\newcommand{\rank}{\operatorname{rank}}
\newcommand{\Frac}{\operatorname{Frac}}

\newcommand{\Z}{\mathbb{Z}}
\newcommand{\Q}{\mathbb{Q}}

\newcommand{\A}{\mathbb{A}}
\newcommand{\F}{\mathbb{F}}

\newcommand{\g}{\mathfrak{g}}

\newcommand\la\lambda
\newcommand\La\Lambda

\newcommand{\Hom}{\mathrm{Hom}}

\newcommand{\Ht}{{\rm ht}}
\newcommand{\wt}{{\rm wt}}

\newcommand{\Span}{{\rm Span}}

\newcommand{\proj}{\text{-}\mathrm{proj}}
\newcommand{\im}{{\rm im}}
\newcommand{\coker}{{\rm coker}}

\newcommand{\rlQ}{\mathsf{Q}}   
\newcommand{\wlP}{\mathsf{P}}   
\newcommand{\weyl}{\mathsf{W}}  
\newcommand{\cmA}{\mathsf{A}}  

\newcommand{\Par}{\mathscr{P}^l}   
\DeclareMathOperator{\cont}{cont} 
\newcommand{\ST}[1]{\mathsf{Std}(#1)}   
\newcommand{\rST}[1]{\mathsf{RowStd}(#1)}   
\newcommand{\rT}[1]{\mathsf{Row}(#1)}   
\newcommand{\res}[1]{\mathrm{res}(#1)}   
\newcommand{\sg}{\mathfrak{S}}   
\newcommand{\belt}{\mathbf{B}}   
\newcommand\gr{\Yfillcolour{black!15!white}} 
\newcommand\gre{\Yfillcolour{black!40!white}} 
\newcommand\wh{\Yfillcolour{white}} 

\newcommand{\ch}{{\rm ch}}   
\newcommand{\bR}{\mathcal{O}}   
\newcommand{\fqH}[1]{R^{\Lambda_{#1}}}   
\newcommand{\Sp}{\mathcal{S}}   
\newcommand{\Pe}{\mathcal{M}}   
\newcommand{\Ga}{\mathcal{G}}   
\newcommand{\Lm}{\mathcal{L}}   
\newcommand{\Gh}{\mathcal{H}}   

\newcommand{\conv}{\circ}   
\newcommand{\gLm}{\mathfrak{l}}   
\newcommand{\rg}{ \mathsf{g}}   
\newcommand{\id}{ \mathrm{id}}   
\newcommand{\shf}{ \mathsf{sh}}   
\newcommand{\gmod}{\text{-}\mathrm{gmod}}
\newcommand{\Mod}{\text{-}\mathrm{Mod}}
\newcommand\ttg{\mathtt{G}}

\renewcommand\geq\geqslant
\renewcommand\leq\leqslant

\renewcommand\ge\geqslant
\renewcommand\le\leqslant

\newcommand{\sym}{\mathfrak S_}

\newcommand{\dom}{\trianglerighteqslant}
\newcommand{\doms}{\vartriangleright}

\newcommand{\domsby}{\vartriangleleft}

\newcommand\tts{\mathtt{S}}
\newcommand\ttt{\mathtt{T}}

\newcommand\tabupto[2]{#1_{\downarrow#2}}
\newcommand\shp[2]{\operatorname{Shp}(\tabupto{#1}{#2})}

\crefname{defn}{Definition}{Definitions}
\crefname{thm}{Theorem}{Theorems}
\crefname{prop}{Proposition}{Propositions}
\crefname{lem}{Lemma}{Lemmas}
\crefname{cor}{Corollary}{Corollaries}
\crefname{conj}{Conjecture}{Conjectures}
\crefname{section}{Section}{Sections}
\crefname{subsection}{Subsection}{Subsections}

\Crefname{defn}{Definition}{Definitions}
\Crefname{thm}{Theorem}{Theorems}
\Crefname{prop}{Proposition}{Propositions}
\Crefname{lem}{Lemma}{Lemmas}
\Crefname{cor}{Corollary}{Corollaries}
\Crefname{conj}{Conjecture}{Conjectures}
\Crefname{section}{Section}{Sections}
\Crefname{subsection}{Subsection}{Subsections}

\allowdisplaybreaks
\DeclareRobustCommand\longtwoheadrightarrow
     {\relbar\joinrel\twoheadrightarrow}
\DeclareRobustCommand{\longhookrightarrow}{\lhook\joinrel\longrightarrow}
     
\begin{document}

\title[Specht modules for quiver Hecke algebras]
{Specht modules for quiver Hecke algebras of type $C$}

\author[Susumu Ariki]{Susumu Ariki$^1$}
\thanks{$^1$ S.A. is supported by the JSPS Grant-in-Aid for Scientific Research 15K04782.}
\address{Department of Pure and Applied Mathematics, Graduate School of Information
Science and Technology, Osaka University, Suita, Osaka 565-0871, Japan}
\email{ariki@ist.osaka-u.ac.jp}

\author[Euiyong Park]{Euiyong Park$^2$ }
\thanks{$^2$ E.P. is supported by the National Research Foundation of Korea(NRF) Grant funded by the Korean Government(MSIP)(NRF-2014R1A1A1002178).}
\address{Department of Mathematics, University of Seoul, Seoul 02504, Korea}
\email{epark@uos.ac.kr}

\author[Liron Speyer]{Liron Speyer$^3$}
\thanks{$^3$ L.S. was an International Research Fellow of the Japan Society for the Promotion of Science while this research was conducted.}
\address{Department of Mathematics, University of Virginia, Charlottesville, VA 22904, USA}
\email{l.speyer@virginia.edu}

\subjclass[2010]{05E10, 16G10, 81R10}
\keywords{Categorification, Cyclotomic quiver Hecke algebras, Specht modules}

\begin{abstract}

We construct and investigate Specht modules $\Sp^\la$ for cyclotomic quiver Hecke algebras in type $C^{(1)}_\ell$ and $C_\infty$, 
which are labelled by multipartitions $\la$.
It is shown that in type $C_\infty$, the Specht module $\Sp^\la$ has a homogeneous basis indexed by standard tableaux of shape $\la$, which yields a 
graded character formula and good properties with the exact functors $E_i^\Lambda$ and $F_i^\Lambda$.
For type $C^{(1)}_\ell$, we propose a conjecture.
\end{abstract}

\maketitle

\section*{Introduction}

Representations of Hecke algebras and the symmetric group have been studied for over a century 
and Specht modules play important roles in the representation theory. Nowadays, we realise that on the one hand
the Hecke algebras generalise to \emph{cyclotomic quiver Hecke algebras} (or \emph{Khovanov--Lauda--Rouquier algebras}) 
in the direction of categorification of quantum groups $U_q(\g)$ \cite{KL09, KL11, R08}, and on the other hand that 
the Hecke algebras are cellular algebras and Specht modules are their cell modules.

In the affine type $A_\ell^{(1)}$ case, cellular algebras and Specht modules for cyclotomic quiver Hecke algebras were studied via the isomorphism to cyclotomic Hecke algebras given in \cite{BK09a}. 
It was shown that the cyclotomic quiver Hecke algebras of affine type $A_\ell^{(1)}$ have graded cellular structure \cite{HM10}. 
Graded Specht modules in type $A_\ell^{(1)}$ were constructed and studied using the combinatorics of multipartitions \cite{BKW11, kmr}.
But so far little is known about Specht modules for \emph{cyclotomic} quiver Hecke algebras of \emph{other types}. 
We remark that it was proved in  \cite{KleLou15} that quiver Hecke algebras of finite type are \emph{graded affine cellular algebras}. 
 
The first and second authors \cite{AP16} studied cyclotomic quiver Hecke algebras $\fqH0(n)$ 
for the fundamental weight $\Lambda_0$ in type $C^{(1)}_\ell$, in which a graded dimension formula for $\fqH0(n)$ is given by using the $C$-type Fock space $\mathcal{F}$ 
\cite{KMM93, KimShin04,Pre04}. 
This Fock space $\mathcal{F}$ is constructed by folding the usual $A$-type Fock space, so the dimension formula is described in terms of combinatorics of Young diagrams. 
In affine type $A_\ell^{(1)}$, graded Specht modules are deeply related to the $A$-type Fock space. It was shown in \cite{BK09} that the graded decomposition numbers of graded Specht modules 
can be described in terms of combinatorics of Young diagrams via the Fock space which is the $q$-version of the first author's result \cite{A96}.  
One can expect that cyclotomic quiver Hecke algebras of type $C$ and the Fock space $\mathcal{F}$ of type $C$ exhibit similar properties to those of type $A$ --
thus it is worth considering Specht modules for cyclotomic quiver Hecke algebras of type $C$.

In this paper, we construct \emph{Specht modules} $\Sp^\la$ for cyclotomic quiver Hecke algebras of affine type $C_\ell^{(1)}$ and type $C_\infty$ which are labelled by multipartitions $\la$.
This is inspired by \cite{kmr}.
Let $\cmA$ be the Cartan matrix of type $C_\ell^{(1)}$ or $C_\infty$, and $U_q(\cmA)$ the quantum group associated with $\cmA$. 
We set 
$R(\beta)$ to be the quiver Hecke algebra 
associated with $\cmA$ and denote by $E_i^\Lambda$ and $F_i^\Lambda$ the functors categorifying Chevalley generators $e_i$ and $f_i$ of $U_q(\cmA)$ on the highest weight irreducible module $V_q(\Lambda)$.
Let  $\Par_n$ be the set of $l$-multipartitions of $n$ with a multicharge $\kappa = (\kappa_1, \dots, \kappa_l)\in \Z^l$.
For $\lambda \in \Par_n$, we first construct a \emph{permutation module} $\Pe^\lambda_{\kappa}$ which has a basis indexed by row-strict tableaux of $\la$.
These permutation modules are built from more fundamental building blocks, namely they are convolution products of the one-dimensional $R(\beta)$-modules $\Lm(k;\ell) $ defined by $\eqref{Eq: def of L}$. 
The modules $\Lm(k;\ell) $ take a role as the \emph{segment modules}, which are given in \cite{kmr}, corresponding to \emph{segments} in type $A_\infty$ and $A_n^{(1)}$.
We also define a module $\Pe^\lambda_{\kappa, A}$ for each Garnir node $A \in [\la]$ and construct homomorphisms between $\Pe^\lambda_{\kappa, A}$ and  $\Pe^\lambda_\kappa$, 
which give an interpretation of Garnir elements in terms of quiver Hecke algebras.
We then define a Specht module $\Sp^\la$ using the cokernel of homomorphisms between $\Pe^\la_{\kappa, A}$ and $\Pe^\la_\kappa$, see \cref{spechtdef}.
The Specht module $\Sp^\la$ is spanned by homogeneous elements indexed by standard tableaux of shape $\la$ (\cref{Cor: Specht modules l}).
We prove in \cref{Cor: Specht modules} that, in type $C_\infty$, this spanning set of $\Sp^\la$ is in fact a basis.
Thus we have a graded character formula for $\Sp^\la$ in terms of standard tableaux and a description of $[E_i^{\Lambda} \Sp^\la]$ in terms of $[\Sp^{\la \nearrow b}]$.
Here, $\la \nearrow b$ is the Young diagram obtained from $\la$ by deleting a \emph{removable} node $b$.
We remark that $\Sp^\la$ is not necessarily \emph{simple}, even in the case of level 1 and type $C_\infty$.
We also investigate a connection between Specht modules $\Sp^\la$ and the Fock space $\mathcal{F}$ of type $C$, which provides a description of $[F_i^{\Lambda} \Sp^\la]$ in terms of $[\Sp^{\la \swarrow b}]$, where  $\la \swarrow b$ is the Young diagram obtained from $\la$ by adding an \emph{addable} node $b$, see \cref{Cor: F_i Sp}.
Recently, the third author provided semisimplicity criteria for the cyclotomic quiver Hecke algebras of type $C_\infty$ and $C^{(1)}_n$ using the Specht modules \cite{Sp17}.

The paper is organised as follows.
In Section~\ref{Sec: tableaux}, we review the combinatorics of tableaux and the Fock space of type $C$, and prove lemmas on Garnir nodes. In Section~\ref{Sec: quiver Hecke algebras}, we recall the notion of quiver Hecke algebras, and prove several lemmas on computations of products of $\psi_i$ and convolution products of modules for proving our main theorem.
In Section~\ref{Sec: Specht Sp}, we construct and investigate Specht modules $\Sp^\la$ and provide the main theorems with examples. 
Section~\ref{Sec: proof of the main thm} is devoted to proving \cref{Thm: Specht modules}.
We may carry out the computation in a manner knot theorists do, but we have found an algebraic proof, which is easier to access for representation theorists.
In Section~\ref{Sec: conjecture}, we propose a conjecture for type $C^{(1)}_\ell$.

\vskip 1em

\section{Combinatorics of Tableaux} \label{Sec: tableaux}

\subsection{Lie theory notation}\

Let $\ell\in \{2,3,\dots\}\cup\{\infty\}$ and $I=\Z_{\geq 0}$ if $\ell=\infty$ or $I= \{ 0,1,2, \ldots, \ell \} $ otherwise.

 For $\ell=\infty$, the corresponding Cartan matrix $ \cmA = (a_{ij})_{i,j\in I}$ of type $C_\infty$ is given by
\[
a_{ij}=
\begin{cases}
2 & \text{if } i=j,\\
-2 & \text{if } (i,j)=(1,0),\\
-1 & \text{if } i=j\pm 1 \text{ and } (i,j)\neq (1,0),\\
0 & \text{otherwise}.
\end{cases}
\]

Otherwise, the affine Cartan matrix of type $C^{(1)}_\ell$ is given by
\[
\cmA = (a_{ij})_{i,j\in I} = \begin{pmatrix}
2 & -1 & 0 & \cdots & 0 & 0 & 0\\
-2 & 2 & -1 & \cdots & 0 & 0 & 0\\
0 & -1 & 2 & \cdots & 0 & 0 & 0\\
\vdots & \vdots & \vdots & \ddots & \vdots & \vdots & \vdots\\
0 & 0 & 0 & \cdots & 2 & -1 & 0\\
0 & 0 & 0 & \cdots & -1 & 2 & -2\\
0 & 0 & 0 & \cdots &  0 & -1 & 2
\end{pmatrix}
\]

We adopt standard notation from \cite{kac} for the root datum; in particular 
we have simple roots $\{\alpha_i \mid i\in I \}$ and fundamental weights $\{\Lambda_i \mid i\in I \}$ in the \emph{weight lattice} $\wlP$, 
and simple coroots $\{\alpha_i^\vee \mid i\in I\}$ in the \emph{dual weight lattice} $\wlP^\vee$. 
There is an invariant symmetric bilinear form $( - , - )$ on $\wlP$ satisfying $( \Lambda_i, \alpha_j ) = d_j \delta_{ij}$ and $( \alpha_i, \alpha_j ) = d_i a_{ij}$ 
where $d=(2,1,1,\dots)$ if $\ell=\infty$ or $d = (2,1\dots,1,2)$ if $\ell<\infty$. Let
\[
\wlP^+ = \{\Lambda \in \wlP \mid \langle\alpha_i^\vee,\Lambda\rangle \geq 0 \text{ for all } i\in I\}
\]
be the set of \emph{dominant integral weights}, where $\langle$ , $\rangle$ is the natural pairing.
We denote by $\rlQ:=\bigoplus_{i\in I} \Z \alpha_i$ the \emph{root lattice} and $\rlQ^+ = \bigoplus_{i\in I} \Z_{\geq 0} \alpha_i$ is the \emph{positive cone} of the root lattice. Note that the \emph{null root} in type $C^{(1)}_\ell$ is given by
$\delta = \alpha_0 + 2\alpha_1 +\dots + 2\alpha_{\ell-1} +\alpha_\ell.$

\subsection{The symmetric group and multipartitions} \label{Sec: muptipartitions}\

Denote by $\sym n$ the symmetric group on $n$ letters, with Coxeter generators $s_1,\dots, s_{n-1}$. For a permutation $w \in \sym n$, a \emph{reduced expression} for $w$ is an expression $w = s_{i_1} \dots s_{i_r}$ of minimal length; $r = \ell(w)$ is the \emph{length} of $w$.

We denote by $\sg_{m+n}/ \sg_{m}{\times}\sg_n$ the set of distinguished left coset representatives of $\sg_{m}{\times}\sg_n$ in $\sg_{m+n}$, i.e.~$\ell(w s_i) = \ell(w) + 1$ for $w \in \sg_{m+n}/ \sg_{m}{\times}\sg_n$ and $i \ne m$.

For $a,b \in \Z_{\ge0}$ with $a+b\le n$, we define $w[a,b] \in \sg_n$ by
\[
w[a,b](x) = \left\{
              \begin{array}{ll}
                x+b & \hbox{ if } 1\le x \le a, \\
                x-a & \hbox{ if } a < x \leq a+b,\\
                x & \hbox{ if } a+b < x \leq n.
              \end{array}
            \right.
\]
In two-line notation, $w[a,b]$ is
\[
\bigl(
\begin{smallmatrix}
1 & 2 & \cdots & a & a+1 & a+2 & \cdots & a+b\\
b+1 & b+2 & \cdots & b+a & 1 & 2 & \cdots  & b
\end{smallmatrix}
\bigr).
\]
Throughout the paper, $w\in\sg_n$ permutes letters of a tableau, but permutes places of $\nu = (\nu_1, \ldots, \nu_n) \in I^n$ as $w\nu=(\nu_{w^{-1}(1)},\dots,\nu_{w^{-1}(n)})$. 
In particular,
\[
w[a,b]\nu=(\nu_{a+1},\dots,\nu_{a+b},\nu_1,\dots,\nu_a,\nu_{a+b+1},\dots,\nu_n).
\]

\begin{lem} \label{use of braid rels}
Let $a,b\ge 1$. 
\begin{itemize}
\item[(1)]
$w[a,b]=(s_b\dots s_{a+b-1})\dots(s_1\dots s_a)=(s_b\dots s_1)\dots(s_{a+b-1}\dots s_a)$.
\item[(2)] $s_{b+1}w[2,b]=w[2,b]s_1$.
\end{itemize}
\end{lem}
\begin{proof}
It is easy to see (1). 
Using the braid relations, we have
\begin{align*}
s_{b+1}w[2,b]&=\underline{s_{b+1}}(\underline{s_{b}}s_{b-1}\dots s_1)(\underline{s_{b+1}}s_b\dots s_2)\\
&=(s_bs_{b+1})\underline{s_{b}}(\underline{s_{b-1}}s_{b-2}\dots s_1)(\underline{s_{b}}s_{b-1}\dots s_2)\\
&=(s_bs_{b+1})(s_{b-1}s_b)\underline{s_{b-1}}(\underline{s_{b-2}}\dots s_1)(\underline{s_{b-1}}\dots s_2)\\
& \qquad \ldots\ldots\ldots\ldots \\
&=(s_bs_{b+1})(s_{b-1}s_b)\dots(s_2s_3)\underline{s_2}(\underline{s_1})(\underline{s_2})\\
&=(s_bs_{b+1})(s_{b-1}s_b)\dots(s_2s_3)(s_1s_2)s_1=w[2,b]s_1,
\end{align*}
which complete the proof of (2). 
Here, the underlines indicate generators to which we can apply the braid relation.
\end{proof}

The following easy lemma will be useful to us later. Note that the equality $s_{b+1}w[2,b]=w[2,b]s_1$ in \cref{use of braid rels}(2) is a special case of this, but the importance of \cref{use of braid rels}(2) lies in the `long-hand derivation' of this equality, which we will utilise later in \cref{2exchange formula}.

\begin{lem}\label{simplemove}
Let $w \in \sym n$ and $1\leq i \leq n-1$. If $w(i+1) = w(i) + 1$, then $s_{w(i)} w = w s_i$.
\end{lem}

For a reduced expression $w = s_{i_1} \dots s_{i_r}$ and $k \in \Z_{\ge 0}$ with $i_j < n-k$ for $1\le j\le r$, we set
\begin{align} \label{shift}
\shf_k(w) = s_{i_1+k} \dots s_{i_r+k}.
\end{align}
Note that $\shf_k(w)$ does not depend on the choice of reduced expressions. 
For $a, b, c \in \Z_{\ge0}$, we define the block permutation $\mathrm{S}_2(c,a,b)$ to be $\shf_c(w[a,b])$.

\begin{defn}
For $v, w \in \sym n$, we write $v \succcurlyeq w$ if there is a reduced expression for $v$ which has an expression for $w$ as a subsequence. 
We write $v \succ w$ if $v \succcurlyeq w$ and $v\neq w$. This partial order is called the \emph{Bruhat order}.

The \emph{left order} (sometimes called the weak Bruhat order) is given by $v\geq_L w$ if there is a reduced expression for $v$ which has a reduced expression for $w$ as a suffix -- that is, $v=s_{i_1}\dots s_{i_r}w$ for some $i_1,\dots, i_r$ with $r=\ell(v)-\ell(w)$.
\end{defn}

We fix an integer $l\geq 1$ throughout, which we refer to as the \emph{level}.

\begin{defn}
For $n\geq 0$, a partition of $n$ is a weakly decreasing sequence of non-negative integers $\la = (\la_1, \la_2, \dots)$ such that the sum 
$|\la|=\la_1+\la_2+\cdots$ is equal to $n$. If $\la$ is a partition of $n$ we write $\la\vdash n$. We write $\varnothing$ for the unique partition of 0. Note that we will in general omit trailing zeros for partitions.

An \emph{$l$-multipartition} of $n$ is an $l$-tuple of partitions $\la = (\la^{(1)}, \dots, \la^{(l)})$ such that the total size $\sum_{i=1}^l |\la^{(i)}| = n$. We denote the set of $l$-multipartitions of $n$ by $\Par_n$ and set $\Par:= \cup_{n\geq 0} \Par_n$.

Similarly, a \emph{composition} is a sequence $\mu = (\mu_1, \mu_1, \dots)$ of non-negative integers,  and an \emph{$l$-multicomposition} is an $l$-tuple of compositions.
\end{defn}

If $\la$ and $\mu$ are $l$-multicompositions of $n$, we say that $\la$ \emph{dominates} $\mu$, and write $\la\dom\mu$ if
\[
|\la^{(1)}| + \dots + |\la^{(t-1)}| + \sum_{j=1}^{k} \la^{(t)}_j \geq |\mu^{(1)}| + \dots + |\mu^{(t-1)}| + \sum_{j=1}^{k} \mu^{(t)}_j
\]
for all $1\leq t \leq l$ and $k\geq 0$.

For any $\la\in \Par_n$, we define its \emph{Young diagram} $[\la]$ to be the set
\[
\{(r,c,t) \in \Z_{\geq 1} \times \Z_{\geq 1} \times \{1,\dots,l\} \mid c\leq \la^{(t)}_r\}.
\]

We will depict a Young diagram for a partition using the English convention, and for a multipartition $\la$ as a column vector of Young diagrams for the components $\la^{(1)},\dots,\la^{(l)}$.
If $l = 1$, then we write simply $(r,c)$ for $(r,c,t)$.

\begin{ex}
Let $\la = ((4,3,1,1),\varnothing, (3,2,1)) \in \mathscr{P}^3_{15}$. Then we write
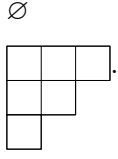
\begin{align*}
[\la] = \ &\yng(4,3,1,1)\\
&\varnothing\\
&\yng(3,2,1).
\end{align*}
\end{ex}

With this convention, we say that for nodes $A=(r,c,t),A'=(r',c',t')\in[\la]$, $A$ is \emph{below}  $A'$ if $t>t'$ or if $t=t'$ with $r>r'$, and $A$ is \emph{above} $A'$ if $A'$ is below $A$.

We define $f_\ell : \Z \rightarrow I$ by $k \mapsto |k|$ if $\ell = \infty$ and, if $\ell \ne \infty$, $f_\ell : \Z / 2\ell\Z \rightarrow I $ by
\begin{align*}
&f_\ell(0 + 2\ell\Z) = 0, \quad f_\ell(\ell+2\ell\Z) = \ell, \\
 &f_\ell(k + 2\ell\Z) = f_\ell(2\ell-k + 2\ell\Z) = k \quad  \text{ for } 1 \le k \le \ell-1.
\end{align*}
Let $p$ be the natural projection $\Z \rightarrow \Z/ 2\ell\Z$ if $\ell \ne \infty$ and $p = \textrm{id} $ if $\ell=\infty$. Then we define
$\pi_\ell=f_\ell \circ p: \Z \rightarrow I$. We denote $\pi_\ell(k)$ by $\overline{k}$, for $k\in \Z$, if there is no confusion.

Now we fix a \emph{multicharge} $\kappa = (\kappa_1, \ldots, \kappa_l) \in \Z^l$ and define $\Lambda\in\wlP^+$ by $\Lambda=\sum_{i=1}^l \Lambda_{\overline{\kappa_i}}.$
Let $\la$ be an $l$-multipartition. Then, to any node $A = (r,c,t) \in [\la]$ we may associate its \emph{residue} by
\[
\res A = \overline{\kappa_t +c - r}.
\]
If $\res A = i$, we call $A$ an $i$-node.
Thus, $l$-multipartitions may be coloured by $I$. We define the \emph{content} of $\la$ to be
\[
\cont(\la) = \sum_{A\in[\la]} \alpha_{\res A}\in \rlQ^+.
\]

\begin{ex}
For $\la = ((4,3,1,1),\varnothing, (3,2,1))$ as above, and $\kappa = (2,0,-1)$, the residues of $[\la]$ are given as follows.
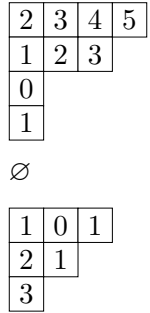
\begin{align*}
&\young(2345,123,0,1)\\
&\varnothing\\
&\young(101,21,3)
\end{align*}
We also have $\cont(\la) = 2\alpha_0 + 5\alpha_1 + 3\alpha_2 + 3\alpha_3 + \alpha_4 + \alpha_5$.
\end{ex}

We say that a node $A$ is \emph{removable} (resp.\ \emph{addable}) if $[\la]\setminus A$ (resp.\ $[\la]\cup A$) is a valid Young diagram for a multipartition of $n-1$ (resp.\ $n+1$).
We write $\la\nearrow A$ (resp.\ $\la\swarrow A$) as shorthand for the multipartition whose Young diagram is $[\la]\setminus A$ (resp.\ $[\la]\cup A$).
For an $i$-node $A\in[\la]$, we set
\begin{align*}
d_i(\la) &= \#\{\text{addable $i$-nodes of $[\la]$} \} - \#\{\text{removable $i$-nodes of $[\la]$}\}, \\
d_A(\la) &= d_i (\#\{\text{addable $i$-nodes of $[\la]$ below } A\} - \#\{\text{removable $i$-nodes of $[\la]$ below } A\}), \\
d^A(\la) &= d_i (\#\{\text{addable $i$-nodes of $[\la]$ above } A\} - \#\{\text{removable $i$-nodes of $[\la]$ above } A\}).
\end{align*}

We define $\mathcal{F}(\kappa)$ to be a $\Q(q)$-vector space  with basis consisting of the coloured $l$-multipartitions. 
Then $\mathcal{F}(\kappa)$ has a $U_q(\mathfrak{g}(\cmA))$-module structure defined by
\begin{align} \label{Eq: def of Fock sp}
q^{h_i} \la = q^{d_i(\la)}\la ,\qquad   e_i \la = \sum_{A} q^{d_A(\la)} \la\nearrow A, \qquad  f_i \la = \sum_{A} q^{-d^A(\la)} \la\swarrow A,
\end{align}
where $A$ runs over all removable $i$-nodes and all addable $i$-nodes respectively.
The above description of $\mathcal{F}(\kappa)$ matches with that of the type A Fock space given in \cite[Section 3.6]{BK09},
which is slightly different from  \cite{ KimShin04, Pre04}.
We call $\mathcal{F}(\kappa)$ the \emph{level $l$ Fock space with multicharge $\kappa$}. 
Note that the weight of a coloured $l$-multipartition $\la$ is $\Lambda-\cont(\la)$, and there is a $U_q(\g(\cmA))$-module isomorphism  
\begin{align*}
\mathcal{F}(\kappa) \simeq \mathcal{F}(\kappa_1)\otimes\dots\otimes\mathcal{F}(\kappa_l).
\end{align*} 
Here, the $U_q(\g(\cmA))$-module structure of the tensor product comes from the comultiplication of $U_q(\g(\cmA))$ given by, for $i\in I$, 
\begin{align*}
\Delta: K_i \longmapsto K_i \otimes K_i, \quad e_i \longmapsto e_i \otimes K_i + 1 \otimes e_i, \quad f_i \longmapsto f_i \otimes 1 + K_i^{-1} \otimes f_i,  
\end{align*}
where $K_i = q^{\frac{(\alpha_i, \alpha_i)}{2}h_i}$. (\protect{cf.~\cite[Section 3.1]{BK09}})

Let $\Lambda = \sum_{t=1}^l \Lambda_{ \overline{\kappa_t}} $ and $V_q(\Lambda)$ the irreducible highest weight $U_q(\g(\cmA))$-module with highest weight $\Lambda$. 
As $\varnothing$ is a highest weight vector of $\mathcal{F}(\kappa)$ with highest weight $\Lambda $, 
we have a canonical $U_q(\g(\cmA))$-module epimorphism
\begin{align} \label{Eq: pk}
p_\kappa : \mathcal{F}(\kappa) \longtwoheadrightarrow V_q(\Lambda), \qquad \varnothing \longmapsto v_\Lambda,
\end{align}
where $v_\Lambda$ is a highest weight vector of $V_q(\Lambda)$.

\subsection{Tableaux} \

We will mostly adopt the notation of \cite{BKW11,kmr} for tableaux.

Let $\la \in \Par_n$. A $\la$-tableau is a bijection $\ttt :[\la] \rightarrow \{1,\dots,n\}$. We depict $\ttt$ by filling each node $(r,c,t)\in [\la]$ with $\ttt(r,c,t)$. We say that a tableau $\ttt$ is \emph{row-strict} if the entries increase along the rows of each component of $\ttt$, and \emph{column-strict} if the entries increase down the columns of each component of $\ttt$. If $\ttt$ is both row- and column-strict, we call it \emph{standard}. We denote the set of standard tableaux by $\ST\la$, the set of row-strict tableaux by $\rST\la$ and the set of row-strict tableaux which are not standard by $\rT\la$ = $\rST\la \setminus \ST\la$. Note that the symmetric group $\sym n$ acts naturally on the left on the set of tableaux.

For each $\la$-tableau $\ttt$, we have the associated residue sequence
\[
\res\ttt = (  \res { \ttt^{-1}(1) } ,\   \res { \ttt^{-1}(2)} , \dots,  \res { \ttt^{-1}(n)} ).
\]

Let $\ttt^\la$ be the \emph{initial tableau}, which is the distinguished tableau 
where we fill the nodes with $1,\dots, n$ first along successive rows in $\la^{(1)}$, then $\la^{(2)}$, and so on.
Then for each $\la$-tableau, $\ttt$, we may define the permutation $w^\ttt \in \sym n$ by $w^\ttt \ttt^\la = \ttt$ and the length $\ell(\ttt)\in\Z_{\ge0}$ by $\ell(\ttt)=\ell(w^\ttt)$.
We simply write $T \ge_L T'$ when $w^T \ge_L w^{T'}$ for tableaux $T$ and $T'$.

\begin{ex}
Continuing our previous example,

\begin{align*}
\ttt^\la = \ &\young(1234,567,8,9)\\
&\varnothing\\
&\mkern-4mu \young(<10><11><12>,<13><14>,<15>).
\end{align*}
\end{ex}

We define the \emph{dominance order} on $\la$-tableaux by setting $\ttt \dom \tts$ if and only if $w^\ttt \preccurlyeq w^\tts$. The matchup of terminology and notation with the dominance order on partitions is justified by \cref{shpdom}. Note in particular that $\ttt^\la \dom \ttt$ for all $\la$-tableaux $\ttt$.

First, we introduce one more concept. Let $\ttt$ be a $\la$-tableau and $0\leq m\leq n$. We denote by $\tabupto\ttt m$ the set of nodes of $[\la]$ whose entries are less than or equal to $m$. If $\ttt\in \ST\la$, then $\tabupto\ttt m$ is a tableau for some multipartition, which we call $\shp\ttt m$. If $\ttt\in \rT\la$, then $\tabupto\ttt m$ is a tableau 
for some multicomposition, which we also call $\shp\ttt m$.

\begin{lem}{\cite[Theorem 3.8]{Mathas}}\label{shpdom}
Suppose $\la\in \Par_n$ and $\ttt,\tts\in \rST\la$. Then $\ttt \dom \tts$ if and only if $\shp\ttt m \dom \shp\tts m$ for all $1\leq m\leq n$.
\end{lem}

For any $\la\in\Par_n$ and $\ttt\in\ST\la$ we define the \emph{degree} $\deg\ttt$ of $\ttt$ as follows. If $n=0$ then $\ttt$ is the unique $\varnothing$-tableau and we set $\deg\ttt:=0$. Otherwise, let $A=\ttt^{-1}(n)\in[\la]$ and suppose $A$ is an $i$-node. We set inductively
\begin{align} \label{Eq: def of deg}
\deg\ttt := \deg\tabupto\ttt{n-1} + d_A(\la).
\end{align}

\begin{ex}
Let $\ell=\infty$, $\kappa = (2,-1)$ and $\la = ((2,2,1),(3,2))$. Then the residue pattern of $\la$ is
\begin{align*}
&\young(23,12,0)\\
&\young(101,21)
\end{align*}
and if $\ttt$ is the tableau
\begin{align*}
&\young(14,25,6)\\
&\young(37<10>,89)
\end{align*}
we have
\[
\deg\ttt = 0 +1+0+0+1 +2 +0+0+0 -1 = 3.
\]
The nodes contributing to the degree are those containing the entries 2 (a 1-node), 5 (a 2-node), 6 (a 0-node) and 10 (a 1-node).

\end{ex}

\subsection{Garnir tableaux}\

\begin{defn}
Let $\la\in \Par_n$ and $A=(r,c,t)\in[\la]$. We call $A$ a \emph{Garnir node} if $(r+1,c,t)\in[\la]$. For a Garnir node $A\in[\la]$, the \emph{Garnir belt} $\belt^A$ is the set of nodes
\[
\{(r,a,t)\in[\la] \mid c \leq a\leq \la^{(t)}_r\} \cup \{(r+1,a,t)\in[\la] \mid 1 \leq a\leq c\}.
\]
\end{defn}

Finally, for a Garnir node $A\in[\la]$, the \emph{Garnir tableau} $\ttg^A$ is the $\la$-tableau which agrees with the initial tableau $\ttt^\la$ outside of $\belt^A$ and has the entries $u,u+1,\dots,v$ from the bottom left to the top right of $\belt^A$, where $u=\ttt^\la(r,c,t)$ and $v=\ttt^\la(r+1,c,t)$.
Then
\begin{align} \label{exp of GA}
w^{\ttg^A} = S_2(a, \la^{(t)}_r-c+1,c),
\end{align}
where $a = \sum_{i=1}^{t-1} |\lambda^{(i)}| + \sum_{j=1}^{r-1} \lambda^{(t)}_j + c-1$.
Note that $S_2(a, \la^{(t)}_r-c+1,c)$ is $321$-avoiding so that $w^{\ttg^A}$ is fully commutative. See \cite[Lemma 2.1]{bjs93} for example.

\begin{ex}
Let $\la = ((4,3,1,1),\varnothing, (3,2,1))$ and $A=(1,3,1)$. Then the Garnir tableau $\ttg^A$, with the Garnir belt $\belt^A$ shaded, is as follows.

\begin{align*}
\ttg^A = \ &\young(12!\gr67,345,!\wh8,9)\\
&\varnothing\\
&\mkern-4mu \young(<10><11><12>,<13><14>,<15>)
\end{align*}
\end{ex}

The following lemma is an easy generalisation of \cite[Lemma 3.14]{Mathas} but we include a proof for the reader's convenience. 
This lemma and \cref{lem: garnir} will be used in the proof of \cref{Thm: Specht modules} in \cref{Sec: proof of the main thm}.

\begin{lem}\label{straighten}
Let $\lambda \in \Par_n$ and  $\ttt \in \rT\la$. If $A=(r,c,t)\in [\la]$ with $\ttt(r,c,t) > \ttt(r+1,c,t)$, 
then there is an element $w \in \sg_n$ such that $\ttt = w \ttg^A$ and $\ell(\ttt) = \ell(w) + \ell(\ttg^A)$. That is, $w^\ttt \geq_L w^{\ttg^A}$. 
Conversely, if $\ttt = w \ttg^A$ with $\ell(\ttt) = \ell(w) + \ell(\ttg^A)$ then $\ttt \in \rT\la$.
\end{lem}

\begin{proof}
Let $u:=\ttt^\la(r,c,t)$, $v:=\ttt^\la(r+1,c,t)$, $a:= \ttt(r,c,t)$ and $b:=\ttt(r+1,c,t)$. 
If $\ttt=\ttg^A$, the result is clear. So we suppose that $\ttt\ne \ttg^A$, and we will choose a basic transposition $s_i$ such that $s_i\ttt \in \rT\la$ and $s_i\ttt \doms \ttt$, from which the result follows by (reverse) induction on the dominance order $\doms$.

If $\ttt$ coincides with $\ttt^\la$ outside of $\belt^A$, there is a gap in the reading word of $\ttt$ in either the first or the second row of $\belt^A$ -- otherwise $u, u+1,\dots, v$ are split into two sets of consecutive numbers and as $\ttt(r,c,t)>\ttt(r+1,c,t)$ the only way to fill in the numbers is $\ttt=\ttg^A$. Thus, we may choose $i+1$ in the first row and $i$ in the second row for some $i$ with $(i,i+1)\ne(b,a)$, so that $s_i\ttt(r,c,t) > s_i\ttt(r+1,c,t)$.

Otherwise, we may choose $s_i$ so that $s_i\ttt(r,c,t) > s_i\ttt(r+1,c,t)$ as follows. First suppose that the reading word of $\ttt$ begins $1,2,\dots,m,m'$ for some $m'>m+1$ and $m<u$. Then setting $i=m'-1$ suffices.
Next suppose that the reading word of $\ttt$ ends $m',m,m+1,\dots,n$ for some $v<m'<m-1$. Then, setting $i=m'$ suffices.

For the converse statement, we argue by induction on $\ell(\ttt)$.
Suppose $\ttt = w \ttg^A$ with $\ell(\ttt) = \ell(w) + \ell(\ttg^A)$ and $s_i \ttt \domsby \ttt$. Then if $s_i\ttt$ is standard, so is $\ttt$, by \cite[Lemma 1.5]{DJ86}.
But this contradicts the induction hypothesis.
\end{proof}

\begin{lem} \label{lem: garnir}
Let $\lambda \in \Par_n$ and $\ttt\in \rT\la$.
\begin{enumerate}
\item If $\ttt(r,c,t)= \ttt(r+1,c,t)+1$, then
there is an element $w \in \sg_n$ such that, for $A:=(r,c,t) \in [\la]$,
\begin{enumerate}
\item[(i)]
$\ttt = w \ttg^A$,

\item[(ii)] $s_p w = w s_q$, where $p = \ttt(r+1,c,t)$ and $q = \ttg^A (r+1,c,t)$.
\end{enumerate}

\item If $\ttt(r,c+1,t)= \ttt(r,c,t)+1$, then
there is a Garnir node $A \in [\la]$ and $w \in \sg_n$  such that
\begin{enumerate}
\item[(i)] $ \ttg^A(r,c+1,t) = \ttg^A(r,c,t)+1$,
\item[(ii)] $\ttt = w \ttg^{A}$,
\item[(iii)] $s_p w = w s_q$, where $p = \ttt(r,c,t)$ and $q = \ttg^{A} (r,c,t)$.
\end{enumerate}

\end{enumerate}
\end{lem}

\begin{proof}
\begin{enumerate}
\item Part (i) follows from \cref{straighten}. For part (ii), note that $\ttg^A(r,c,t) = q+1$ by definition, so that
\begin{align*}
1 &= \ttt(r,c,t) - \ttt(r+1,c,t) = w(\ttg^A(r,c,t)) - w(\ttg^A(r+1,c,t)) \\
  &= w(q+1) - w(q).
\end{align*}
It follows from \cref{simplemove} that $ws_q=s_{w(q)}w=s_pw$.
\item We begin by choosing a node $A=(r',c',t') \in [\la]$ such that $\ttt(r',c',t')>\ttt(r'+1,c',t')$ and $\ttg^A(r,c+1,t) = \ttg^A(r,c,t) + 1$ as follows.

If $\ttt(r,c+1,t)>\ttt(r+1,c+1,t)$ then we know that $\ttt(r,c,t)>\ttt(r+1,c,t)$, so we may choose $A=(r,c,t)$. 

Next, suppose $\ttt(r-1,c,t)>\ttt(r,c,t)$. Then we have $\ttt(r-1,c+1,t)>\ttt(r,c+1,t)$, and we may choose $A= (r-1,c+1,t)$.

Otherwise, if $\ttt(r,c+1,t)<\ttt(r+1,c+1,t)$ and $\ttt(r-1,c,t)<\ttt(r,c,t)$, then, as $\ttt \in \rT\la$, there is some node 
$A=(r',c',t')$ such that $\ttt(r',c',t') > \ttt(r'+1,c',t')$ and $(r',c',t') \neq (r-1,c,t),(r,c+1,t)$. Since 
$\ttg^A(x,y+1,z) = \ttg^A(x,y,z)+1$ holds unless $(x,y,z)=(r',c'-1,t'), (r'+1,c',t')$, $(r,c,t)\ne(r',c'-1,t'), (r'+1,c',t')$
implies $ \ttg^A(r,c+1,t) = \ttg^A(r,c,t)+1$. Hence (i) is proved.

Now, by \cref{straighten} and the fact that $\ttt(r',c',t') > \ttt(r'+1,c',t')$, there is some $w\in \sg_n$ such that 
$\ttt = w\ttg^A$ and $\ell(\ttt) = \ell(w) + \ell(\ttg^A)$. We have proved (ii). Moreover,
\begin{align*}
1 &= \ttt(r,c+1,t) - \ttt(r,c,t) = w(\ttg^A(r,c+1,t)) - w(\ttg^A(r,c,t)) \\
  &= w(q+1) - w(q).
\end{align*}
Thus, (iii) follows from \cref{simplemove}.\qedhere
\end{enumerate}
\end{proof}

\begin{lem}\label{garjoin}
Let $A$ and $B$ be distinct Garnir nodes of $\la \in \Par_n$.
Then there is a unique tableau $\ttg^{A, B} \in \rT\la$ such that
\begin{enumerate}
\item $\ttg^{A,B} \geq_L \ttg^A $ and $\ttg^{A,B} \geq_L \ttg^B$,
\item $\ttt \geq_L \ttg^{A,B}$ for any $\ttt \in \rST\la$ with $\ttt\geq_L \ttg^A$ and $\ttt\geq_L \ttg^B$.
\end{enumerate}
\end{lem}

\begin{proof}
It is known that $\rST\la$ is a lattice with respect to the left order. See for example \cite[Theorem 7.1]{bw88} (with some slight modification to generalise to $\rST\la$). Thus $\ttg^{A,B} = \ttg^A \vee \ttg^B$.
\end{proof}

We redefine $\ttg^{A, B}$ in \cref{gengar} below in a more concrete manner and show in \cref{construction of G^{A,B}} that it coincides with $\ttg^{A, B}$ in \cref{garjoin}.


\begin{defn}\label{gengar}
Suppose $A,B \in [\la]$ are distinct Garnir nodes. We define the sets $\belt^A(2)$ and $\belt^B(1)$ to be the second row of $\belt^A$ and the first row of $\belt^B$, respectively.

We define the \emph{generalised Garnir belt} $\belt^{A,B}$ of $[\la]$ to be the following set of nodes.

\begin{enumerate}
\item If $\belt^A \cap \belt^B =\emptyset$, then $\belt^{A,B}:= \belt^A \cup \belt^B$.
\item If $A=(r,c,t)$ and $B=(r,c',t)$ for some $c'> c$, then
\begin{align*}
\belt^{A,B}:&= \{(r,a,t) \mid a\geq c\} \cup \{(r+1,a,t) \mid a\leq c'\}\\
&= \belt^A\cup\belt^B\\
&= (\belt^A\setminus\belt^B(1)) \cup (\belt^B\setminus\belt^A(2)).
\end{align*}
In this case, we set $\belt^{A,B}(1) = \belt^A\setminus\belt^B(1)$ and $\belt^{A,B}(2) = \belt^B\setminus\belt^A(2)$.
\item If $A=(r,c,t)$ and $B=(r-1,c',t)$ for some $c' \geq c$, then
\[
\belt^{A,B}:= \{(r-1,a,t) \mid a\geq c'\} \cup \{(r,a,t) \mid c\leq a\leq c'\} \cup \{(r+1,a,t) \mid a\leq c\}.
\]
\end{enumerate}

Finally, we define the \emph{generalised Garnir tableau} in the first two cases above to be the $\la$-tableau $\ttg^{A,B}$ which agrees with $\ttt^\la$ outside of $\belt^{A,B}$ and has the entries of $\belt^{A,B}$ as follows:
\begin{enumerate}
\item If $\belt^A \cap \belt^B =\emptyset$, then we fill each of $\belt^A$ and $\belt^B$ as in $\ttg^A$ and $\ttg^B$, respectively.
\item If $A=(r,c,t)$ and $B=(r,c',t)$ for some $c'\geq c$, then we first fill the entries of $\belt^{A,B}(1)$, from bottom left to top right, and then we fill the entries of $\belt^{A,B}(2)$, from bottom left to top right.
\end{enumerate}
In the third case above, $\ttg^{A,B}$ is defined as follows.
\begin{enumerate}\setcounter{enumi}{2}
\item If $A=(r,c,t)$ and $B=(r-1,c',t)$ for some $c' \geq c$, $\ttg^{A,B}$ is defined to be the $\la$-tableau which agrees with $\ttt^\la$ outside of the three rows of $[\la]$ which contain elements of $\belt^{A,B}$, and we fill the entries of these three rows first in order along rows above $\belt^{A,B}$, then from bottom left to top right in $\belt^{A,B}$, and finally in order along rows below $\belt^{A,B}$.
\end{enumerate}
\end{defn}

\begin{ex}
Let $\la = ((1),(10,9,6,2))$ and $A=(2,3,2)\in[\la]$. Then we have the following tableaux $\ttg^{A,B}$ in cases corresponding to \cref{gengar}, where we have shaded the generalised Garnir belts $\belt^{A,B}$ in each case.
\begin{enumerate}
\item Let $B=(1,1,2)$. Then

\begin{align*}
\ttg^{A,B} = \ &\young(1)\\
&\mkern-4mu \young(!\gr3456789<10><11><12>,2!\wh<13>!\gre<17><18><19><20><21><22><23>,<14><15><16>!\wh<24><25><26>,<27><28>)
\end{align*}

\item Let $B=(2,6,2)$. Then

\begin{align*}
\ttg^{A,B} = \ &\young(1)\\
&\mkern-4mu \young(23456789<10><11>,<12><13>!\gr<17><18><19>!\gre<23><24><25><26>,!\gr<14><15><16>!\gre<20><21><22>,!\wh<27><28>)
\end{align*}

\item Let $B=(1,6,2)$. Then
\begin{align*}
\ttg^{A,B} = \ &\young(1)\\
&\mkern-4mu \young(23456!\gr<16><17><18><19><20>,!\wh78!\gr<12><13><14><15>!\wh<21><22><23>,!\gr9<10><11>!\wh<24><25><26>,!\wh<27><28>)
\end{align*}
\end{enumerate}

\end{ex}

\begin{lem} \label{construction of G^{A,B}}
The construction of $\ttg^{A,B}$ in \cref{gengar} satisfies $\ttg^{A,B} = \ttg^A \vee \ttg^B$ and thus coincides with the tableau $\ttg^{A,B}$ defined in \cref{garjoin}.
\end{lem}

\begin{proof}
It is easy to see that $\ttg^{A,B} \geq_L \ttg^A$ and $\ttg^{A,B} \geq_L \ttg^B$, so that we have $\ttg^{A,B} \geq_L \ttg^A\vee \ttg^B$. If the inequality were strict, 
then there exists a basic transposition $s$ such that
\[
\ttg^{A,B} >_L s \ttg^{A,B} \geq_L \ttg^A\vee \ttg^B.
\]
However, the explicit construction of $\ttg^{A,B}$ shows that either $s \ttg^{A,B} \not\geq_L \ttg^A$ or $s \ttg^{A,B} \not\geq_L \ttg^B$ occurs 
for any basic transposition $s$ with $\ttg^{A,B} >_L s \ttg^{A,B}$. Hence, we must have equality.
\end{proof}

\vskip 1em

\begin{lem} \label{lem: wA wB}
Let $A = (r,c,t)$ and $B = (r',c',t')$ be Garnir nodes of $\la \in \Par_n$.
\begin{enumerate}
\item If $\belt^A \cap \belt^B = \emptyset$ then $w^{\ttg^{A, B}}$ is fully commutative.
\item If $r = r'$ and $t = t'$ then $w^{\ttg^{A, B}}$ is fully commutative.
\item If $\ttg^{A,B} = w^A \ttg^A = w^B \ttg^B$, then $w^A$ and $w^B$ are fully commutative.
\end{enumerate}
\end{lem}

\begin{proof}
Take $w^A, w^B \in \weyl$ such that  $\ttg^{A,B} = w^A \ttg^A = w^B \ttg^B$.
We consider the three cases (1), (2) and (3) in \cref{gengar}. 
\begin{itemize}
\item
In the first case, $\ttg^{A,B}=w^{\ttg^A}\ttg^B=w^{\ttg^B}\ttg^A$ and it is clear that each of
$w^{\ttg^A}$ and $w^{\ttg^B}$ are of the form $S_2(c,a,b)=\shf_cw[a,b]$ for some $a,b,c$. Further,  $w^{\ttg^{A, B}}$ has a unique descent pattern 
of $2143$. Thus, $w^A=w^{\ttg^B}$, $w^B=w^{\ttg^A}$ and $w^{\ttg^{A, B}}$ are $321$-avoiding.
This implies that (1) holds, and (3) holds when $\belt^A \cap \belt^B = \emptyset$.
\item
In the second case, $w^A$ is a shift of $w[\la^{(t)}_r-c'+1,c'-c]$ and $w^B$ is a shift of $w[c'-c,c]$. Thus $w^A$ and $w^B$ are $321$-avoiding. 
Further, the two-line notation for $w^{\ttg^{A, B}}$ is
\[
\bigl(
\begin{smallmatrix}
1 & 2 & \cdots & c'-c & c'-c+1 & \cdots & \la^{(t)}_r & \la^{(t)}_r+1 & \cdots & \la^{(t)}_r+c & \la^{(t)}_r+c+1 & \cdots & \la^{(t)}_r+c'\\
c+1 & c+2 & \cdots & c' & 2c'-c+1 & \cdots & \la^{(t)}_r+c'-c+1 & 1 & \cdots  & c & c'+1 & \cdots & 2c'-c
\end{smallmatrix}
\bigr)
\]
up to shift. Hence, $w^{\ttg^{A, B}}$ is $321$-avoiding, which yields that (2) holds, and (3) holds when $r=r'$ and $t=t'$.
\item
In the third case, $w^A$ and $w^B$ are $w[\la^{(t)}_{r-1}-c'+1,c+c']$ and $w[\la^{(t)}_{r-1}+\la^{(t)}_r-c-c'+2,c]$ up to shift, respectively. 
Thus, they are also $321$-avoiding, which completes the proof of $(3)$.
\end{itemize}
\end{proof}

\vskip 1em

\section{Quiver Hecke algebras} \label{Sec: quiver Hecke algebras}

\subsection{Affine and cyclotomic quiver Hecke algebras}\

In this subsection, $\cmA$ is an arbitrary symmetrisable Cartan matrix. 

Let $\bR$ be a unital commutative ring and we fix a system of polynomials $Q_{i,j}(u,v) \in \bR[u,v]$ for $i,j \in I$ of the form
\[
Q_{i,j}(u,v) =
\begin{cases}
\sum_{p(\alpha_i, \alpha_i) + q(\alpha_j, \alpha_j) + 2(\alpha_i, \alpha_j)=0} t_{i,j;p,q} u^p v^q & \text{if } i\neq j,\\
0 & \text{if } i=j,
\end{cases}
\]
where $t_{i,j;p,q} \in \bR$ are such that $t_{i,j;-a_{ij},0} \in \bR^\times$ and $Q_{i,j}(u,v) = Q_{j,i}(v,u)$.

For $\nu \in I^n$ and $\nu' \in I^{n'}$, we denote the concatenation of $\nu$ and $\nu'$ by $\nu * \nu'\in I^{n+n'}$. Here, we understand that
$I^0 :=\{ \emptyset \}$ and $\emptyset * \nu = \nu * \emptyset = \nu$.

\begin{defn} \label{defn - cyclotomic}
The \emph{cyclotomic quiver Hecke algebra} $R^\Lambda(n)$ associated with polynomials $(Q_{i,j}(u,v))_{i,j\in I}$ and $\Lambda\in\wlP^+$ is the $\Z$-graded unital $\bR$-algebra generated by
\[
\{e(\nu) \mid \nu =(\nu_1, \ldots, \nu_n) \in I^n\} \cup \{x_1,\dots, x_n\} \cup \{\psi_1,\dots,\psi_{n-1}\}
\]
subject to the following relations.
{\allowdisplaybreaks
\begin{align*}
e(\nu)e(\nu')&=\delta_{\nu,\nu'} e(\nu); \\
\sum_{\nu \in I^n} e(\nu)&=1;\\
x_re(\nu)&=e(\nu)x_r;\\
\psi_r e(\nu) &= e(s_r \nu) \psi_r;\\
x_rx_s&=x_sx_r;\\
\psi_rx_s&=\mathrlap{x_s\psi_r}\hphantom{\smash{\begin{cases}
\frac{Q_{\nu_r,\nu_{r+1}}(x_r,x_{r+1}) - Q_{\nu_r,\nu_{r+1}}(x_{r+2},x_{r+1})}{x_r - x_{r+2}} e(\nu)\end{cases}}}\kern\nulldelimiterspace
\text{if } s\neq r,r+1;\\
\psi_r\psi_s&=\mathrlap{\psi_s\psi_r}\hphantom{\smash{\begin{cases}
\frac{Q_{\nu_r,\nu_{r+1}}(x_r,x_{r+1}) - Q_{\nu_r,\nu_{r+1}}(x_{r+2},x_{r+1})}{x_r - x_{r+2}} e(\nu)\end{cases}}}\kern\nulldelimiterspace\text{if } |r-s|>1;\\
x_r \psi_r e(\nu) &=(\psi_r x_{r+1} - \delta_{\nu_r,\nu_{r+1}})e(\nu);\\
x_{r+1} \psi_r e(\nu) &=(\psi_r x_r + \delta_{\nu_r,\nu_{r+1}})e(\nu);\\
\psi_r^2 e(\nu)&=Q_{\nu_r,\nu_{r+1}}(x_r,x_{r+1})e(\nu);\\
(\psi_{r+1}\psi_{r}\psi_{r+1} - \psi_{r}\psi_{r+1}\psi_{r})e(\nu)&=\begin{cases}
\frac{Q_{\nu_r,\nu_{r+1}}(x_r,x_{r+1}) - Q_{\nu_r,\nu_{r+1}}(x_{r+2},x_{r+1})}{x_r - x_{r+2}} e(\nu) & \text{if } \nu_r=\nu_{r+2},\\
0 & \text{otherwise};\end{cases}
\end{align*}
}for all admissible $r,s,\nu,\nu'$, and $x_1^{\langle \alpha^\vee_{\nu_1},\Lambda \rangle} e(\nu)=0$ for $\nu \in I^n$.
\end{defn}

The algebra $R^\Lambda(n)$ is given a $\Z$-grading by setting
\[
\deg(e(\nu))=0, \qquad \deg(x_r e(\nu)) = (\alpha_{\nu_r},\alpha_{\nu_r}), \qquad \deg(\psi_s e(\nu)) = -( \alpha_{\nu_s}, \alpha_{\nu_{s+1}})
\]
for all admissible $r,s$ and $\nu$.

For $\beta \in \rlQ^+$ with $ \Ht(\beta) = n$, we set
\[
I^\beta = \{\nu \in I^n \mid \alpha_{\nu_1} + \dots + \alpha_{\nu_n} = \beta\}.
\]
Then $e(\beta):=\sum_{\nu\in I^\beta} e(\nu)$ is a central idempotent. We define $R^\Lambda(\beta):=R^\Lambda(n)e(\beta)$, which is also an $\bR$-algebra. 
It is clear that $R^\Lambda(\beta)$ may be defined by the same set of relations if we replace $I^n$ with $I^\beta$. 
We have the following decomposition of $R^\Lambda(n)$ into a direct sum of $\bR$-algebras. 
\[
R^\Lambda(n)=\bigoplus_{\substack{\beta\in \rlQ^+\\ \Ht(\beta) = n}} R^\Lambda(\beta).
\]
When we drop the relation $x_1^{\langle \alpha^\vee_{\nu_1},\Lambda \rangle} e(\nu)=0$ for $\nu \in I^\beta$, we obtain 
the \emph{quiver Hecke algebra} $R(\beta)$.

For each element $w \in \sg_n$, we fix a preferred reduced expression $w = s_{i_1} \dots s_{i_t}$ and define
\[
\psi_w := \psi_{i_1} \dots \psi_{i_t} \in R(\beta).
\]
Note that $\psi_w$ depends on the choice of reduced expressions of $w$ unless $w$ is fully commutative. The following comes from the defining relations.

\begin{prop}[\protect{\cite[Proof of Proposition 2.5]{BKW11}}] \label{psi_w}
For two reduced expressions $s_{i_1} \dots s_{i_t} = s_{j_1} \dots s_{j_t}$ for an element $w \in \sym{n}$,
$(\psi_{i_1} \dots \psi_{i_t} - \psi_{j_1} \dots \psi_{j_t}) e(\nu)$
can be written as a linear combination of elements of the form $\psi_u f(x) e(\nu)$, where $u \prec w$ with $\ell(u)\leq \ell(w)-3$, and $f(x)$ is a polynomial in the generators $x_1, \dots, x_n$.
\end{prop}

\begin{thm}[\cite{KL09, KL11, R08}] \label{basis thm}
Let $\beta \in \rlQ^+$ with $\Ht(\beta)=n$. Then
the set
\[\{
\psi_w x_1^{t_1} \dots x_n^{t_n} e(\nu) \mid w \in \sg_n, t_1, \dots, t_n \in \Z_{\ge 0}, \nu \in I^\beta
\}\]
is an $\bR$-basis of $R(\beta)$.
\end{thm}

\begin{prop}\label{prop: Z-free}
Suppose that $Q_{ij}(u,v)$ have integral coefficients. 
We denote the cyclotomic quiver Hecke algebra defined over $\Z$ by $R_\Z^{\Lambda}(n)$. Then $R_\Z^{\Lambda}(n)$ is free of finite rank over $\Z$. Further, 
$R^{\Lambda}(n) \simeq R_\Z^{\Lambda}(n)\otimes_\Z\bR$ as $\bR$-algebras.
\end{prop}

\begin{proof}
We prove by induction on $n$ that  $R^{\Lambda}(n)$ is a projective $\bR$-module. It is clear that  $R^{\Lambda}(0)=\bR$ is a projective $\bR$-module. Suppose that $R^{\Lambda}(n-1)$ is a projective $\bR$-module. By \cite[Thm 4.5]{KK11} , $R^{\Lambda}(n)$ is a projective $R^{\Lambda}(n-1)$-module. Thus, the induction hypothesis implies that $R^{\Lambda}(n)$ is a projective $\bR$-module. 
Applying the argument to $\bR=\Z$ and noting that $\Z$ is a principal ideal domain, we deduce that $R_\Z^{\Lambda}(n)$ is a free $\Z$-module of finite rank. 

As the defining relations of $R_\Z^{\Lambda}(n)$ hold in $R^{\Lambda}(n)$, the $\Z$-algebra homomorphism 
\[
R_\Z^{\Lambda}(n)\longrightarrow R^{\Lambda}(n)
\]
given by mapping the generators $\psi_i, x_j, e(\nu)$ to the corresponding generators is well-defined. Hence we have a surjective 
$\bR$-algebra homomorphism 
\[
R_\Z^{\Lambda}(n)\otimes_\Z\bR \longrightarrow R^{\Lambda}(n).
\]
On the other hand, as the defining relations of $R^{\Lambda}(n)$ hold in $R_\Z^{\Lambda}(n)\otimes_\Z\bR$, we have a surjective $\bR$-algebra homomorphism 
\[
R^{\Lambda}(n) \longrightarrow R_\Z^{\Lambda}(n)\otimes_\Z\bR.
\]
Thus, $R^{\Lambda}(n) \simeq R_\Z^{\Lambda}(n)\otimes_\Z\bR$. 
\end{proof}

Note that our choices $\eqref{Qij for affine C}$ and $\eqref{Qij for C_infty}$ of $Q_{ij}(u,v)$ being integral coefficients allow us to define the cyclotomic Hecke algebra over $\Z$.

\subsection{The $C_\infty$ case}

In this subsection, we carry out some computations in type $C_\infty$.
We choose the following system of polynomials $Q_{i,j}(u,v)$ as our preferred choice: 
if the Cartan matrix $\cmA$ is of type $C^{(1)}_\ell$ then, for $i<j$,
\begin{align} \label{Qij for affine C}
Q_{i,j}(u,v) = \begin{cases}
u+v^2 & \text{if } (i,j)=(0,1),\\
u+v & \text{if }  i\neq 0, j=i+1, j\neq \ell,\\
u^2+v & \text{if } (i,j)=(\ell-1,\ell),\\
1 & \text{otherwise,}
\end{cases}
\end{align}
and if the Cartan matrix $\cmA$ is of type $C_\infty$ then,  for $i<j$,
\begin{align} \label{Qij for C_infty}
Q_{i,j}(u,v) = \begin{cases}
u+v^2 & \text{if } (i,j)=(0,1),\\
u+v & \text{if } i\neq 0, j=i+1,\\
1 & \text{otherwise.}
\end{cases}
\end{align}
Note that if we assume that $\bR$ is a field and that any element of $\bR$ has a square root, then other choices of the polynomials $Q_{i,j}(u,v)$ yield isomorphic algebras 
\cite[Lemma 3.2]{AIP14}. 
Further we have the following graded dimension formulas.
For $\nu \in I^n$, let
\[
K_q(\lambda, \nu) := \; \sum_{\mathclap{\substack{\ttt \in \ST{\la}\\ \res{\ttt} = \nu}}} \; q^{\deg(\ttt)}, \qquad K_q(\lambda) := \; \sum_{\mathclap{\ttt \in \ST{\la}}} \; q^{\deg(\ttt)}.
\]

\begin{thm} \label{Thm: dimension formula}
For $\nu, \nu' \in I^\beta$, we have
\begin{align*}
\rank_q e(\nu) \fqH{}(\beta) e(\nu') &= \; \sum_{\mathclap{\substack{\lambda \in \Par_n\\ \wt(\lambda) = \Lambda - \beta}}}  \; K_q(\lambda, \nu) K_q(\lambda, \nu') ,\\
\rank_q \fqH{}(\beta) &= \; \sum_{\mathclap{\substack{\la \in \Par_n\\ \wt(\lambda) = \Lambda - \beta}}} \; K_q(\lambda)^2, \\
\rank_q \fqH{}(n) &= \; \sum_{\la \in \Par_n} \; K_q(\lambda)^2,
\end{align*}
where $\rank_q M:= \sum_{k\in \Z} \rank_\bR(M_k)q^k$ for a free graded $\bR$-module $M = \bigoplus_{k \in \Z} M_k$.
\end{thm}

\begin{proof} 
By virtue of \cref{prop: Z-free}, it suffices to prove the result when $\bR$ is a field.
The irreducible highest weight $U_q(\mathfrak{g}(\cmA))$-module with highest weight $\sum_{i=1}^l \Lambda_{\overline{\kappa_i}}\in\wlP^+$ is realised as the submodule 
$U_q(\mathfrak{g}(\cmA))\varnothing\subseteq \mathcal{F}(\kappa)$. 
Thus, the proof is entirely similar to \cite{AP16}. The only difference is that we use the tensor product Fock space $\mathcal{F}(\kappa)$. 
\end{proof}

Now we assume that the Cartan matrix $\cmA$ is of type $C_\infty$ and prepare some technical results. 
We consider fully commutative elements $S_2(c,a,b)=\shf_cw[a,b]$. Then $\psi_w$ for $w=S_2(c,a,b)$ does not depend on 
the choice of a preferred reduced expression. We denote it by $\Psi_2(c,a,b)$. If $c=0$ we denote it by $\Psi[a,b]$ instead. We have
\[
\deg(\Psi[a,b]e(\nu))=-\left( \sum_{k=1}^a \alpha_{\nu_k},\; \sum_{k=a+1}^{a+b} \alpha_{\nu_k} \right).
\]

For $1\le x\le y$, we define
\[
\psi\uparrow^y_x=\psi_x\psi_{x+1}\dots\psi_y \;\;\text{and}\;\; \psi\downarrow^y_x=\psi_y\psi_{y-1}\dots\psi_x.
\]
Then \cref{use of braid rels}(1) implies
\[
\Psi_2(c,a,b)=\psi\downarrow^{c+b}_{c+1}\dots\psi\downarrow^{c+a+b-1}_{c+a}=\psi\uparrow^{c+a+b-1}_{c+b}\dots\psi\uparrow^{c+a}_{c+1}.
\]
In particular, $\psi$-generators that appear in $\Psi_2(c,a,b)$ are $\psi_{c+1},\dots,\psi_{c+a+b-1}$. We also have the following formulae.
\begin{align*}
\Psi_2(c,a,b) &= \Psi_2(c,1,b)\Psi_2(c+1,1,b)\dots\Psi_2(c+a-1,1,b)\\
              &= \Psi_2(c+b-1,a,1)\dots\Psi_2(c+1,a,1)\Psi_2(c,a,1).\\[5pt]
\Psi_2(c,a,b) & = \Psi_2(c,x,b)\Psi_2(c+x,a-x,b) \;\; \text{for $0\le x\le a$},\\
              & = \Psi_2(c+y,a,b-y)\Psi_2(c,a,y) \;\; \text{for $0\le y\le b$}.
\end{align*}

\begin{Rmk}
The algebra $R(\beta)$ admits an anti-involution which fixes the generators. Then it sends 
$\Psi_2(c,a,b)$ to $\Psi_2(c,b,a)$ because
\[
\Psi_2(c,a,b)=\psi\downarrow^{c+b}_{c+1}\dots\psi\downarrow^{c+a+b-1}_{c+a} \mapsto
\psi\uparrow^{c+a+b-1}_{c+a}\dots\psi\uparrow^{c+b}_{c+1}=\Psi_2(c,b,a).
\]
\end{Rmk}

\begin{defn}
For $a_1,\dots,a_t\in \Z_{\ge0}$, we define a block transposition $S_i(a_1,\dots,a_t)$ by
\[
S_i(a_1,\dots,a_t)=S_2(\sum_{k=1}^{i-1} a_k, a_i, a_{i+1}).
\]
Then it is fully commutative and we may define $\Psi_i(a_1,\dots,a_t)$ by
\[
\Psi_i(a_1,\dots,a_t)=\Psi_2(\sum_{k=1}^{i-1} a_k, a_i, a_{i+1}).
\]
More generally, we define block permutations $S_{i_1}\dots S_{i_p}(a_1,\dots,a_t)$ by
\[
S_{i_1}\dots S_{i_p}(a_1,\dots,a_t) = S_{i_1}\dots S_{i_{p-1}}(a_{s_{i_p}(1)},\dots,a_{s_{i_p}(t)})S_{i_p}(a_1,\dots,a_t),
\]
and the corresponding $\Psi_{i_1}\dots \Psi_{i_p}(a_1,\dots,a_t)$ by
\[
\Psi_{i_1}\dots \Psi_{i_p}(a_1,\dots,a_t) = \Psi_{i_1}\dots \Psi_{i_{p-1}}(a_{s_{i_p}(1)},\dots,a_{s_{i_p}(t)})\Psi_{i_p}(a_1,\dots,a_t).
\]
\end{defn}

Observing that $s_{i_1},\dots,s_{i_p}$ permute places, the following is clear.

\begin{lem} \label{two line notation for block permutation}
Let $w=s_{i_1}\dots s_{i_p}\in \sg_t$ and $a_1,\dots,a_t\in \Z_{\ge0}$. If we define
\begin{align*}
A_{w^{-1}(1)}&=\{1,2,\dots,a_{w^{-1}(1)}\},\\
A_{w^{-1}(2)}&=\{ a_{w^{-1}(1)}+1,\dots, a_{w^{-1}(1)}+a_{w^{-1}(2)}\},\\
\vdots\qquad & \qquad\qquad\qquad \vdots \\
A_{w^{-1}(t)}&=\{ a_{w^{-1}(1)}+\dots+a_{w^{-1}(t-1)}+1,\dots, a_{w^{-1}(1)}+\dots+a_{w^{-1}(t)}\},
\end{align*}
then the two-line notation of $S_{i_1}\dots S_{i_p}(a_1,\dots,a_t)$ is given as follows.
\[
\begin{pmatrix}
A_{w^{-1}(1)} & \cdots &  A_{w^{-1}(t)}\\
A_1  & \cdots  & A_t
\end{pmatrix}
\]
\end{lem}

\begin{cor}\label{cor: reduced S_i}
Suppose that each $S_i$ is given by the reduced expressions in \cref{use of braid rels}(1). Then $S_{i_1}\dots S_{i_p} (a_1,\dots, a_t)$ is a reduced expression if and only if $S_{i_1}\dots S_{i_p} (1,\dots, 1)$ is.
\end{cor}

The two-line notation may be used to represent $\Psi_{i_1}\dots\Psi_{i_p}(a_1,\dots,a_t)e(\nu)$ by diagrams.

\begin{ex}
Let $\nu^1\in I^a$, $\nu^2\in I^b$, $\nu^3\in I^c$, for $a,b,c\ge1$, and $\nu=\nu^1*\nu^2*\nu^3$.  
Then $\Psi[a+b,c]e(\nu)$ is represented by
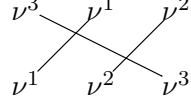
\[
\begin{xy}
(0,10) *{\nu^3}="A", (10,10) *{\nu^1}="B", (20,10) *{\nu^2}="C",
(0,0) *{\nu^1}="A'", (10,0) *{\nu^2}="B'", (20,0) *{\nu^3}="C'",

\ar@{-} "A";"C'"
\ar@{-} "B";"A'"
\ar@{-} "C";"B'"
\end{xy}
\]
and it follows that $\Psi[a+b,c]e(\nu)=\Psi_1\Psi_2(a,b,c)e(\nu)$.

\end{ex}

\begin{cor}
Let $\underline{a}=(a_1,\dots,a_t)$. If $j\ne i\pm1$ then $\Psi_i\Psi_j(\underline{a})=\Psi_j\Psi_i(\underline{a})$. 
\end{cor}

\begin{lem} \label{square of block transposition}
Suppose that the Cartan matrix $\cmA$ is of type $C_\infty$.
Let $\nu=(\nu_1,\nu_2, \dots, \nu_n) \in I^n$ and $a,b \in \Z_{>0}$ with $a < n$ and $a+b \le n$.
\begin{enumerate}
\item
If $|\nu_i-\nu_{a+1}|\ge2$, for $1\le i\le a$, then
$\Psi[1,a]\Psi[a,1]e(\nu)=e(\nu)$. 
\item
If $|\nu_1-\nu_i|\ge2$, for $2\le i\le a+1$, then
$\Psi[a,1]\Psi[1,a]e(\nu)=e(\nu)$.
\item
If $\nu_1 = \nu_{a+1} \ne \nu_2,\dots,\nu_{a}$, then $x_{a+1}\Psi[1,a]e(\nu) = (\Psi[1,a]x_1 + \Psi[1,a-1])e(\nu)$.
\item
If 
$\nu_1 = \nu_{a+1} \ne \nu_2,\dots,\nu_{a}$, then $x_1\Psi[a,1]e(\nu) = (\Psi[a,1]x_{a+1} - \Psi_2(1,a-1,1))e(\nu)$.
\item
If $|\nu_i-\nu_j|\ge2$, for $1\le i\le a$, $a+1\le j\le a+b$, then $\Psi[b,a]\Psi[a,b]e(\nu)=e(\nu)$.
\end{enumerate}
\end{lem}
\begin{proof}
(1) As $\Psi[1,a]=\psi_a\dots\psi_1$ and $\Psi[a,1]=\psi_1\dots\psi_a$, we have 
\begin{align*}
\Psi[1,a]\Psi[a,1]e(\nu)&=(\psi_a\dots\psi_2)\psi_1^2e(\mu^1)(\psi_2\dots\psi_a)\\
&=(\psi_a\dots\psi_2)(\psi_2\dots\psi_a)e(\nu),
\end{align*}
where $\mu^1=(\nu_1,\nu_{a+1},\nu_2,\dots,\nu_a, \nu_{a+2}, \dots, \nu_n)$, then
\begin{align*}
\Psi[1,a]\Psi[a,1]e(\nu)&=(\psi_a\dots\psi_3)\psi_2^2e(\mu^2)(\psi_3\dots\psi_a)\\
&=(\psi_a\dots\psi_3)(\psi_3\dots\psi_a)e(\nu),
\end{align*}
where $\mu^2=(\nu_1,\nu_2,\nu_{a+1},\nu_3,\dots,\nu_a, \nu_{a+2},  \dots, \nu_n)$, and so on.
We continue the computation in this way and we reach $\Psi[1,a]\Psi[a,1]e(\nu)=e(\nu)$.

\noindent
(2) The proof is similar to (1) and left to the reader.

\noindent
(3) By the assumption, we have
\begin{align*}
x_{a+1}\Psi[1,a]e(\nu)&=x_{a+1}\psi_a\dots\psi_1e(\nu)= (\psi_a x_a \psi_{a-1}\dots\psi_2\psi_1+ \psi_{a-1}\dots\psi_1) e(\nu)\\
&=(\Psi[1,a]x_1 + \Psi[1,a-1])e(\nu).
\end{align*}

\noindent
(4) By a similar computation to (3), we have
\begin{align*}
x_1\Psi[a,1]e(\nu)&=x_1\psi_1\dots\psi_ae(\nu) = (\psi_1x_2 \psi_2 \dots \psi_a - \psi_2 \dots \psi_a) e(\nu)\\
&=(\Psi[a,1]x_{a+1} - \Psi_2(1,a-1,1))e(\nu).
\end{align*}

\noindent
(5) We recall the formulas
\begin{align*}
\Psi[b,a]&= \Psi_2(a-1,b,1)\dots \Psi_2(1,b,1)\Psi[b,1],\\
\Psi[a,b]&= \Psi[1,b]\Psi_2(1,1,b)\dots \Psi_2(a-1,1,b).
\end{align*}
Then repeated use of (2) proves the result.
\end{proof}

\begin{lem} \label{failure of braid rels}
Suppose that the Cartan matrix $\cmA$ is of type $C_\infty$. Then 
\[
\psi_{r+1}\psi_r\psi_{r+1}e(\nu)=\psi_r\psi_{r+1}\psi_re(\nu)+ \epsilon(r,\nu)e(\nu)
\]
where $\epsilon(r,\nu)$ is given as follows.
\[
\epsilon(r,\nu)=
\begin{cases}
x_r+x_{r+2} \quad &\text{if}\; (\nu_r,\nu_{r+1},\nu_{r+2})=(1,0,1),\\
1 &\text{if}\; (\nu_r,\nu_{r+1},\nu_{r+2})=(i,i\pm1,i)\;\text{and}\; \nu_{r+1}\ne0,\\
0 &\text{otherwise.}\end{cases}
\]
\end{lem}

\begin{lem}\label{2exchange formula}
Let $\nu=(\mu_1\mu_2\nu_1\dots\nu_b)\in I^{b+2}$ for $b\ge1$. Then
\[
\psi_{b+1}\Psi[2,b]e(\nu)=\Psi[2,b]\psi_1e(\nu)+\sum_{k=1}^b \Psi_2(k,2,b-k)c_k\Psi[2,k-1]e(\nu),
\]
where $c_k$ is given by
\[
c_k=\begin{cases}
x_k+x_{k+2} \quad &\text{if}\; (\mu_1,\mu_2,\nu_k)=(1,0,1),\\
1 &\text{if}\; (\mu_1,\mu_2,\nu_k)=(i,i\pm1,i)\;\text{and}\; \mu_2\ne0,\\
0 &\text{otherwise.}\end{cases}
\]
\end{lem}
\begin{proof}
We follow \cref{use of braid rels}(2). Then, 
\begin{align*}
\psi_{b+1}\Psi[2,b]e(\nu)&=\underline{\psi_{b+1}}(\underline{\psi_{b}}\psi_{b-1}\dots \psi_1)(\underline{\psi_{b+1}}\psi_b\dots \psi_2)e(\nu)\\
&=\psi_{b+1}\psi_{b}\psi_{b+1}e(\nu_1\dots\nu_{b-1}\mu_1\mu_2\nu_b)(\psi_{b-1}\dots \psi_1)(\psi_b\dots \psi_2)\\
&=(\psi_b\psi_{b+1})\underline{\psi_{b}}(\underline{\psi_{b-1}}\psi_{b-2}\dots \psi_1)(\underline{\psi_{b}}\psi_{b-1}\dots \psi_2)e(\nu)\\
& \qquad + c_b(\psi_{b-1}\psi_{b-2}\dots \psi_1)(\psi_{b}\psi_{b-1}\dots \psi_2)e(\nu)\\
&=\quad \ldots\ldots\ldots\ldots \\
&=\Psi[2,b]\psi_1e(\nu)+\sum_{k=1}^b \Psi_2(k,2,b-k)c_k\Psi[2,k-1]e(\nu),
\end{align*}
where the error terms are computed by using \cref{failure of braid rels}.
Here, the underlines indicate generators to which we can apply the braid-type relation.
\end{proof}

\begin{lem} \label{exchange formula}
Let $a,b\ge1$ and $1\le i\le a-1$. Then
\begin{multline*}
\psi_{i+b}\Psi[a,b]e(\nu)-\Psi[a,b]\psi_ie(\nu)\\
=\sum_{k=1}^b \Psi[i-1,b]\Psi_2(i-1+k,2,b-k)c_k\Psi_2(i-1,2,k-1)\Psi_2(i+1,a-i-1,b)e(\nu)\\
=\sum_{k=1}^b \Psi_2\Psi_1\Psi_3(i-1,k,2,b-k)c_k\Psi_2\Psi_4\Psi_3(i-1,2,a-i-1,k-1,b-k+1)e(\nu),
\end{multline*}
where $c_k=x_{k+i-1}+x_{k+i+1}$ if $(\nu_i,\nu_{i+1},\nu_{a+k})=(1,0,1)$, $c_k=1$ if 
$(\nu_i,\nu_{i+1},\nu_{a+k})=(j,j\pm1,j)$ for some $j\ge0$ such that $\nu_{i+1}\ne0$, and $c_k=0$ otherwise.
\end{lem}
\begin{proof}
As the $\psi$-generators that appear in $\Psi_2(c,a,b)$ are $\psi_{c+1},\dots,\psi_{c+a+b-1}$, we have
\begin{align*}
\psi_{i+b}\Psi[a,b]e(\nu)&=\Psi[i-1,b]\psi_{i+b}\Psi_2(i-1,2,b)\Psi_2(i+1,a-i-1,b)e(\nu)\\
&=\Psi[i-1,b]\psi_{i+b}\Psi_2(i-1,2,b)e(\mu)\Psi_2(i+1,a-i-1,b)
\end{align*}
where $\mu=(\nu_1\dots\nu_{i+1}\nu_{a+1}\dots\nu_{a+b}\nu_{i+2}\dots\nu_a\nu_{a+b+1}\dots\nu_n)$. We apply \cref{2exchange formula} to substitute
\begin{align*}
\psi_{i+b}\Psi_2(i-1,2,b)e(\mu)&=\Psi_2(i-1,2,b)\psi_ie(\mu)\\
&\quad\quad +\sum_{k=1}^b \Psi_2(i-1+k,2,b-k)c_k\Psi_2(i-1,2,k-1)e(\mu).
\end{align*}
Then we have the desired formula.
\end{proof}

We may also compute $(\psi_{b+i_1}\dots\psi_{b+i_r}\Psi[a,b]-\Psi[a,b]\psi_{i_1}\dots\psi_{i_r})e(\nu)$
by applying Lemma \ref{exchange formula} to $(\psi_{b+i_k}\Psi[a,b]-\Psi[a,b]\psi_{i_k})e(s_{i_{k+1}}\dots s_{i_r}\nu)$ in the expression
\[
\sum_{k=1}^r \psi_{b+i_1}\dots\psi_{b+i_{k-1}}(\psi_{b+i_k}\Psi[a,b]-\Psi[a,b]\psi_{i_k})e(s_{i_{k+1}}\dots s_{i_r}\nu)\psi_{i_{k+1}}\dots \psi_{i_r}.
\]
In particular, we obtain the following. 

\begin{lem} \label{commutation formula1}
Let $a,b,c\ge1$ and $1\le m\le b$. Then,
\begin{multline*}
\Psi_2(c+m-1,a,1)\Psi[a+b,c]e(\nu)-\Psi[a+b,c]\Psi_2(m-1,a,1)e(\nu)\\
= \sum_{s=1}^a\sum_{t=1}^c \Psi_2(c+m-1,s-1,1)\Psi[m+s-2,c]\Psi_2(m+s+t-2,2,c-t)\qquad\qquad \\
\times c_{st}\Psi_2(m+s-2,2,t-1)\Psi_2(m+s,a+b-m-s,c)\Psi_2(m+s-1,a-s,1)e(\nu),
\end{multline*}
where $c_{st}=x_{m+s+t-2}+x_{m+s+t}$ if $(\nu_{m+s-1},\nu_{m+a},\nu_{a+b+t})=(1,0,1)$, 
$c_{st}=1$ if $  (\nu_{m+s-1}, \nu_{m+a}, $ $\nu_{a+b+t}) = (j,j\pm1,j)$ for some $j$ such that $\nu_{m+a}\ne0$, $c_{st}=0$ otherwise.
\end{lem}
\begin{proof}
As $\Psi_2(c+m-1,a,1)=\psi_{c+m}\dots\psi_{c+m+a-1}$, the left-hand side is
\[
\sum_{s=1}^a \psi\uparrow_{c+m}^{c+m+s-2}\bigl(\psi_{c+m+s-1}\Psi[a+b,c]-\Psi[a+b,c]\psi_{m+s-1}\bigr)e(\mu^s)\psi\uparrow_{m+s}^{m+a-1},
\]
where $\mu^s=s_{m+s}\dots s_{m+a-1}\nu=(\nu_1\dots\nu_{m+s-1}\nu_{m+a}\nu_{m+s}\dots\widehat{\nu_{m+a}}\dots\nu_n)$. Thus, we apply Lemma \ref{exchange formula} with $i=m+s-1$.
\end{proof}

\begin{Rmk} \label{commutation formula2}
We have prepared Lemma \ref{commutation formula1} for computing  
\begin{multline*}
\Psi_2\Psi_1\Psi_2(a,b,c)e(\nu)-\Psi_1\Psi_2\Psi_1(a,b,c)e(\nu)\\
=\Psi_2(c,a,b)\Psi_1(a,c,b)\Psi_2(a,b,c)e(\nu)-\Psi_1(b,c,a)\Psi_2(b,a,c)\Psi_1(a,b,c)e(\nu)
\end{multline*}
in later sections. If $a=0$ or $b=0$ or $c=0$ then it is zero. Thus we assume $a,b,c\ge1$. First we observe that
\begin{align*}
\Psi_1(a,c,b)\Psi_2(a,b,c)&=\Psi_2(0,a,c)\Psi_2(0+a,(a+b)-a,c)=\Psi[a+b,c],\\
\Psi_1(b,c,a)\Psi_2(b,a,c)&=\Psi_2(0,b,c)\Psi_2(0+b,(a+b)-b,c)=\Psi[a+b,c].
\end{align*}
Hence we compute $\Psi_2(c,a,b)\Psi[a+b,c]e(\nu)-\Psi[a+b,c]\Psi_1(a,b,c)e(\nu)$, which is equal to
\[
\sum_{m=1}^b \Psi_2(c+m,a,b-m)\bigl(\Psi_2(c+m-1,a,1)\Psi[a+b,c]-\Psi[a+b,c]\Psi_2(m-1,a,1)\bigr)\Psi[a,m-1]e(\nu)
\]
since $\Psi_2(c,a,b)=\Psi_2(c+b-1,a,1)\dots\Psi_2(c,a,1)$. Then
\[
\Psi[a,m-1]e(\nu)=e(w[a,m-1]\nu)\Psi[a,m-1],
\]
where $w[a,m-1]\nu=(\nu_{a+1}\dots\nu_{a+m-1}\nu_1\dots\nu_a\nu_{a+m}\dots\nu_n)$. To compute
\[
(\Psi_2(c+m-1,a,1)\Psi[a+b,c]-\Psi[a+b,c]\Psi_2(m-1,a,1))e(w[a,m-1]\nu)
\]
using Lemma \ref{commutation formula1}, we check whether
$(\nu_s,\nu_{a+m},\nu_{a+b+t})$ is $(1,0,1)$ or $(j,j+1,j)$ for $j\ge0$, or $(j,j-1,j)$ for $j\ge2$, for any $1\leq s\leq a$ and $1\leq t\leq c$. 
\end{Rmk}

\subsection{Module categories} \label{Sec: module categories} \

In the subsequent \cref{Sec: module categories,Sec: convolutions,Sec: dual space},
we keep the assumption that $\cmA$ is an arbitrary symmetrisable Cartan matrix but assume that $\bR$ is a field. 

We denote by $R(\beta)\proj$ and $R(\beta)\gmod$ the full subcategories in the category $R(\beta)\Mod$ of graded $R(\beta)$-modules
which consist of finitely generated projective graded $R(\beta)$-modules or finite dimensional graded
$R(\beta)$-modules, respectively. We set
\[
R\proj := \bigoplus_{\beta \in \rlQ^+} R(\beta)\proj \quad \text{and}\quad R\gmod := \bigoplus_{\beta \in \rlQ^+} R(\beta)\gmod.
\]

Similarly, $R^\Lambda(\beta)\proj$ and $R^\Lambda(\beta)\gmod$ are the full subcategories in the category  $R^\Lambda(\beta)\Mod$ of graded $R^\Lambda(\beta)$-modules
which consist of finitely generated projective graded $R^\Lambda(\beta)$-modules or finite dimensional graded
$R^\Lambda(\beta)$-modules, respectively. We set
\[
R^\Lambda\proj := \bigoplus_{\beta \in \rlQ^+} R^\Lambda(\beta)\proj \quad \text{and}\quad R^\Lambda\gmod := \bigoplus_{\beta \in \rlQ^+} R^\Lambda(\beta)\gmod.
\]

Let us denote by $q$ the {\em grading shift functor}, i.e.~$(qM)_k = M_{k-1}$ for a graded module $M = \bigoplus_{k \in \Z} M_k $.
For $M \in R(\beta)\gmod$, the {\em $q$-character} of $M$ is defined by
\[
\ch_q(M) := \sum_{\nu \in I^\beta} \dim_q (e(\nu)M) \nu,
\]
where $ \dim_q V := \sum_{k\in \Z} \dim(V_k)q^{k} $ for a graded vector space $V = \bigoplus_{k\in \Z} V_k $.

For graded $R(\beta)$-modules $M$ and $N $, we denote by $\Hom_{R(\beta)}(M,N)$ the space of degree preserving module homomorphisms. 
If $f \in \Hom_{R(\beta)}(q^{k}M, N)$, we set $ \deg(f) := k.$
Then we define the following graded vector space:
\[
\mathrm{HOM}_{R(\beta)}( M,N ) := \bigoplus_{k \in \Z} \Hom_{R(\beta)}(q^{k}M, N).
\]
We write $\Hom( M, N )$ and $\mathrm{HOM}( M,N )$ if there is no confusion. For $\beta, \beta' \in \rlQ^+$, we set
\[
e(\beta, \beta') := \; \sum_{\mathclap{\nu \in I^\beta, \nu' \in I^{\beta'}}} \; e(\nu * \nu').
\]

\begin{defn}
Let $M$ be a graded $R(\beta)$-module, $N$ a graded $R(\beta')$-module. Then the \emph{convolution product} $M\conv N$ is the graded $R(\beta+\beta')$-module defined by
\[
M \conv N := R(\beta+\beta') e(\beta, \beta') \otimes_{R(\beta) \otimes R(\beta')} (M \otimes N).
\]
\end{defn}

Let $\A = \Z[q,q^{-1}]$. We denote by $[R\proj]$ and $[R\gmod]$ the Grothendieck groups of $R\proj$ and $R\gmod$ respectively. 
The convolution product makes $[R\proj]$ and $[R\gmod]$ into $\A$-algebras, and we have the following theorem. 

\begin{thm} [{\cite{KL09, KL11, R08}}]
There exist isomorphisms of $\A$-algebras
\[[R\proj] \simeq  U_\A^-(\g) \quad \text{and} \quad [R\gmod] \simeq  U_\A^-(\g)^\vee.\]
\end{thm}

Now we explain the cyclotomic categorification theorem proved by Kang and Kashiwara.
For this, we introduce the induction and restriction functors $F_i^\Lambda$ and $E_i^\Lambda$, for $i\in I$, as follows.
\begin{itemize}
\item
The induction functors $F_i^\Lambda:R^\Lambda(\beta)\Mod \to R^\Lambda(\beta+\alpha_i)\Mod$ are defined by
\[
F_i^\Lambda=R^\Lambda(\beta+\alpha_i)e(\beta,\alpha_i)\otimes_{R^\Lambda(\beta)}\;-.
\]
\item
The restriction functors $E_i^\Lambda:R^\Lambda(\beta)\Mod \to R^\Lambda(\beta-\alpha_i)\Mod$ are defined by
\[
E_i^\Lambda=e(\beta-\alpha_i,\alpha_i)R^\Lambda(\beta)\otimes_{R^\Lambda(\beta)}\;-.
\]
\end{itemize}

The following theorem is proved by showing that $F_i^\Lambda$ and $E_i^\Lambda$ are biadjoint functors. 
The action of the Chevalley generators on the left-hand side of each of the isomorphisms in the theorem is given by the linear operators induced by the functors: for $\beta\in\rlQ^+$,
\begin{equation} \label{Eq: categorification}
\begin{aligned}
\xymatrix{
R^\Lambda(\beta)\proj  \ar@<0.3em>[rrrr]^{F_i^{\Lambda}} & & &  &  \ar@<0.3em>[llll]^{q^{(1-\langle \alpha_i^\vee, \Lambda-\beta \rangle)(\alpha_i, \alpha_i)/2}E_i^\Lambda}  R^\Lambda(\beta+\alpha_i)\proj   \\R^\Lambda(\beta)\gmod \ar@<0.3em>[rrrr]^{q^{(1-\langle \alpha_i^\vee, \Lambda-\beta \rangle)(\alpha_i, \alpha_i)/2}F_i^{\Lambda}} & & &  &  \ar@<0.3em>[llll]^{E_i^\Lambda}  R^\Lambda(\beta+\alpha_i)\gmod 
}
\end{aligned}
\end{equation}

\begin{thm}[\protect{\cite[Theorem 6.2]{KK11}}] \label{Thm: categorification thm}
For $\Lambda \in \wlP^+$, there exist $U_\A(\g)$-module isomorphisms
\[
[R^\Lambda\proj] \simeq V_\A(\Lambda), \quad [R^\Lambda\gmod] \simeq V_\A(\Lambda)^\vee.
\]
\end{thm}

\subsection{Convolution product for cyclotomic quiver Hecke algebras} \label{Sec: convolutions} \

The aim of this subsection is to prove the following.

\begin{prop} \label{conv with cyclotomic}
If $M \in R^\Lambda(\beta)\gmod$ for $\Lambda\in\wlP^+$ and $N \in R^{\Lambda'}(\beta')\gmod$ for $\Lambda'\in\wlP^+$, then
$M \conv N \in R^{\Lambda + \Lambda'}(\beta + \beta')\gmod.$
\end{prop}
\begin{proof}
Let $m = \Ht(\beta)$ and $n = \Ht(\beta')$. We may assume that $M$ and $N$ are non-zero modules and $m, n>0$, so that we may take non-zero elements $a\in e(\nu_1\dots\nu_m)M$ and $b\in e(\nu_{m+1}\dots\nu_{m+n})N$, 
for some $\nu = (\nu_1, \nu_2, \dots, \nu_{m+n}) \in I^{m+n}$. As $e(\nu_1\dots\nu_m)\ne0$ and $e(\nu_{m+1}\dots\nu_{m+n})\ne0$, 
the defining relations
\[
x_1^{\langle \alpha^\vee_{\nu_1},\; \Lambda \rangle}e(\nu_1\dots\nu_m)=0 \ \text{ and }\ x_1^{\langle \alpha^\vee_{\nu_{m+1}},\; \Lambda' \rangle}e(\nu_{m+1}\dots\nu_{m+n})=0
\]
imply $\langle \alpha^\vee_{\nu_1}, \Lambda \rangle > 0$ and $\langle \alpha^\vee_{\nu_{m+1}}, \Lambda' \rangle > 0$. We take $w \in \sg_{m+n}/\sg_{m} {\times} \sg_{n}$. Note that
\[
w(1)<\cdots<w(m) \;\; \text{and}\;\; w(m+1)<\cdots<w(m+n),
\]
so that $w$ is fully commutative and
\begin{align*}
&w^{-1}(1) = 1 \text{ or } m+1, \\
&e(\nu_{w^{-1}(1)}, \nu_{w^{-1}(2)}, \dots, \nu_{w^{-1}(m+n)}) \psi_w (a\otimes b) = \psi_w (a\otimes b).
\end{align*}
To prove the assertion, it suffices to show that 
\[
x_1^{l + l'} \psi_w (a \otimes b) = 0,
\]
where $l = \langle \alpha^\vee_{\nu_{ w^{-1}(1)}}, \Lambda \rangle$ and $l' = \langle \alpha^\vee_{\nu_{ w^{-1}(1)}}, \Lambda' \rangle$.

If $w^{-1}(1) = 1$, then $x_1 \psi_w =  \psi_w x_1$.  Thus, we have
\[
x_1^{l + l'} \psi_w (a \otimes b) = \psi_w x_1^{l + l'} (a \otimes b) = \psi_w ( ( x_1^{l + l'} a) \otimes b) = 0.
\]

Suppose that $w^{-1}(1) = m+1$. We set $u = w s_{m} s_{m-1} \dots s_{1}$, whose two-line notation is
\[
\bigl(
\begin{smallmatrix}
1 & 2 & \cdots & m+1 & m+2 & \cdots & m+n \\
1 & w(1) & \cdots & w(m) & w(m+2) & \cdots & w(m+n)
\end{smallmatrix}
\bigr),
\]
so that $u(1) = 1$ and $ \ell(u s_{1} s_{2} \dots s_{m}) = \ell(u) + \ell(s_{1} s_{2} \dots s_{m}) = \ell(u) + m$.
As $w$ is fully commutative, $\psi_w = \psi_u \psi_1 \psi_2 \dots \psi_m $.
It follows from
\begin{align*}
x_1 \psi_1 \psi_2 \dots \psi_m e(\nu) =  \psi_1 \psi_2 \dots \psi_m x_{m+1} e(\nu)
- \sum_{t=1}^m \delta_{\nu_{m+1}, \nu_t}  \psi_1 \dots \psi_{t-1} \psi_{t+1} \dots \psi_m e(\nu)
\end{align*}
and $s_1\dots s_{t-1}(\nu_1 \dots \nu_m)=(\nu_t \nu_1 \dots \widehat{\nu_t} \dots \nu_m)$ that
\begin{align*}
x_1^{l + l'} \psi_w (a \otimes b) &= x_1^{l + l'-1} \psi_u x_1 \psi_1 \psi_2 \dots \psi_m  (a \otimes b)  \\
&=
x_1^{l + l'-1} \psi_w x_{m+1} (a \otimes b) \\
& \quad - \sum_{t=1}^m  x_1^{l + l'-1} \delta_{\nu_{m+1}, \nu_t} \psi_u \psi_1 \dots \psi_{t-1} \psi_{t+1} \dots \psi_m (a\otimes b) \\
&=
x_1^{l + l'-1} \psi_w x_{m+1} (a \otimes b) \\
& \quad  - \sum_{t=1}^m   \delta_{\nu_{m+1}, \nu_t} \psi_u \psi_{t+1} \dots \psi_m ( (x_1^{l + l'-1} \psi_1 \dots \psi_{t-1} a )  \otimes b) \\
&=
x_1^{l + l'-1} \psi_w x_{m+1} (a \otimes b).
\end{align*}
Continuing this process, we have
\begin{align*}
x_1^{l + l'} \psi_w (a \otimes b) = x_1^{l} \psi_w x_{m+1}^{l'} (a \otimes b)
 = x_1^{l} \psi_w  (a \otimes (x_{1}^{l'} b) ) = 0,
\end{align*}
which completes the proof.
\end{proof}

\subsection{Dual space for the convolution product} \label{Sec: dual space} \

Let $\tau: R(\beta) \rightarrow R(\beta)$ be the graded anti-involution which is the identity on generators.
For $M \in \text{$R(\beta)\gmod$}$, we define $ M^{\circledast} := \mathrm{HOM}_{\bR} ( M, \bR ) $ to be the dual of $M$ whose
$R(\beta)$-action is given by $(xf)(v) = f( \tau(x) v )$ for $x \in R(\beta)$, $f\in M^{\circledast}$ and $v \in M$.

We take self-dual simple modules $M \in \text{$R(\beta)\gmod$}$ and $N \in \text{$R(\gamma)\gmod$}$ with $m = \Ht(\beta)$ and $n = \Ht(\gamma)$.
Let $\mathfrak{b}_M$ and $\mathfrak{b}_N $ be bases of $M$ and $N$ over $\bR$ respectively. Then
\[
\mathfrak{b}_{M\circ N} = \{ \psi_w\otimes x \otimes y \mid w \in \sg_{m+n}/ \sg_{m} {\times} \sg_{n}, \ x \in \mathfrak b_M, y\in \mathfrak b_N \}
\]
is a basis of $M \circ N$. We define
\[
\mathfrak{b}_{M\circ N}^{\circledast} = \{ \xi_w^{x,y} \mid w \in \sg_{m+n}/ \sg_{m} {\times} \sg_{n}, \ x \in \mathfrak b_M, y\in \mathfrak b_N  \} \subseteq (M\circ N)^{\circledast}
\]
to be the dual basis of $\mathfrak{b}_{M\circ N}$, i.e.~$\xi_w^{x,y} (\psi_{w'}\otimes x' \otimes y') = \delta_{(w,x,y),(w',x',y')} $.
It is known that there is an $R(\beta+\gamma)$-module isomorphism
\begin{align*}
 N \conv M \buildrel \sim \over \longrightarrow q^{-(\beta, \gamma)} (M\conv N)^{\circledast}
\end{align*}
which sends $ 1 \otimes y\otimes x $ to $\xi_{w[m,n]}^{x,y}$ for $y\in N$ and $x\in M$. See \cite[Theorem 2.2(2)]{lv11}.

\begin{lem}\label{lemma for dual}
The isomorphism $N \conv M \simeq q^{-(\beta, \gamma)}(M\conv N)^{\circledast}$ sends $\psi_{w}\otimes y \otimes x\in N\conv M$ to
\begin{align*}
\xi_{w^{-1}w[m,n]}^{x,y} + \sum_{\mathclap{\substack{w' \succ w^{-1}w[m,n]\\ x' \in M, y'\in N}}} \; a_{w', x', y'} \xi_{w'}^{x',y'} \in q^{-(\beta, \gamma)}(M\conv N)^{\circledast}
\end{align*}
for some $a_{w', x', y'} \in \bR$, and $\xi_{w^{-1}w[m,n]}^{x,y}\in q^{-(\beta, \gamma)}(M\conv N)^{\circledast}$ to
\begin{align} \label{dual}
\psi_{w}\otimes{y \otimes x} + \sum_{\mathclap{\substack{w' \prec w,\\ x' \in M, y'\in N}}} \; b_{w', x', y'} \psi_{w'}\otimes{y' \otimes x'}\in N \conv M
\end{align}
for some $b_{w', x', y'} \in \bR$.
\end{lem}
\begin{proof}
The first assertion is clear because $\psi_w\xi_{w[m,n]}^{x,y}$ has the desired form. The second assertion follows from the first.
\end{proof}

\vskip 1em

\section{Specht modules in affine and infinite type C} \label{Sec: Specht Sp}

In this section, we introduce Specht modules for cyclotomic quiver Hecke algebras in type $C^{(1)}_\ell$ or $C_\infty$ 
and provide a basis theorem for Specht modules in type $C_\infty$.
From now until \cref{defn: Specht over ring}, we assume that $\bR$ is a field.

\subsection{The modules $\Lm(k;\ell)$}\

For $k \in \Z$ and $\ell \in \Z_{>0}$, let
\[
\beta_{(k;\ell)} = \sum_{t=k}^{k+\ell-1} \alpha_{\overline{t}} \quad \text{and} \quad \nu_{(k;\ell)} = ( \overline{k}, \overline{k+1}, \dots, \overline{k+\ell-1}) \in I^{\beta_{(k ; \ell)}}.
\]
Then $\Lm(k;\ell)=\bR\gLm_{(k;\ell)}$ is the one-dimensional graded $R(\beta_{(k;\ell)})$-module defined by $\deg(\gLm_{(k;\ell)})=0$ and
\begin{align} \label{Eq: def of L}
x_i \gLm_{(k;\ell)}= \psi_j \gLm_{(k;\ell)}= 0,  \quad e(\nu) \gLm_{(k;\ell)}= \delta_{\nu, \nu_{(k;\ell)}  }\gLm_{(k;\ell)}
\end{align}
for $1 \le i \le \ell$, $1 \le j \le \ell-1$, $\nu \in I^{\beta_{(k;\ell)}}$.
If there is no confusion, we write $\gLm$ for $\gLm_{(k;\ell)}$ and 
we sometimes write $\Lm( \overline{k}, \overline{k+1}, \ldots, \overline{k+\ell-1} )$ instead of $\Lm(k;\ell)$.

Let $k \in \Z$ and $\ell_1, \ell_2 \in \Z_{\ge0}$. As
\[
\Lm(k;\ell_1) \otimes \Lm(k+\ell_1 ; \ell_2) \simeq   e( \beta_{(k;\ell_1)}, \beta_{(k+\ell_1 ; \ell_2)} ) \Lm(k;\ell_1+\ell_2)
\]
as an $R(\beta_{(k;\ell_1)}) \otimes R(\beta_{(k + \ell_1; \ell_2)}) $-module by construction, we have
\[
\Hom_{R(\beta_{(k;\ell_1+\ell_2)})}(\Lm(k;\ell_1) \conv \Lm(k+\ell_1 ; \ell_2), \Lm(k;\ell_1+\ell_2))\ne0
\]
by Frobenius reciprocity so that there exists a surjective $R(\beta_{(k;\ell_1+\ell_2)})$-module homomorphism
\begin{align} \label{eq: projection-p}
p: \Lm(k;\ell_1)  \conv \Lm(k+\ell_1 ; \ell_2) \longrightarrow \Lm(k; \ell_1+ \ell_2)
\end{align}
sending  $\gLm \otimes \gLm $ to $ \gLm $.
Taking the dual of $p$, we have the graded monomorphism
\[
\iota: \Lm(k; \ell_1+ \ell_2) \longhookrightarrow   q^{( \beta_{(k; \ell_1)},\ \beta_{(k+ \ell_1; \ell_2) } ) } \Lm(k+\ell_1 ; \ell_2) \conv \Lm(k;\ell_1).
\]
Then, noting that $p(\psi_w\otimes \gLm \otimes \gLm)=\psi_w\gLm$ implies $\iota(\gLm)=\xi_1^{\gLm,\gLm}$, $\eqref{dual}$ from Lemma \ref{lemma for dual} shows that
\[
\iota(\gLm_{(k; \ell_1+\ell_2)}) = \psi_{w[\ell_2, \ell_1]} (\gLm \otimes \gLm) + \; \sum_{\mathclap{\substack{w\in \sg_{\ell_1 + \ell_2}\\ w\prec w[\ell_2, \ell_1]}}} \; a_{w} \psi_w  (\gLm \otimes \gLm) \quad \text{for some }a_{w} \in \bR,
\]
with $\psi_w (\gLm \otimes \gLm) \in  e(\nu_{(k; \ell_1 + \ell_2)} ) \Lm(k+\ell_1 ; \ell_2) \conv \Lm(k;\ell_1)$ and 
$\deg(\psi_w (\gLm \otimes \gLm))=0$ whenever $a_w\neq 0$. 
Here, $\deg(\gLm\otimes\gLm)=( \beta_{(k; \ell_1)}, \beta_{(k+ \ell_1; \ell_2) } )$ because of the shift. 
The following lemma is easy to see by construction.

\begin{lem} \label{homo-r}
Define
\begin{align*}
r := \iota \circ p :  \Lm(k;\ell_1)  \conv \Lm(k+\ell_1 ; \ell_2) \rightarrow q^{( \beta_{(k; \ell_1)},\ \beta_{(k+ \ell_1; \ell_2) } ) } \Lm(k+\ell_1 ; \ell_2) \conv \Lm(k;\ell_1).
\end{align*}
\begin{enumerate}
\item Let $\gLm_1 = \gLm_{(k; \ell_1)}$ and $\gLm_2 = \gLm_{(k+ \ell_1; \ell_2)} $. Then
\[
r( \gLm_1 \otimes \gLm_2 ) = \psi_{w[\ell_2, \ell_1]} (\gLm_2 \otimes \gLm_1)
+ \sum_{\mathclap{\substack{w\in \sg_{\ell_1 + \ell_2}\\w\prec w[\ell_2, \ell_1]}}} a_{w} \psi_w  (\gLm_2 \otimes \gLm_1)
\quad \text{for some $ a_{w} \in \bR$},
\]
with $\psi_w (\gLm_2 \otimes \gLm_1) \in  e(\nu_{(k; \ell_1 + \ell_2)} ) \Lm(k+\ell_1 ; \ell_2) \conv \Lm(k;\ell_1)$ and $\deg(\psi_w (\gLm_2 \otimes \gLm_1))=0$ whenever $a_w\neq 0$.

\item $\im (r)$ is isomorphic to $ \Lm(k; \ell_1+ \ell_2) $.
\end{enumerate}
\end{lem}

\begin{cor} \label{Cor: homo-r for C_infty}
If the Cartan matrix is of type $C_\infty$, then
\[
r( \gLm_1 \otimes \gLm_2 ) = \psi_{w[\ell_2, \ell_1]} (\gLm_2 \otimes \gLm_1).
\]
\end{cor}
\begin{proof}
In type $C_\infty$, we know by examining residues that $e(\nu)\Lm(k+\ell_1 ; \ell_2) \conv \Lm(k;\ell_1) \ne 0$ if and only if $\nu$ is a shuffle of $\nu_{(k+\ell_1; \ell_2)}$ and $\nu_{(k; \ell_1)}$.
Thus it is straightforward to check that
\[
e(\nu_{(k; \ell_1 + \ell_2)}) \Lm(k+\ell_1 ; \ell_2) \conv \Lm(k;\ell_1)
= \Span_\bR\{ \psi_{w[\ell_2, \ell_1]} (\gLm_2 \otimes \gLm_1) \},
\]
which completes the proof by \cref{homo-r}.
\end{proof}

\begin{Rmk}
It is easy to show that $\Lm(k; \ell)$ admits an affinization for any $k$ and $\ell$. If $\cmA$ is of type $C_\infty$, then $\Lm(k; \ell)$ is real and $r$ in Lemma \ref{homo-r} is the $R$-matrix \cite{KP15}. Note that, if $\cmA$ is affine, $\Lm(k; \ell)$ is not real in general.
\end{Rmk}

\begin{prop} \label{homo-g}
Let $k\in \Z$ and $a, b,c \in \Z_{\ge0}$ with $b  \ge c > 0 $. Then, there is a non-zero $R(\beta_{(k;a)}+\beta_{(k-1;a+b+1)}+\beta_{(k+a;c-1)})$-module homomorphism
\begin{align*}
\rg : \Lm(k; a) \conv \Lm( k-1; a+b+1) & \conv \Lm( k+a; c-1) \\
& \longrightarrow
q^{ ( \beta_{( k-1;a+1)}, \beta_{( k+a;b)}) } \Lm(k;a+b) \conv \Lm(k-1; a+c)
\end{align*}
such that
\begin{align*}
\rg(\gLm \otimes \gLm \otimes \gLm ) =  \Psi_2(a, b, a+1)(\gLm \otimes \gLm)
+  \; \; \sum_{\mathclap{w \prec S_2(a, b, a+1) }} \; \; a_{w} \psi_w  (\gLm \otimes \gLm)
\quad \text{ for some $ a_{w} \in \bR$}.
\end{align*}
If the Cartan matrix $\cmA$ is of type $C_\infty$ then $a_w=0$ for all $w \prec S_2(a, b, a+1)$.
\end{prop}

\begin{proof}
Combining Lemma \ref{homo-r} with the surjectivity of $p$, we have a non-zero homomorphism
\begin{align*}
 &
\Lm(k; a) \conv \Lm(k-1; a+1) \conv \Lm(k+a; b) \conv   \Lm(k+a; c-1) \\
& \quad  \buildrel \id \conv r \conv \id \over \longrightarrow
 q^{ (\beta_{( k-1;a+1)}, \beta_{( k+a;b)} ) } \Lm(k; a) \conv \Lm(k+a; b) \conv \Lm(k-1; a+1)  \conv  \Lm(k+a; c-1) \\
& \quad \buildrel p \conv p \over  \longrightarrow
q^{ (\beta_{( k-1;a+1)}, \beta_{( k+a;b)} ) } \Lm(k; a+b)  \conv \Lm(k-1; a+c).
\end{align*}
Lemma \ref{homo-r} (2) tells us that the image of the first homomorphism is isomorphic to
\[
\Lm(k; a) \conv \Lm(k-1; a+b+1) \conv   \Lm(k+a; c-1),
\]
which is generated by
\begin{multline*}
\gLm \otimes \left(\Psi[b, a+1](\gLm \otimes \gLm)+\sum_{\mathclap{w \prec w[b, a+1] }} a_{w} \psi_w  (\gLm \otimes \gLm)\right)\otimes \gLm\\
= \Psi_2(a, b, a+1)(\gLm \otimes \gLm \otimes \gLm \otimes \gLm)
+  \sum_{\mathclap{w \prec S_2(a, b, a+1) }} \; \; a_{w} \psi_w  (\gLm \otimes \gLm \otimes \gLm \otimes \gLm)
\end{multline*}
by \cref{homo-r} (1). Thus it gives a non-zero homomorphism
\begin{align*}
\rg : \Lm(k; a) \conv \Lm( k-1; a+b+1) & \conv \Lm( k+a; c-1) \\
 & \longrightarrow
q^{ (\beta_{( k-1;a+1)}, \beta_{( k+a;b)} ) } \Lm(k;a+b) \conv \Lm(k-1; a+c)
\end{align*}
such that $\rg(\gLm \otimes \gLm \otimes \gLm)$ has the desired form.
\end{proof}

\subsection{The modules $\Sp^\lambda$} \label{Sec: Specht}  \

Let $ \lambda =(\lambda_1, \lambda_2, \ldots, \lambda_t) \in \mathscr{P}^1_n$ with a charge $\kappa \in \Z$.
Note that the level of $\lambda$ is 1. Let $\beta:=\cont(\la)$ and define
\begin{align*}
\Pe^{\la}_\kappa := \Lm(\kappa ; \lambda_1) \conv \Lm(\kappa-1 ; \lambda_2) \conv \cdots \conv \Lm(\kappa - t+1 ; \lambda_t) \in R(\beta)\gmod.
\end{align*}

For a Garnir node $A = (r,c) \in [\lambda]$, let
\begin{align*}
p_{\kappa, A}^{\la} &:=  ( \beta_{ (\kappa - r; c )}, \beta_{( \kappa-r+c; \la_r - c+1 )}  ) , \\
\Pe_{\kappa,A}^\la &:= \Pe^{\la_{<r}}_\kappa
\conv \Lm(\kappa - r +1; c-1) \conv \Lm(\kappa - r ; \lambda_r + 1)\conv \Lm(\kappa - r +c; \lambda_{r+1}- c)
\circ \Pe^{ \lambda_{> r+1}}_{\kappa-r-1},
\end{align*}
where $\lambda_{< r} = (\lambda_1, \ldots , \lambda_{r-1}) $ and
$\lambda_{> r+1} = (\lambda_{r+2}, \ldots , \lambda_{t}) $.
We denote by $m^\lambda_\kappa$ (resp.\ $m^\lambda_{\kappa,A}$) the distinguished generator $\gLm \otimes \dots \otimes \gLm$ of $\Pe^\lambda_{\kappa}$. (resp.\ $\Pe^\lambda_{\kappa,A}$.)
By Proposition \ref{homo-g}, we have the non-zero homomorphism
\begin{align*}
q^{-p_{\kappa,A}^\lambda}
\Lm(\kappa - r +1; c-1) \conv \Lm(\kappa - r ; \lambda_r + 1) & \conv \Lm(\kappa - r +c; \lambda_{r+1}- c) \\
& \qquad  \longrightarrow
\Lm(\kappa - r +1; \lambda_r) \conv \Lm(\kappa - r ; \lambda_{r + 1}).
\end{align*}
which gives the induced homomorphism
\[
\Gh^\lambda_{\kappa, A} : q^{-p_{\kappa,A}^\lambda} \Pe_{\kappa,A}^\lambda \longrightarrow  \Pe^\lambda_\kappa.
\]

\begin{defn}
Let $\la \in \mathscr{P}^1_n$ with a charge $\kappa \in \Z$. Then we define, for a Garnir node $A$,
\[g^{\lambda}_{\kappa,A}  =   \Gh^{\lambda}_{\kappa,A} (m^\lambda_{\kappa, A}).
\]
\end{defn}

By Proposition \ref{homo-g} and $\eqref{exp of GA}$, we have
\begin{align} \label{g_A and Garnir elt}
g^{\lambda}_{\kappa,A} - \psi_{w^{\ttg^A}} m^\lambda_\kappa = \sum_{u \prec w^{\ttg^A}} a_u \psi_u m^\lambda_\kappa \quad \text{for some }a_u \in \bR.
\end{align}

\begin{Rmk}\label{garnirformremark}
If the Cartan matrix $\cmA$ is of type $C_\infty$, then $g^{\lambda}_{\kappa,A} = \psi_{w^{\ttg^A}} m^\lambda_\kappa$. 
This is reminiscent of the Garnir element defined in \cite{kmr} for type $A_\infty$.
\end{Rmk}

\begin{lem} \label{lem: base}
We have
\begin{itemize}
\item[(i)]
$x_i g^{\lambda}_{\kappa,A} = 0$, for $1\le i \le n$,
\item[(ii)]
$\psi_j g^{\lambda}_{\kappa,A} = 0$ unless $s_j \ttg^A \in \rT{\la}$.
\end{itemize}
\end{lem}
\begin{proof}
For $1\le i \le n$, we have
\[
x_i g^{\lambda}_{\kappa,A} = x_i\Gh^{\lambda}_{\kappa,A} (m^\lambda_{\kappa, A}) = \Gh^{\lambda}_{\kappa,A} (x_i m^\lambda_{\kappa, A}) = 0.
\]

Let $A = (r,c) \in \la$ and $l_p = \sum_{k=1}^p \la_p$. Considering the definition of $\ttg^A$, we know that
\[
\text{$s_j \ttg^A \in \rT{\la}$ if and only if $j = l_1, \dots, l_{r-1}, l_{r-1}+c-1, l_{r}+c, l_{r+1}, \dots, l_t$.}
\]
Thus, by the construction of $m^\lambda_{\kappa, A}$,
\[
\psi_j g^{\lambda}_{\kappa,A} = \psi_j\Gh^{\lambda}_{\kappa,A} (m^\lambda_{\kappa, A}) = \Gh^{\lambda}_{\kappa,A} (\psi_j m^\lambda_{\kappa, A}) = 0
\]
unless $s_j \ttg^A \in \rT{\la}$.
\end{proof}

We define $\Gh^\la_\kappa: \oplus_{A} q^{-p_{\kappa,A}^\lambda} \Pe^\lambda_{\kappa,A} \rightarrow \Pe^\la_\kappa$ 
as the sum of $\Gh^\lambda_{\kappa,A}$ over Garnir nodes $A$ of $\lambda$ and set
\[
\Ga^{\la}_\kappa =  \im \  \Gh^{\la}_\kappa \subset \Pe^\la_\kappa, \quad \Sp^{\la}_\kappa = q^{\deg(\ttt^\lambda)} \coker\ \Gh^\la_\kappa.
\]
If there is no possibility of confusion, we will drop the subscript $\kappa$ from our notation, i.e.~we will simply write $\Pe^{\la}$, $m^\la$, $g^{\la}_A$, $\Sp^{\la}$, etc.

\begin{defn}\label{spechtdef}
For $\la = (\la^{(1)}, \dots, \la^{(l)} ) \in \Par_n$ and $\kappa = (\kappa_1, \dots, \kappa_l) \in \Z^l$,
we define
\begin{align*}
\Pe^\la = \Pe^\la_\kappa &:= \Pe^{\la^{(1)}}_{\kappa_1} \circ \cdots \circ \Pe^{\la^{(l)}}_{\kappa_l}. \\
\Sp^\la = \Sp^\la_\kappa &:=
q^{\deg(\ttt^\la)} \coker\ \Gh^{\la^{(1)}}_{\kappa_1} \conv \cdots \conv \coker\ \Gh^{\la^{(l)}}_{\kappa_l}.
\end{align*}
We write $\Sp^{\la}(\bR)$ when we need to emphasise the field.
\end{defn}

\begin{Rmk}
By Theorem \ref{basis thm}, the set
\begin{align} \label{basis of M}
\{
\psi_{w^\ttt} m^\lambda \mid \ttt \in \rST{\lambda}
\}
\end{align}
is an $\bR$-basis of $\Pe^\lambda$. 
\end{Rmk}

Note that $\deg \psi_{w^\ttt}m^\la = \deg \psi_{w^\ttt} e(\res \ttt)$ by definition. Then we have the following result. 

\begin{prop}
If $\ttt \in \ST\la$, then $\deg \psi_{w^\ttt} m^\la = \deg\ttt - \deg \ttt^\la$.
\end{prop}
\begin{proof}
We closely follow the proof of \cite[Proposition 3.13]{BKW11}. 
If $\ttt = \ttt^\la$, then we have 
\[
\deg \psi_{w^{\ttt^\la}} m^\la = \deg m^\la  =0 = \deg\ttt^\la - \deg \ttt^\la.
\]
Thus, it suffices to prove our statement in the case that $\tts,\ttt \in \ST\la$ are such that $\ell(\tts) = \ell(\ttt) +1$ and $\tts = s_r \ttt$. Let $\res\ttt = (\nu_1,\nu_2,\dots,\nu_n)$. We may assume that $r=n-1$. We want to show that $\deg \ttt - \deg \tts = (\alpha_{\nu_{n-1}},\alpha_{\nu_{n}})$. Let $A = \ttt^{-1} (n)$ and $B = \ttt^{-1} (n-1)$. By assumption, $B$ is above $A$ in $[\la]$. Now,
\begin{align*}
\deg \ttt &= d_A(\la) + d_B(\la \nearrow A) + \deg(\tabupto\ttt{n-2}),\\
\deg \tts &= d_B(\la) + d_A(\la \nearrow B) + \deg(\tabupto\tts{n-2}).
\end{align*}
Note that $\tabupto\ttt{n-2} = \tabupto\tts{n-2}$, and since $B$ is above $A$, $d_A (\la) = d_A (\la \nearrow B)$. So we must show that $d_B(\la \nearrow A) - d_B(\la) = (\alpha_{\nu_{n-1}},\alpha_{\nu_{n}})$.

If $\res A = \res B = i$, then removing $A$ leads to the disappearance of a removable $i$-node and the appearance of a new addable $i$-node below $B$, so that $d_B(\la \nearrow A) - d_B(\la) = 4$ if $i=0$, or $2$ otherwise.

If $\res A = 0$ and $\res B = 1$, removing $A$ leaves either one fewer addable $1$-node and one extra removable $1$-node,
or two extra removable $1$-nodes, or two fewer addable $1$-nodes, so that $d_B(\la \nearrow A) - d_B(\la) = -2$.

If $\res A = 1$ and $\res B = 0$, removing $A$ leaves either one extra removable $0$-node or one fewer addable $0$-node, so that $d_B(\la \nearrow A) - d_B(\la) = -2$.

If $\res A = \ell - 1$ and $\res B = \ell$ or $\res A = \ell$ and $\res B = \ell - 1$  in type $C^{(1)}_\ell$, 
similar arguments show that $d_B(\la \nearrow A) - d_B(\la) = -2$.

If $\res A = i \pm 1$ and $\res B = i$, with neither residue equal to $0$ or $\ell$, then removing $A$ leaves either one extra removable $i$-node or one fewer addable $i$-node, so $d_B(\la \nearrow A) - d_B(\la) = -1$.

In all other cases, removing $A$ does not change the degree, so $d_B(\la \nearrow A) - d_B(\la) = 0$.
\end{proof}

We denote by $\overline{m}^{\la}$ the image of $m^\la$ under the projection $q^{\deg(\ttt^\la)}\Pe^\la \rightarrow \Sp^\la$.

\begin{defn}\label{defn: Specht over ring}
Let $\bR$ be an integral domain. Then for $\la \in \Par_n$ and $\kappa \in \Z^l$, we define $\Sp^\la_\kappa(\bR)$ over $\bR$ to be the lattice 
$R_{\bR}(\cont(\la)) \overline{m}^{\la}$
generated by $\overline{m}^{\la}$ in $\Sp^\la_\kappa(\F)$, where $\F = \Frac(\bR)$ and $R_{\bR}(\cont(\la))$ is the quiver Hecke algebra over $\bR$.
\end{defn}

From now on, let $\bR$ denote an arbitrary integral domain.

\begin{thm} \label{Thm for Specht}
Let $\la \in \mathscr{P}^1_n$ with a charge $\kappa \in \Z$, and let $\beta = \cont(\la)$.
\begin{enumerate}
\item $\Sp^{\la}$ is generated by $\{ \psi_{w^\ttt}\overline{m}^{\la} \mid \ttt \in \ST\la \}$ as an $\bR$-module.
\item $\Sp^{\la}$ is a graded $\fqH{\overline{\kappa}} (\beta)$-module.
\end{enumerate}
\end{thm}

\begin{proof}
(1) For $\ell=0,1,\dots$, we define
\begin{align*}
A_\ell &:= \{ \psi_{w^\ttt}\overline{m}^{\la} \mid \ttt \in \rST\la, \ \ell(\ttt) \le \ell \} \subseteq \Sp^{\la}, \\
B_\ell &:= \{  \psi_{w^\ttt}\overline{m}^{\la} \mid \ttt \in \ST\la, \ \ell(\ttt) \le \ell \} \subseteq A_\ell .
\end{align*}
Then $\eqref{basis of M}$ implies that $\Sp^{\la}$ is generated by $\bigcup_{\ell \ge 0} A_\ell$ as an $\bR$-module, so it suffices to show that
\[
\Span_\bR A_\ell = \Span_\bR B_\ell
\]
for all $\ell \ge 0$ by induction on $\ell$. If $\ell = 0$, there is nothing to prove. 
Suppose that $\ell >0$ and take $\ttt \in \rST\la$ with $\ell = \ell(\ttt)$. We will show that
$\psi_{w^\ttt}\overline{m}^{\lambda} \in \Span_\bR B_\ell$. Since it is trivial when $\ttt \in \ST\la$,
we assume that $\ttt \in \rT\la$ and prove  $\psi_{w^\ttt} \overline{m}^\lambda \in \Span_\bR B_{\ell-1}$. We set
\begin{align*}
\Pe^\lambda_{\ell-1} &:= \Span_\bR \{  \psi_{w^\ttt} m^{\lambda} \mid \ttt \in \rST\la, \ \ell(\ttt) \le \ell-1  \} \subseteq \Pe^\la,\\
\Sp^\lambda_{\ell-1} &:= \Span_\bR A_{\ell-1} \subseteq \Sp^\la.
\end{align*}
By Lemma \ref{straighten}, there are a Garnir node $A \in [\lambda]$ and an element $w \in \sg_n$ such that
$\ttt = w\ttg^A$ and $\ell(\ttt) = \ell(w) + \ell(\ttg^A)$.
It follows from Proposition \ref{psi_w} and $\eqref{basis of M}$ that
\[
\psi_{w^\ttt} m^\lambda - \psi_{w} \psi_{w^{\ttg^A}} m^\lambda \equiv 0 \pmod{\Pe^\lambda_{\ell-1}}.
\]
By $\eqref{g_A and Garnir elt}$, we have
\[
\psi_{w} \psi_{w^{\ttg^A}} m^\lambda - \psi_{w} g_A^\lambda \equiv 0 \pmod{\Pe^\lambda_{\ell-1}},
\]
which implies that
\[
\psi_{w^\ttt} \overline{m}^\lambda \equiv 0 \pmod{\Sp^\lambda_{\ell-1}},
\]
proving $\psi_{w^\ttt} \overline{m}^\lambda \in \Span_\bR B_{\ell-1} $ by the induction hypothesis $\Span_\bR A_{\ell-1} = \Span_\bR B_{\ell-1}$.

(2) It follows from (1) that it suffices to prove $x_1^{\langle \alpha_{\nu_1}^\vee, \Lambda_{\overline{\kappa}} \rangle}e(\nu)\psi_{w^\ttt}m^{\la} = 0$, for 
$\ttt \in \ST\la$. But if $\ttt \in \ST\la$ then $w^\ttt(1)=1$, so that $\psi_{w^\ttt}$ is a product of $\psi_2, \dots, \psi_{n-1}$ and $x_1\psi_{w^\ttt}=\psi_{w^\ttt}x_1$ holds. 
Then, since $ e(\nu) \psi_{w^\ttt} m^{\lambda} \ne 0 $ implies $ \nu_1 = \overline{\kappa}$, 
\[
x_1^{\langle \alpha_{\nu_1}^\vee, \Lambda_{\overline{\kappa}} \rangle
}e(\nu)\psi_{w^\ttt}m^{\la} = x_1e(\nu) \psi_{w^\ttt} m^{\lambda}  = e(\nu)\psi_{w^\ttt}x_1 m^{\lambda} = 0.\qedhere
\]
\end{proof}

\begin{cor} \label{Cor: Specht modules l}
Let $l \in \Z_{>0}$, $\la \in \Par_n$, $\kappa = (\kappa_1, \dots, \kappa_l) \in \Z^{l}$, and let $\beta = \cont(\la)$.
\begin{enumerate}
\item $\Sp^{\la}$ is generated by $\{ \psi_{w^\ttt} \overline{m}^{ \la } \mid \ttt \in \ST\la \}$ as an $\bR$-module.
\item Let $\Lambda = \Lambda_{\overline{\kappa_1}}+ \dots + \Lambda_{\overline{\kappa_l}}$. Then $\Sp^{\la}$ is a graded $R^{\Lambda}(\beta)$-module.
\end{enumerate}
\end{cor}
\begin{proof}
This follows from Theorem \ref{Thm for Specht} and Proposition \ref{conv with cyclotomic}.
\end{proof}

\begin{defn}
Let $l \in \Z_{>0}$, $\kappa = (\kappa_1, \dots, \kappa_l) \in \Z^{l}$ and $\Lambda = \Lambda_{\overline{\kappa_1}}+ \dots + \Lambda_{\overline{\kappa_l}}$. 
Then we call the graded $R^{\Lambda}(\beta)$-modules $\Sp^{\la}$, for $\la \in \Par_n$, \emph{Specht modules}.
\end{defn}

\begin{Rmk}
One can easily construct a `column version' of the Specht modules by the same argument. (\protect{cf.~\cite[Section 7]{kmr}}).
\end{Rmk}

\begin{ex}
Let $\kappa \in \Z$ and $\la = (n), \la' = (1^n) \in \mathscr{P}^1_n$. It is straightforward to prove that 
\[
\Sp^\la_{\kappa} \simeq \Lm( \kappa ;n ) \simeq \Sp^{\la'}_{-\kappa} .
\]
In particular, $\Sp^\la_0 \simeq  \Sp^{\la'}_0$.
\end{ex}

\begin{ex}
Suppose that $\cmA$ is of type $C_\infty$ or $C_\ell^{(1)}$ with $\ell > 2$. Let $\kappa=-1$ and $\lambda = (4),  \mu=(3,1) \in \mathscr{P}^1_4 $. As $\la$ has no Garnir nodes, we have 
\[
\Sp^\la = \Lm(1012).
\]
Since $\mu$ has only one Garnir node $(1,1)$, we have $p_{(1,1)}^{\mu} = ( \alpha_2, \alpha_1 + \alpha_0 + \alpha_1) = -2 $ and 
\[
\Gh^\mu : q^{2} \Lm(2101) \longrightarrow  \Pe^\mu := \Lm(101)\circ \Lm(2), \qquad \gLm \mapsto \psi_1\psi_2\psi_3  m^\mu. 
\] 
Thus, we have $ \Ga^{\la} \simeq  q^{2} \Lm(2101)$ and 
\[
\ch_q\Sp^\mu = (1012) + q(1021) + q(1201).
\]
The epimorphism $ \Lm(101) \circ \Lm(2) \twoheadrightarrow \Lm(1012)$ gives the epimorphism 
\[
\Sp^\mu \longtwoheadrightarrow \Sp^\la,
\]
which tells us that $\Sp^\mu$ is not simple and the head of $\Sp^\mu$ is isomorphic to $\Sp^\la$.
\end{ex}

\begin{ex} \label{Ex: 1}
Suppose that $\cmA$ is of type $C_\infty$ or $C_\ell^{(1)}$ with $\ell > 2$. Let $\kappa = 0$ and $\la = (3,2,1) \in \mathscr{P}^1_6$. Then 
$ \deg(\ttt^\la) = 1$, $ \Pe^\la = \Lm(012) \circ \Lm(10) \circ \Lm(2) $ and the Garnir nodes of $\la$ are $A_1 := (1,1)$, $A_2 := (1,2)$ and $A_3 := (2,1)$.
Since
\begin{align*}
p_{A_1}^\la &= ( \alpha_1, \alpha_0+\alpha_1+\alpha_2 ) = -1, \\ 
p_{A_2}^\la &= ( \alpha_1+\alpha_0, \alpha_1+\alpha_2 ) = -1,  \\
p_{A_3}^\la &= ( \alpha_2, \alpha_1+\alpha_0 ) = -1, 
\end{align*}
we have
\begin{align*}
\Gh_{A_1}^\la : & q \Lm(1012) \circ \Lm(0) \circ \Lm(2) \longrightarrow  \Pe^\la, \qquad m^\la_{A_1} \longmapsto \psi_1\psi_2\psi_3  m^\la, \\ 
\Gh_{A_2}^\la : & q \Lm(0) \circ \Lm(1012) \circ \Lm(2) \longrightarrow  \Pe^\la, \qquad m^\la_{A_2} \longmapsto \psi_3\psi_2\psi_4\psi_3  m^\la, \\ 
\Gh_{A_3}^\la : & q \Lm(012) \circ  \Lm(210) \longrightarrow  \Pe^\la, \qquad \qquad \   m^\la_{A_3} \longmapsto \psi_4\psi_5  m^\la.
\end{align*}
Thus, $\Ga^{\la} = \langle  \psi_1\psi_2\psi_3  m^\la,  \psi_3\psi_2\psi_4\psi_3  m^\la, \psi_4\psi_5  m^\la \rangle \subset \Pe^\la$ and 
\[
\Sp^\la = q \Pe^\la / \Ga^\la.
\]
\end{ex}

\subsection{Basis theorem for type $C_\infty$}\

Suppose that the Cartan matrix is of type $C_\infty$. Then we have the following basis theorem for Specht modules, whose proof is postponed to \cref{Sec: proof of the main thm}.

\begin{thm} \label{Thm: Specht modules}
Let $\la \in \mathscr{P}^1_n$ with a charge $\kappa \in \Z$. Then the set $\{ \psi_{w^\ttt}\overline{m}^{\la} \mid \ttt \in \ST\la  \}$ is an $\bR$-basis of $\Sp^{\la}$. Moreover, we have the following graded character formula.
\[
\ch_q \Sp^{\la} = \; \sum_{\mathclap{\ttt \in \ST\la }} \; q^{\deg(\ttt)} \res\ttt.
\]
\end{thm}

\begin{cor} \label{Cor: Specht 1}
In the Grothendieck group of $\fqH{\overline{\kappa}}(n-1)\gmod$, we have
\[
[E_i^{\Lambda_{\overline{\kappa}}} \Sp^{\lambda}] = \sum_{b} q^{d_b(\lambda)} [ \Sp^{\la \nearrow b} ],
\]
where $b$ runs over all removable $i$-nodes.
\end{cor}
\begin{proof}
We rewrite the graded character formula from Theorem \ref{Thm: Specht modules} as follows.
\[
\ch_q \Sp^{\la} = \; \sum_b \quad \; \sum_{\mathclap{\ttt \in \ST{\la\nearrow b} }} \; q^{\deg(\ttt)+d_b(\la)} \res\ttt*\res{b},
\]
where $b$ runs over all removable nodes. Thus,
\[
\ch_q(E_i^{\Lambda_{\overline{\kappa}}}\Sp^{\la})= \sum_{b} q^{d_b(\lambda)}\ch_q(\Sp^{\la \nearrow b}),
\]
where $b$ runs over all removable $i$-nodes.
\end{proof}

Let $l \in \Z_{>0}$, $\kappa = (\kappa_1, \ldots, \kappa_l) \in \Z^{l}$ and $\Lambda=\Lambda_{\overline{\kappa}_1}+ \dots +\Lambda_{\overline{\kappa}_l}$. 
One can easily prove \cref{Cor: Specht modules} from \cref{Thm: Specht modules} and \cref{Cor: Specht 1}.

\begin{cor}\label{Cor: Specht modules}
Let $\la \in \Par_n$.
\begin{enumerate}
\item
The set $\{ \psi_{w^\ttt} \overline{m}^{ \lambda } \mid \ttt \in \ST\la  \}$ is an $\bR$-basis of $\Sp^{ \lambda }$. Moreover,
\[
\ch_q \Sp^{ \lambda } = \; \sum_{\mathclap{\ttt \in \ST\la}} \; q^{\deg(\ttt)} \res\ttt.
\]
\item In the Grothendieck group of $\fqH{}(n-1)\gmod$, we have
\[
[E_i^\Lambda \Sp^{ \lambda }] = \sum_{b} q^{d_b(\lambda)} [ \Sp^{ \lambda  \nearrow b} ],
\]
where $b$ runs over all removable $i$-nodes.
\end{enumerate}
\end{cor}

We revisit the Fock space $\mathcal{F}(\kappa)$. 
As $V_q(\Lambda) \simeq V_q(\Lambda)^{\vee}$, by \cref{Thm: categorification thm}, we can identify 
$V_q(\Lambda) \simeq V_q(\Lambda)^\vee \simeq  \Q(q)\otimes_\A [R^\Lambda\gmod]$. 
Thus, from $\eqref{Eq: pk}$, we have the $U_q(\g(\cmA))$-module epimorphism
\begin{align*}
p_\kappa: \mathcal{F}(\kappa) \longtwoheadrightarrow \Q(q)\otimes_\A [R^\Lambda\gmod].
\end{align*}

\begin{prop} \label{Prop: categorification of pk}
For $\la \in \Par_n$, we have
\begin{align*}
p_\kappa( \la  ) = [\Sp^\la].
\end{align*}
\end{prop}
\begin{proof}
It is obvious that $p_\kappa( \varnothing ) = [\Sp^{\varnothing}]$ and $\wt(\varnothing) = \wt([\Sp^{\varnothing}]) = \Lambda $.
Since both of $\mathcal{F}(\kappa) $ and $ \Q(q)\otimes_\A [R^\Lambda\gmod]$ are integrable $U_q(\g(\cmA))$-modules and $ \Q(q)\otimes_\A [R^\Lambda\gmod]$ is simple,
it suffices to show that $e_i ( p_\kappa(\la) - [\Sp^\la] )  = 0$ for all $\la$.  
By $\eqref{Eq: def of Fock sp}$, \cref{Cor: Specht modules} and the induction hypothesis, we have
\begin{align*}
e_i p_\kappa(\la) & =  p_\kappa( e_i \la) \\ 
&=  p_\kappa \left(\sum_{A} q^{d_A(\la)} \la\nearrow A \right) \\ 
& =  \sum_{A} q^{d_A(\la)} p_\kappa( \la\nearrow A)
=  \sum_{A} q^{d_A(\la)}  [\Sp^{\la\nearrow A}]
= [E_i^\Lambda \Sp^\la],
\end{align*}
which completes the proof.
\end{proof}

\cref{Cor: F_i Sp} follows from $\eqref{Eq: def of Fock sp}$, $\eqref{Eq: categorification}$ and \cref{Prop: categorification of pk}.

\begin{cor} \label{Cor: F_i Sp}
Let $\beta = \cont(\la)$. In the Grothendieck group of $\fqH{}(n+1)\gmod$, we have
\[
[F_i^\Lambda \Sp^{ \lambda }] = \sum_{b} q^{-d^b(\lambda) + (\langle \alpha_i^\vee,  \Lambda - \beta \rangle -1) (\alpha_i, \alpha_i)/2} [ \Sp^{ \lambda  \swarrow b} ],
\]
where $b$ runs over all addable $i$-nodes.
\end{cor}

\begin{ex}
We use the same notation as in Example \ref{Ex: 1}. Let $ \mu=(2,2)$, $\mu_1 = (3,2)$ and $\mu_2 = (2,2,1)$. 
By \cref{Thm: Specht modules}, we have
\begin{align*}
\ch_q \Sp^\la &= [2]_q(012102) + [2]_q(012120) + [2]_q^2(011202) + [2]_q^2(011220) + [2]_q^2(011022), \\ 
\ch_q \Sp^{\mu} &= [2]_q(0110), \\ 
\ch_q \Sp^{\mu_1} &= q(01210) + q[2]_q(01120) + q[2]_q(01102), \\ 
\ch_q \Sp^{\mu_2} &= (01210) + [2]_q(01120) + [2]_q(01102),
\end{align*}
where $[n]_q := \frac{q^n - q^{-n}}{q - q^{-1}}$ for $n\in \Z_{\ge0}$.
Let $B_1 = (1,3)$ and $B_2 = (3,1)$. Then  
\begin{align*}
d_{B_1}(\la) = -1,  \quad d_{B_2}(\la) = 0, \quad d^{B_1}(\mu) = 0, \quad d^{B_2}(\mu) = 1, \quad  
\langle \alpha_2^\vee,  \Lambda_0 - 2\alpha_0 - 2\alpha_1 \rangle = 2.
\end{align*}
By \cref{Cor: Specht 1} and \cref{Cor: F_i Sp}, we have 
\begin{align*}
\ch_q E^{\Lambda_0}_2 \Sp^\la &  = \ch_q \Sp^{\mu_1}  + q^{-1} \ch_q \Sp^{\mu_2} = [2]_q(01210) +  [2]_q^2(01120) + [2]_q^2(01102), \\
\ch_q F^{\Lambda_0}_2 \Sp^\mu & = q \ch_q \Sp^{\mu_1}  + \ch_q  \Sp^{\mu_2} =  q[2]_q(01210) + q[2]_q^2(01120) + q[2]_q^2(01102) .
\end{align*}

\end{ex}

\begin{Rmk} 
It looks like we need a modified version of the upper global basis in the Fock space to describe the simple modules.
It is an interesting problem to characterise the elements in the Fock space which correspond to the simple modules.
\end{Rmk}

\section{Proof of \cref{Thm: Specht modules}} \label{Sec: proof of the main thm}

We assume that the Cartan matrix $\cmA$ is of type $C_\infty$ and take the parameters $\eqref{Qij for C_infty}$ for the quiver Hecke algebra $R(\beta)$.
Let us fix $\la \in \mathscr{P}^1_n$, $\kappa \in \Z$ and $\beta = \cont(\la)$.

\begin{defn}
For $t \in \Z_{>0}$, we define
\begin{align*}
\Ga^{\lambda}_{ < t} := & \text{ $\bR$-submodule of $\Ga^{\lambda}$ generated by $\psi_{w} g^\lambda_A$ for all Garnir nodes $A \in [\lambda]$ and}\\
& \text{all $w \in \sg_n$ such that $ w \ttg^A \in \rT \la$ and $ \ell( w \ttg^A ) = \ell(w) + \ell(\ttg^A) < t$.}
\end{align*}
\end{defn}

\begin{Rmk}
Note that we require $ w \ttg^A \in \rT \la$ in the definition of $\Ga^{\lambda}_{ < t}$. In Theorem \ref{thm: basis of g^la} below, 
we will eliminate the possibility that $\Ga^{\lambda}$ is strictly larger than $\sum_{t\in \Z_{>0}} \Ga^{\lambda}_{ < t}$.
\end{Rmk}

\cref{lem: row-strict,lem: T and G A} are needed for proving Lemma \ref{lem: T and G AB}.

\begin{lem} \label{lem: row-strict}
Let $\ttt \in \rST \la $.
\begin{enumerate}
\item[(1)] If
\begin{enumerate}
\item[(i)] $\la = (\la_1, \la_2)$ with $\la_2 > 0$,
\item [(ii)] $\res{ \ttt^{-1}(n) } = \overline{\kappa + \la_1 - 1}$ and $ \res{ \ttt^{-1}(n-1) } = \overline{\kappa + \la_1 - 2}$,
\end{enumerate}
then $\ttt(1, \lambda_1) = n$.
\item[(2)]
If
\begin{enumerate}
\item[(i)] $\la = (\la_1, \la_2, \la_3)$ with $\la_3 > 0$,
\item [(ii)] $ \res{ \ttt^{-1}(n) } = \overline{\kappa + \la_1 - 1}$, $ \res{ \ttt^{-1}(n-1) } = \overline{\kappa + \la_1 - 2}$ and
$\res{ \ttt^{-1}(n-2) } = \overline{\kappa + \la_1 - 3}$,
\item[(iii)] $\res{3,\la_3} \ne \res{1,\la_2}$,
\end{enumerate}
then $\ttt(1, \lambda_1) = n$.
\end{enumerate}
\end{lem}

\begin{proof}
(1) As $\ttt \in \rST \la $, $\ttt^{-1}(n)$ = $(1, \lambda_1)$ or $(2, \lambda_2)$. If $\kappa + \la_1 - 1 \leq0$ then $(1,\la_1)$ is the only node of residue $\overline{\kappa + \la_1 - 1}$, 
so that  $ \ttt(1, \la_1) = n$. Suppose that $\kappa + \la_1 - 1 >  0$ and  $\ttt^{-1}(n)$ = $(2, \la_2)$. Then 
the assumption (ii) implies that
\[
\kappa + \la_2 -2 = - (\kappa + \la_1 -1) < 0
\]
and $\ttt^{-1}(n-1)$ = $(1, \la_1)$ or $(2, \la_2-1)$. Thus, $\res{\ttt^{-1}(n-1) }$ is either $\kappa + \la_1 -1$ or $\kappa + \la_1$, which are not equal to 
$\overline{\kappa + \la_1 - 2}=\kappa + \la_1 -2$.

(2) As $\ttt \in \rST \la $, we know that $\ttt^{-1}(n)$ = $(1, \lambda_1)$, $(2, \lambda_2)$ or $(3, \lambda_3)$.
Further, by the same reasoning as in (1), $\kappa + \la_1 - 1 >  0$ and $\ttt^{-1}(n-1) \ne (1,\la_1)$ hold if $\ttt^{-1}(n)$ = $(2, \lambda_2)$ or $(3, \lambda_3)$. 

Suppose that $\ttt^{-1}(n) = (2, \la_2)$. Then
\[
\kappa + \la_2 -2 = - (\kappa + \la_1 -1) < 0
\]
and $\ttt^{-1}(n-1)$ = $(2, \la_2 - 1)$ or $(3, \la_3)$. Thus, $\res{\ttt^{-1}(n-1) } \ge \kappa + \la_1$, which is not equal to 
$\overline{\kappa + \la_1 - 2}=\kappa + \la_1 -2$.

Now suppose that $\ttt^{-1}(n) = (3, \lambda_3)$. Then
\[
\kappa + \la_3 -3 = - (\kappa + \la_1 -1) < 0
\]
and  $\ttt^{-1}(n-1)$ = $(2, \la_2)$ or $(3, \la_3-1)$. Then $\res{ \ttt^{-1}(n-1) }$ = $\kappa + \la_1 -2 < \res{ \ttt^{-1}(n)}$ implies that  $\ttt^{-1}(n-1)$ = $(2, \la_2)$.
In particular, $\ttt^{-1}(n-2)$ = $(1,\la_1)$, $(2,\la_2-1)$ or $(3,\la_3-1)$.

If $\kappa + \la_1 -1 =1$, then $\res{ \ttt^{-1}(n-1)}=0 $ by condition (ii). It follows that $\kappa + \la_2 -2 = 0$ and $\kappa + \la_3 -3 = -1$, so that $\la_1 = \la_2 = \la_3$. But this contradicts condition (iii).

If $\kappa + \la_1 - 1 \ge 2$, then, since $\res{ \ttt^{-1}(n-1)}=\kappa + \la_1 -2$, we have one of the following. 
\begin{itemize}
\item[(a)]
$\kappa + \la_2 -2 = \kappa + \la_1 -2 > 0$ and $\la_1 = \la_2$. 
\item[(b)]
$\kappa + \la_2 -2 = -(\kappa + \la_1 -2) < 0$ and $\la_3=\la_2$.
\end{itemize}
Suppose that we are in case (a). Then
\[
\res{ 1, \la_2} = \res{ 2, \la_2} + 1 = \res{ \ttt^{-1}(n-1)} + 1 = \kappa + \la_1 -1 = \res{ \ttt^{-1}(n)} = \res{3, \la_3},
\]
which contradicts condition (iii). Suppose that we are in case (b). Then none of  $(1,\la_1)$, $(2,\la_2-1)$ or $(3,\la_3-1)$ has residue $\kappa + \la_1 -3$.\qedhere
\end{proof}

\begin{lem} \label{lem: T and G A}
Let $A=(r,c)$ be a Garnir node of $[\la]$. Then
there is no tableau $\ttt \in \rST \la $ such that
\[
\ttt \rhd \ttg^{A}  \ \text{ and } \ \res \ttt = \res{\ttg^{A}}.
\]
\end{lem}

\begin{proof}

We take $\ttt \in \rST \la$ such that
\[
\ttt \dom \ttg^{A}  \ \text{ and } \ \res \ttt = \res{\ttg^{A}}.
\]
Note that Lemma \ref{shpdom} implies that
\begin{align} \label{eq: seq1}
\ttt(1,k)=k \qquad \text{for $k=1,2, \ldots, c-1$}.
\end{align}
Thus, we may assume that $\la = (\la_1, c)$ and $r=1$.
As $\res \ttt = \res{\ttg^{A}}$, we have $\ttt(1, \la_1) = n$ by Lemma \ref{lem: row-strict}(1).

If $\la_1 = c$, then we have $\ttt = \ttg^{A} $ by $\eqref{eq: seq1}$.

If $\la_1 > c$, then we have
\[
\tabupto\ttt {n-1} \dom \tabupto{\ttg^{A}} {n-1} \ \text{ and } \ \res {\tabupto\ttt {n-1}} = \res{\tabupto{\ttg^{A}} {n-1}},
\]
which, by induction on $\la_1 - c$, implies that $\tabupto\ttt {n-1} = \tabupto{\ttg^{A}} {n-1}$. Therefore, we have $\ttt = \ttg^{A}$.
\end{proof}

\begin{lem} \label{lem: T and G AB}
Let $A=(r,c)$ and $B=(r',c')$ be Garnir nodes of $[\la]$ with $ c \le c'$.
Suppose that either $(r \ne r'+1)$ or $(r=r'+1$ and $\res B \ne \res{r+1,c} )$.
Then
there is no tableau $\ttt \in \rST \la$ such that
\[
\ttt \rhd \ttg^{A,B}  \ \text{ and } \ \res \ttt = \res{\ttg^{A,B}}.
\]
\end{lem}

\begin{proof}
Let $\ttt \in \rST \la$ such that
\[
\ttt \dom \ttg^{A,B}  \ \text{ and } \ \res \ttt = \res{\ttg^{A,B}},
\]
and $\belt^A$ and $\belt^B$ the Garnir belts corresponding to $A$ and $B$ respectively.
If $\belt^A \cap \belt^B = \emptyset$, then we may argue as in the proof of \cref{lem: T and G A} to see that $\ttt = \ttg^{A,B}$.
So we assume that $\belt^A \cap \belt^B \ne \emptyset$. Then, we have two cases -- either $r = r'$, or $r = r'+1$.

First suppose that $r = r'$. By \cref{shpdom}, we may assume that $\la = (\la_1, c')$, $c \ne c'$ and $r=r'=1$.
Since $\res \ttt = \res{\ttg^{A,B}}$, we have $\ttt(1, \la_1) = n$ by Lemma \ref{lem: row-strict}(1).

Suppose that $\la_1 = c'$ and consider the condition
\[
\tabupto\ttt {n-p} \dom \tabupto{\ttg^{A,B}} {n-p} \ \text{ and } \ \res {\tabupto\ttt {n-p}} = \res{\tabupto{\ttg^{A,B}} {n-p}},
\]
where $p = c'-c+1$. As $\tabupto{\ttg^{A,B}} {n-p}$ is a Garnir tableau with Garnir node $A$, we conclude that $\tabupto\ttt {n-p} = \tabupto{\ttg^{A,B}} {n-p}$ by \cref{lem: T and G A}. Since $\ttt\in \rST \la$ and $\ttt^{-1}(n)=(1,\la_1)$, it follows that the entries $n-p+1,\dots,n-1$ must appear in the nodes $(2,c+1), \dots, (2,c')$ respectively, and thus $\ttt = \ttg^{A,B}$.

If $\la_1 > c'$, then
\[
\tabupto\ttt {n-1} \dom \tabupto{\ttg^{A,B}} {n-1} \ \text{ and } \ \res {\tabupto\ttt {n-1}} = \res{\tabupto{\ttg^{A,B}} {n-1}},
\]
which, by induction on $\la_1 - c'$, implies that $\tabupto\ttt {n-1} = \tabupto{\ttg^{A,B}} {n-1}$. Thus we have $\ttt = \ttg^{A,B}$.

Next, suppose that $r = r'+1$. By \cref{shpdom}, we may assume that $\la = (\la_1, \la_2, c)$ and $r'=1$.
Note that $\res{1,c'} \ne \res{3,c}$ by our assumption and $\ttt(1,k) = k$ for $k=1,2,\ldots, c'-1$ by \cref{shpdom}. We now proceed by induction on $\la_2 - c'$.

First, suppose that $\la_2 = c'$. We must have $\ttt(1,\la_1) = n$, by \cref{lem: row-strict}(2).

If $\la_1 = c'$, then we define row-strict tableaux $\ttt'$ and $\ttg'$ of shape $(\la_2, c)$ by
\begin{align*}
\ttt'(i,j) = \ttt(i+1,j) - c' + 1, \qquad \ttg'(i,j) = \ttg^{A,B}(i+1,j) - c' + 1.
\end{align*}
Then $\ttg'$ becomes a Garnir tableau with Garnir node $A$ and
\[
\ttt' \dom \ttg'  \ \text{ and } \ \res {\ttt'} = \res{\ttg'},
\]
which implies that $\ttt' = \ttg'$ by Lemma \ref{lem: T and G A}. Thus we have $\ttt = \ttg^{A,B}$.

If $\la_1 > c'$, then
\[
\tabupto\ttt {n-1} \dom \tabupto{\ttg^{A,B}} {n-1} \ \text{ and } \ \res {\tabupto\ttt {n-1}} = \res{\tabupto{\ttg^{A,B}} {n-1}},
\]
which implies that $\tabupto\ttt {n-1} = \tabupto{\ttg^{A,B}} {n-1}$ by induction on $\la_1 - c'$. Thus we have $\ttt = \ttg^{A,B}$.

Now suppose that $\la_2 > c'$. 
Then $n$ is located in the node $(2,\la_2)$ in $\ttg^{A,B}$. Thus
$\ttt^{-1}(n)$ = $(2, \la_2)$ or $(3, c)$ since $\ttt \dom \ttg^{A,B}$. Suppose that $\ttt^{-1}(n) =(3, c)$.
Since $\res{\ttt}=\res{\ttg^{A,B}}$, we have $\res{3,c} = \res{2,\la_2}$ and therefore $\kappa + \la_2 -2 > 0$ and $\kappa + c - 3 < 0$.

If $\la_2 = c'+1$, then $\res{3,c} = \res{2,\la_2} = \res{1,c'}$ which is a contradiction.

If $\la_2 > c'+1$, then we have  $\res{\ttt^{-1}(n-1)} = \kappa + \la_2 - 3  = \res{\ttt^{-1}(n)} -1$.
Then $\kappa + c - 3 <0$ implies that $\ttt^{-1}(n-1)$ cannot be in the third row, so that $\ttt^{-1}(n-1)$ = $(1,\la_1)$ or $(2,\la_2)$.
But then $\res{ \ttt^{-1}(n-1)} \ne \kappa + \la_2 -3$, another contradiction.
Therefore we must have $\ttt^{-1}(n) = (2, \la_2)$.

Thus we have
\[
\tabupto\ttt {n-1} \dom \tabupto{\ttg^{A,B}} {n-1} \ \text{ and } \ \res {\tabupto\ttt {n-1}} = \res{\tabupto{\ttg^{A,B}} {n-1}},
\]
which implies, by induction on $\la_2 - c'$, that $\tabupto\ttt {n-1} = \tabupto{\ttg^{A,B}} {n-1}$.
We conclude that $\ttt = \ttg^{A,B}$.
\end{proof}

\subsection{A lemma for block braid relations}

\begin{lem} \label{Lem: G^AB 1}
Let $\nu=\nu^1*\nu^2*\nu^3$ where $a_1,a_3\ge1$ and
\begin{align*}
\nu^1&=(i+1,i+2,\dots,i+a_1),\\
\nu^2&=(i,i-1,\dots,1,0,1,\dots,i-1,i),\\
\nu^3&=(i+a_3,i+a_3-1,\dots,i+1),
\end{align*}
for some $i\ge0$. We set $a_2=2i+1$ and $\underline{a}=(1,a_1-1,a_2,a_3-1,1)$. Then
\[
\Psi_2\Psi_1\Psi_2(a_1,a_2,a_3)e(\nu)-\Psi_1\Psi_2\Psi_1(a_1,a_2,a_3)e(\nu)
\]
is given as follows. 
\begin{enumerate}
\item
Suppose $i\ne0$. Then it is equal to
\[
\Psi_1\Psi_4\Psi_2\Psi_3\Psi_2(\underline{a})(x_1+x_{a_1+1}+x_{a_1+a_2}+x_{a_1+a_2+a_3})e(\nu).
\]
\item
Suppose $i=0$. Then it is equal to
$\Psi_1\Psi_4\Psi_2\Psi_3\Psi_2(\underline{a})(x_1+x_{a_1+a_3+1})e(\nu)$.
\end{enumerate}
\end{lem}
\begin{proof}
Following Remark \ref{commutation formula2}, we compute
\begin{gather*}
\Psi_2\Psi_1\Psi_2(a_1,a_2,a_3)e(\nu) - \Psi_1\Psi_2\Psi_1(a_1,a_2,a_3)e(\nu)\\
\qquad =\sum_{k=1}^{a_2} \Psi_2(a_3+k,a_1,a_2-k)X_k\Psi[a_1,k-1]e(\nu),
\end{gather*}
where $X_k=\Psi_2(a_3+k-1,a_1,1)\Psi[a_1+a_2,a_3]-\Psi[a_1+a_2,a_3]\Psi_2(k-1,a_1,1)$.

Then, Lemma \ref{commutation formula1} tells that the term $\Psi_2(a_3+k,a_1,a_2-k)X_k\Psi[a_1,k-1]e(\nu)$ survives only if for some $1\leq s\leq a_1$ and $1\leq t\leq a_3$,
\begin{itemize}
\item
the $s$th entry of $\nu^1=(i+1,i+2,\dots,i+a_1)$,
\item
the $k$th entry of $\nu^2=(i,\dots,0\dots,i)$,
\item
the $t$th entry of $\nu^3=(i+a_3,\dots,i+1)$
\end{itemize}
form a triple of the form $(1,0,1)$, $(j,j+1,j)$ for $j\ge0$, or $(j,j-1,j)$ for $j\ge2$. Hence, either $(s,k,t)=(1,1,a_3)$ or
$(1,a_2,a_3)$ is possible. Thus, if $i\ne0$ we insert
\begin{align*}
X_1&= \Psi[2,a_3-1]\Psi_2(2,a_1+a_2-2,a_3)\Psi_2(1,a_1-1,1), \\
X_{a_2}&= \Psi[a_2-1,a_3]\Psi_2(a_2-1,2,a_3-1)\Psi_2(a_2+1,a_1-1,a_3)\Psi_2(a_2,a_1-1,1), 
\end{align*}
and $X_k=0$ for $k\ne 1, a_2$, to obtain
\begin{gather*}
\Psi_2\Psi_1\Psi_2(a_1,a_2,a_3)e(\nu) - \Psi_1\Psi_2\Psi_1(a_1,a_2,a_3)e(\nu)\\
\qquad\qquad=\left(\Psi_2(a_3+1,a_1,a_2-1)X_1+X_{a_2}\Psi[a_1,a_2-1]\right)e(\nu).
\end{gather*}
On the other hand, if $i=0$ we obtain
\begin{gather*}
\Psi_2\Psi_1\Psi_2(a_1,a_2,a_3)e(\nu) - \Psi_1\Psi_2\Psi_1(a_1,a_2,a_3)e(\nu)\qquad\qquad\\
\qquad\qquad\qquad=(x_{a_3}+x_{a_3+2})\Psi[2,a_3-1]\Psi_2(2,a_1-1,a_3)\Psi_2(1,a_1-1,1)e(\nu).
\end{gather*}

\bigskip
(1) Suppose that $i\ne0$. We write
\[
\nu^1=(i+1)*\nu^a,\; \nu^2=(i)*\nu^b*(i),\; \nu^3=\nu^c*(i+1)
\]
and let $\nu = (i+1)*\nu^a*(i)*\nu^b*(i)*\nu^c*(i+1)$. Then the first term is 
\begin{gather*}
\Psi_2(a_3+1,a_1,a_2-1) \Psi[2,a_3-1]\Psi_2(2,a_1+a_2-2,a_3)\Psi_2(1,a_1-1,1)e(\nu)\\
=(\Psi_5\Psi_4\Psi_6\Psi_5)(\Psi_1\Psi_2)(\Psi_4\Psi_3\Psi_5\Psi_4\Psi_6\Psi_5)(\Psi_2)(\underline{a})e(\nu),
\end{gather*}
where $\underline{a}= (1,a_1-1,1,a_2-2,1,a_3-1,1)$. Then, following the recipe in Remark \ref{commutation formula2}, we know that 
there is no error term in $\Psi_5\Psi_4\Psi_5(\underline{b})e(\mu)-\Psi_4\Psi_5\Psi_4(\underline{b})e(\mu)$, so that
\begin{align*}
=&(\Psi_1\Psi_2)(\Psi_5\Psi_4\Psi_6)\Psi_5\Psi_4\Psi_5(\underline{b})e(\mu)\Psi_3\Psi_4\Psi_6\Psi_5\Psi_2(\underline{a})\\
=&\Psi_1\Psi_2\Psi_5\Psi_4\Psi_6\Psi_4\Psi_5\Psi_4\Psi_3\Psi_4\Psi_6\Psi_5\Psi_2(\underline{a})e(\nu),
\end{align*}
where 
$\mu=(i+1)*(i)*\nu^c*\nu^a*\nu^b*(i+1)*(i)$ and $\underline{b}=(1,1,a_3-1,a_1-1,a_2-2,1,1)$. 
Then, \cref{square of block transposition}(1) implies
\begin{align*}
=&(\Psi_1\Psi_2\Psi_5)\Psi_4^2(\underline{b}')e(\mu')\Psi_6\Psi_5\Psi_4\Psi_3\Psi_4\Psi_6\Psi_5\Psi_2(\underline{a})\\
=&\Psi_1\Psi_2\Psi_5\Psi_6\Psi_5\Psi_4\Psi_3\Psi_4\Psi_6\Psi_5\Psi_2(\underline{a})e(\nu)
\end{align*}
where 
$\mu'=(i+1)*(i)*\nu^c*\nu^b*(i+1)*(i)*\nu^a$ and $\underline{b}'=(1,1,a_3-1,a_2-2,1,1,a_1-1)$. We continue with similar arguments:
\begin{align*}
=&(\Psi_1\Psi_2)\Psi_5\Psi_6\Psi_5(\underline{b}'')e(\mu'')\Psi_4\Psi_3\Psi_4\Psi_6\Psi_5\Psi_2(\underline{a})\\
=&\Psi_1\Psi_2\Psi_6\Psi_5\Psi_6\Psi_4\Psi_3\Psi_4\Psi_6\Psi_5\Psi_2(\underline{a})e(\nu)
\end{align*}
where 
$\mu''=(i+1)*(i)*\nu^c*\nu^b*\nu^a*(i+1)*(i)$ and $\underline{b}''=(1,1,a_3-1,a_2-2,a_1-1,1,1)$, 
\begin{align*}
=&(\Psi_1\Psi_2\Psi_6\Psi_5\Psi_4\Psi_3\Psi_4)\Psi_6^2(\underline{b}''')e(\mu''')\Psi_5\Psi_2(\underline{a})\\
=&\Psi_1\Psi_2\Psi_6\Psi_5\Psi_4\Psi_3\Psi_4(x_{a_1+a_2+a_3-1}+x_{a_1+a_2+a_3})\Psi_5\Psi_2(\underline{a})e(\nu)\\
=&\Psi_1\Psi_2\Psi_6\Psi_5\Psi_4\Psi_3\Psi_4\Psi_5\Psi_2(\underline{a})(x_{a_1+a_2}+x_{a_1+a_2+a_3})e(\nu)\\
=&(\Psi_1\Psi_2\Psi_6\Psi_5)\Psi_4\Psi_3\Psi_4(\underline{b}''')e(\mu''')\Psi_5\Psi_2(\underline{a})(x_{a_1+a_2}+x_{a_1+a_2+a_3})\\
=&\Psi_1\Psi_2\Psi_6\Psi_5\Psi_3\Psi_4\Psi_3\Psi_5\Psi_2(\underline{a})(x_{a_1+a_2}+x_{a_1+a_2+a_3})e(\nu)\\
=&\Psi_1\Psi_6\Psi_2\Psi_3\Psi_5\Psi_4\Psi_5\Psi_3\Psi_2(\underline{a})(x_{a_1+a_2}+x_{a_1+a_2+a_3})e(\nu),
\end{align*}
where 
$\mu'''=(i+1)*(i)*\nu^a*\nu^b*\nu^c*(i)*(i+1)$ and $\underline{b}'''=(1,1,a_1-1,a_2-2,a_3-1,1,1)$, and after one more step, we obtain
\[
\Psi_1\Psi_6\Psi_2\Psi_3\Psi_4\Psi_5\Psi_4\Psi_3\Psi_2(\underline{a})(x_{a_1+a_2}+x_{a_1+a_2+a_3})e(\nu).
\]
Then, we can check that this is equal to $\Psi_1\Psi_4\Psi_2\Psi_3\Psi_2(\underline{a})(x_{a_1+a_2}+x_{a_1+a_2+a_3})e(\nu)$ if we change 
$\underline{a}$ to $\underline{a}=(1,a_1-1,a_2,a_3-1,1)$.

Next, we consider the second term 
\begin{gather*}
\Psi[a_2-1,a_3]\Psi_2(a_2-1,2,a_3-1)\Psi_2(a_2+1,a_1-1,a_3)\Psi_2(a_2,a_1-1,1)\Psi[a_1,a_2-1]e(\nu)\\
=(\Psi_2\Psi_1\Psi_3\Psi_2)(\Psi_3\Psi_4)(\Psi_6\Psi_5)(\Psi_4)(\Psi_2\Psi_1\Psi_3\Psi_2)(\underline{a})e(\nu),
\end{gather*}
where $\underline{a}= (1,a_1-1,1,a_2-2,1,a_3-1,1)$. Then
\begin{align*}
=& (\Psi_2\Psi_1)\Psi_3\Psi_2\Psi_3(\underline{b})e(\mu)\Psi_4\Psi_6\Psi_5\Psi_4\Psi_2\Psi_1\Psi_3\Psi_2 (\underline{a})\\
=& \Psi_2\Psi_1\Psi_2\Psi_3\Psi_2\Psi_4\Psi_6\Psi_5\Psi_4\Psi_2\Psi_1\Psi_3\Psi_2 (\underline{a}) e(\nu),
\end{align*}
where 
$\mu=(i)*\nu^b*(i+1)*\nu^c*(i)*(i+1)*\nu^a$ and $\underline{b}=(1,a_2-2,1,a_3-1,1,1,a_1-1)$,
\begin{align*}
=& (\Psi_2\Psi_1\Psi_2\Psi_3\Psi_4\Psi_6\Psi_5\Psi_4)\Psi_2^2(\underline{b}')e(\mu')\Psi_1\Psi_3\Psi_2 (\underline{a})\\
=& \Psi_2\Psi_1\Psi_2\Psi_3\Psi_4\Psi_6\Psi_5\Psi_4\Psi_1\Psi_3\Psi_2 (\underline{a})e(\nu),
\end{align*}
where
$\mu'=(i)*(i+1)*\nu^b*\nu^a*(i)*\nu^c*(i+1)$ and $\underline{b}'=(1,1,a_2-2,a_1-1,1,a_3-1,1)$,
\begin{align*}
=& \Psi_2\Psi_1\Psi_2(\underline{b}'')e(\mu'')\Psi_3\Psi_4\Psi_6\Psi_5\Psi_4\Psi_1\Psi_3\Psi_2 (\underline{a})\\
=& \Psi_1\Psi_2\Psi_1\Psi_3\Psi_4\Psi_6\Psi_5\Psi_4\Psi_1\Psi_3\Psi_2 (\underline{a})e(\nu),
\end{align*}
where
$\mu''=(i)*(i+1)*\nu^c*\nu^b*(i)*(i+1)*\nu^a$ and $\underline{b}''=(1,1,a_3-1,a_2-2,1,1,a_1-1)$,
\begin{align*}
=& (\Psi_1\Psi_2\Psi_3\Psi_4\Psi_6\Psi_5\Psi_4)\Psi_1^2(\underline{b}''')e(\mu''')\Psi_3\Psi_2 (\underline{a})\\
=& \Psi_1\Psi_2\Psi_3\Psi_4\Psi_6\Psi_5\Psi_4(x_1+x_2)\Psi_3\Psi_2 (\underline{a})e(\nu)\\
=& \Psi_6\Psi_1\Psi_2\Psi_3\Psi_4\Psi_5\Psi_4\Psi_3\Psi_2 (\underline{a})(x_1+x_{a_1+1})e(\nu),
\end{align*}
where $\mu'''=(i+1)*(i)*\nu^b*\nu^a*(i)*\nu^c*(i+1)$ and $\underline{b}'''=(1,1,a_2-2,a_1-1,1,a_3-1,1)$.
Then this is equal to $\Psi_1\Psi_4\Psi_2\Psi_3\Psi_2 (\underline{a})(x_1+x_{a_1+1}) e(\nu)$ if we change $\underline{a}$ 
to $\underline{a}=(1,a_1-1,a_2,a_3-1,1)$. 

\medskip
(2) Suppose that $i=0$. We write $\nu^1=(1)*\nu^a,\; \nu^2=(0),\; \nu^3=\nu^c*(1)$ as before, and let
$\nu = (1)*\nu^a*(0)*\nu^c*(1)$ and $\underline{a} = (1,a_1 - 1, a_3 - 1, 1)$. Then 
\begin{multline*}
(x_{a_3}+x_{a_3+2})\Psi[2,a_3-1]\Psi_2(2,a_1-1,a_3)\Psi_2(1,a_1-1,1)e(\nu)\\
=(x_{a_3}+x_{a_3+2})\Psi_1\Psi_2\Psi_4\Psi_3\Psi_2(\underline{a})e(\nu)
=\Psi_1\Psi_4\Psi_2\Psi_3\Psi_2(\underline{a})(x_1+x_{a_1+a_3+1})e(\nu),
\end{multline*}
which is the desired result.
\end{proof}

\subsection{A special three row case}
To handle the case that the Garnir belt of $\ttg^{A,B}$ has three rows, we may assume that 
$\la = (\la_1, \la_2, \la_3)$ with $\la_3>0$ and that Garnir nodes are $A=(2,c)$, $B=(1,c')$ with $c\le c'$. 
In this subsection, we consider a special case, that is, we assume
\begin{itemize}
\item[(i)]
The first row of the Garnir belt has residues $\nu^1_r=(i+1,i+2,\dots, \la_1-c'+i+1)$ from left to right.
\item[(ii)]
The second row of the Garnir belt has residues $\nu^2_m=(i,i-1,\dots,1,0,1,\dots,i-1,i)$ from left to right.
\item[(iii)]
The third row of the Garnir belt has residues $\nu^3_l=(c+i,c+i-1,\dots,i+1)$ from left to right.
\end{itemize}
In particular, $\res A=i$ and $\res B=i+1$. We denote 
\begin{itemize}
\item[(i)]
the residues of the first row of $\lambda$ by $\nu^1_l*\nu^1_r$,
\item[(ii)]
the residues of the second row of $\lambda$ by $\nu^2_l*\nu^2_m*\nu^2_r$,
\item[(iii)]
the residues of the third row of $\lambda$ by $\nu^3_l*\nu^3_r$, 
\end{itemize}
respectively. Pictorially, if $c \ne 1$ and $c' \ne \la_2$ then 

\[
\begin{tikzpicture}
\Yboxdim{20pt} \scriptsize{
\tgyoung(0cm,0cm,;!\wh_7!\gr;<i+1>;<i+2>_4,!\wh;< >_2;<i+1>!\gr;i_3;i!\wh;<i+1>_3,!\gr;< >_2;<i+2>;<i+1>!\wh_4)}
\tiny{
\draw[gray,<->] (0.05,0.6) -- node[below=-1pt] {\color{black}$\nu^1_l$} (5.568,0.6);
\draw[gray,<->] (5.68,0.6) -- node[below=-1pt] {\color{black}$\nu^1_r$} (9.79,0.6);
\draw[gray,<->] (0.05,-0.1) -- node[below=-1pt] {\color{black}$\nu^2_l$} (2.75,-0.1);
\draw[gray,<->] (2.87,-0.1) -- node[below=-1pt] {\color{black}$\nu^2_m$} (6.27,-0.1);
\draw[gray,<->] (6.38,-0.1) -- node[below=-1pt] {\color{black}$\nu^2_r$} (9.08,-0.1);
\draw[gray,<->] (0.05,-0.8) -- node[below=-1pt] {\color{black}$\nu^3_l$} (3.46,-0.8);
\draw[gray,<->] (3.57,-0.8) -- node[below=-1pt] {\color{black}$\nu^3_r$} (6.27,-0.8);
\draw[black] (5.75,0.1) node{$\mathbf{B}$};
\draw[black] (2.94,-0.55) node{$\mathbf{A}$};
}
\end{tikzpicture}
\]

\bigskip

Recall from \cref{lem: wA wB} that there are fully commutative elements $w^A$ and $w^B$ such that
\begin{align*}
&\ttg^{A,B} = w^A \ttg^A = w^B \ttg^B, \\
&\ell(\ttg^{A,B}) = \ell(w^A ) + \ell(\ttg^A) =  \ell(w^B ) + \ell(\ttg^B).
\end{align*}

We will show that $\psi_{w^A} g_A^\lambda \equiv \psi_{w^B} g_B^\lambda \pmod{ \Ga^{\lambda}_{ < \ell(\ttg^{A,B})}}$ in \cref{lem: wAgA-wBgB 4}. We build up to this with several smaller lemmas, as the calculation is quite lengthy.

We denote the length of $\nu^1_l,\nu^1_r,\nu^2_l,\nu^2_m,\nu^2_r,\nu^3_l,\nu^3_r$ by 
$a^1_l,a^1_r,a^2_l,a^2_m,a^2_r,a^3_l,a^3_r$, respectively, and define
\begin{align*}
\underline{a}&=(a^1_l,a^1_r,a^2_l,a^2_m,a^2_r,a^3_l,a^3_r),\\
\underline{a}'&=(a^1_l,a^1_r,a^2_l,a^3_l,a^2_m,a^2_r,a^3_r),\\
\underline{a}''&=(a^1_l,a^2_l,a^2_m,a^1_r,a^2_r,a^3_l,a^3_r),\\
\underline{a}'''&=(a^1_l,a^2_l,a^1_r,a^2_m,a^3_l,a^2_r,a^3_r).
\end{align*}

\begin{lem}\label{lem: wAgA-wBgB 1}
We have
\[
\psi_{w^A} g_A^\lambda - \psi_{w^B} g_B^\lambda = \Psi_3\Psi_6\Psi_4\Psi_5\Psi_4(\underline{b}')(x_{a^1_l+a^2_l+1}+x_{a^1_l+a^1_r+a^2_l+a^2_m+a^3_l})\Psi_7\Psi_6\Psi_2\Psi_3(\underline{b})m^\la
\]
where
\begin{align*}
\underline{b}'&=(a^1_l,a^2_l,1,a^1_r-1,a^2_m,a^3_l-1,1,a^2_r,a^3_r),\\
\underline{b}&=(a^1_l,1,a^1_r-1,a^2_l,a^2_m,a^2_r,a^3_l-1,1,a^3_r). 
\end{align*}
\end{lem}

\begin{proof}
Note that $a^1_r\ge1$, $a^2_m\ge1$, $a^3_l\ge1$ by definition. 
Then, by considering the two-line notation for $w^{\ttg^A}$ and $w^{\ttg^B}$ as in Lemma \ref{two line notation for block permutation}, we know that
\[
w^{\ttg^A}=S_4S_5(\underline{a}) \quad\text{and}\quad w^{\ttg^B}=S_3S_2(\underline{a}).
\]
Similarly, we have $w^A=S_4S_3S_2(\underline{a}')$ and $w^B=S_3S_4S_5(\underline{a}'')$. Thus, 
\begin{align*} \label{eq: G^AB 1}
\psi_{w^A} g_A^\lambda - \psi_{w^B} g_B^\lambda
&= \Psi_4\Psi_3\Psi_2\Psi_4\Psi_5(\underline{a})m^\lambda - \Psi_3\Psi_4\Psi_5\Psi_3\Psi_2(\underline{a})m^\lambda\\
&= \left(\Psi_4\Psi_3\Psi_4(\underline{a}''')-\Psi_3\Psi_4\Psi_3(\underline{a}''')\right)e(\mu)\Psi_2\Psi_5(\underline{a})m^\lambda,
\end{align*}
where $\mu=\nu^1_l*\nu^2_l*\nu^1_r*\nu^2_m*\nu^3_l*\nu^2_r*\nu^3_r$. We apply \cref{Lem: G^AB 1} to compute
\[
\Psi_2\Psi_1\Psi_2(a^1_r,a^2_m,a^3_l)e(\nu^1_r*\nu^2_m*\nu^3_l) - \Psi_1\Psi_2\Psi_1(a^1_r,a^2_m,a^3_l)e(\nu^1_r*\nu^2_m*\nu^3_l).
\]
Then the result is as follows.
\begin{itemize}
\item[(i)]
If $i\ne0$ then
\[
\psi_{w^A} g_A^\lambda - \psi_{w^B} g_B^\lambda = \Psi_3\Psi_6\Psi_4\Psi_5\Psi_4(\underline{b}')(X_1+X_2)\Psi_7\Psi_6\Psi_2\Psi_3(\underline{b})m^\la,
\]
where $X_1 = x_{a^1_l+a^2_l+1} + x_{a^1_l+a^1_r+a^2_l+a^2_m+a^3_l}$ and $X_2 = x_{a^1_l+a^1_r+a^2_l+1} + x_{a^1_l+a^1_r+a^2_l+a^2_m}$.
\item[(ii)]
If $i=0$ then
\[
\psi_{w^A} g_A^\lambda - \psi_{w^B} g_B^\lambda = \Psi_3\Psi_6\Psi_4\Psi_5\Psi_4(\underline{b}')X_1\Psi_7\Psi_6\Psi_2\Psi_3(\underline{b})m^\la,
\]
where $X_1=x_{a^1_l+a^2_l+1}+x_{a^1_l+a^1_r+a^2_l+a^3_l+1}$.
\end{itemize}

\medskip
As $\Psi_7\Psi_6\Psi_2\Psi_3(\underline{b})$ does not touch the fifth block of $\underline{b}$, if $i\ne0$ then
\[
X_2\Psi_7\Psi_6\Psi_2\Psi_3(\underline{b})m^\la=\Psi_7\Psi_6\Psi_2\Psi_3(\underline{b})X_2m^\la=0.\qedhere
\]
\end{proof}

\begin{lem}\label{lem: wAgA-wBgB 2}
Let
\begin{align*}
\underline{b}&= (a^1_l,1,a^1_r-1,a^2_l,a^2_m,a^2_r,a^3_l-1,1,a^3_r),\\
\underline{b}'&= (a^1_l,a^2_l,1,a^1_r-1,a^2_m,a^3_l-1,1,a^2_r,a^3_r),\\
\underline{c}&= (a^1_l,1,a^1_r-1,a^2_l-1,1,a^2_m,a^2_r,a^3_l-1,1,a^3_r),\\
\underline{d}&= (a^1_l, 1, a^1_r-1, a^2_l, a^2_m, 1, a^2_r-1, a^3_l-1, 1, a^3_r).
\end{align*}
\begin{enumerate}
\item
If $a^2_l =0$, then
\[
\Psi_3 \Psi_6 \Psi_4 \Psi_5 \Psi_4 (\underline{b}') x_{a^1_l+a^2_l+1} \Psi_7 \Psi_6 \Psi_2 \Psi_3 (\underline{b}) m^\la = 0.
\]
Otherwise,
\begin{multline*}
\Psi_3 \Psi_6 \Psi_4 \Psi_5 \Psi_4 (\underline{b}') x_{a^1_l+a^2_l+1} \Psi_7 \Psi_6 \Psi_2 \Psi_3 (\underline{b}) m^\la\\
= (\Psi_7 \Psi_8 \Psi_6 \Psi_2) \Psi_5 \Psi_4 \Psi_3 \Psi_5 \Psi_6 \Psi_7(\underline{c})m^\la
+(\Psi_7\Psi_8\Psi_6\Psi_2)(\Psi_4\Psi_5\Psi_4-\Psi_5\Psi_4\Psi_5)\Psi_3\Psi_6\Psi_7(\underline{c})m^\la.
\end{multline*}
\item
If $a^2_r = 0$, then
\[
\Psi_3 \Psi_6 \Psi_4 \Psi_5 \Psi_4 (\underline{b}') x_{a^1_l+a^1_r+a^2_l+a^2_m +a^3_l} \Psi_7 \Psi_6 \Psi_2 \Psi_3 (\underline{b}) m^\la = 0.
\]
Otherwise,
\begin{multline*}
\Psi_3 \Psi_6 \Psi_4 \Psi_5 \Psi_4 (\underline{b}') x_{a^1_l+a^1_r+a^2_l+a^2_m +a^3_l} \Psi_7 \Psi_6 \Psi_2 \Psi_3 (\underline{b}) m^\la\\
= -(\Psi_3 \Psi_2 \Psi_8 \Psi_4) \Psi_5 \Psi_6 \Psi_7 \Psi_5 \Psi_4 \Psi_3 (\underline{d}) m^\la 
- (\Psi_3 \Psi_2 \Psi_8 \Psi_4) (\Psi_6 \Psi_5 \Psi_6 -\Psi_5 \Psi_6 \Psi_5) \Psi_7 \Psi_4 \Psi_3 (\underline{d}) m^\la.
\end{multline*}
\end{enumerate}
\end{lem}

\begin{proof}
First, note that $\Psi_7\Psi_6\Psi_2\Psi_3(\underline{b})$ is the product of
\begin{align*}
\Psi_7(a^1_l,a^2_l,1,a^1_r-1,a^2_m,a^3_l-1,\underline{a^2_r,1},a^3_r)e(\nu^1_l*\nu^2_l*(i+1)*\dot{\nu}^1_r*\nu^2_m*\dot{\nu}^3_l*\underline{\nu^2_r*(i+1)}*\nu^3_r)\\
\Psi_6(a^1_l,a^2_l,1,a^1_r-1,a^2_m,\underline{a^2_r,a^3_l-1},1,a^3_r)e(\nu^1_l*\nu^2_l*(i+1)*\dot{\nu}^1_r*\nu^2_m*\underline{\nu^2_r*\dot{\nu}^3_l}*(i+1)*\nu^3_r)\\
\Psi_2(a^1_l,\underline{1,a^2_l},a^1_r-1,a^2_m,a^2_r,a^3_l-1,1,a^3_r)e(\nu^1_l*\underline{(i+1)*\nu^2_l}*\dot{\nu}^1_r*\nu^2_m*\nu^2_r*\dot{\nu}^3_l*(i+1)*\nu^3_r)\\
\Psi_3(a^1_l,1,\underline{a^1_r-1,a^2_l},a^2_m,a^2_r,a^3_l-1,1,a^3_r)e(\nu^1_l*(i+1)*\underline{\dot{\nu}^1_r*\nu^2_l}*\nu^2_m*\nu^2_r*\dot{\nu}^3_l*(i+1)*\nu^3_r)
\end{align*}
in this order, where $\nu^1_r=(i+1)*\dot{\nu}^1_r$, $\nu^3_l=\dot{\nu}^3_l*(i+1)$. 

\begin{enumerate}
\item
Moving $x_{a^1_l+a^2_l+1}$ to the right, 
\[
x_{a^1_l+a^2_l+1}\Psi_7\Psi_6\Psi_2\Psi_3(\underline{b})m^\la=\Psi_7\Psi_6(x_{a^1_l+a^2_l+1}\Psi_2-\Psi_2x_{a^1_l+1})\Psi_3(\underline{b})m^\la.
\]
We apply \cref{square of block transposition}(3) to $(x_{a^1_l+a^2_l+1}\Psi_2-\Psi_2x_{a^1_l+1})e(\mu)$, where
\[
\mu = \nu^1_l*(i+1)*\nu^2_l*\dot{\nu}^1_r*\nu^2_m*\nu^2_r*\dot{\nu}^3_l*(i+1)*\nu^3_r.
\]
If $a^2_l=0$ then $\Psi_2$ and $\Psi_3$ are the identity and 
\[
x_{a^1_l+a^2_l+1}\Psi_7\Psi_6\Psi_2\Psi_3(\underline{b})m^\la=x_{a^1_l+a^2_l+1}\Psi_7\Psi_6(\underline{b})m^\la=\Psi_7\Psi_6(\underline{b})x_{a^1_l+a^2_l+1}m^\la=0.
\]
Now suppose that $a^2_l\geq 1$. Note that $a^3_l=a^2_l+1\ge2$.
We write
\[
\nu=\nu^1_l*(i+1)*\dot{\nu}^1_r*\dot{\nu}^2_l*(i+1)*\nu^2_m*\nu^2_r*\dot{\nu}^3_l*(i+1)*\nu^3_r.
\]
Then, after applying Lemma \ref{square of block transposition}(3), we obtain
\[
\Psi_3\Psi_6\Psi_4\Psi_5\Psi_4(\underline{b}')x_{a^1_l+a^2_l+1}\Psi_7\Psi_6\Psi_2\Psi_3(\underline{b})m^\la
= \Psi_4\Psi_7\Psi_5\Psi_6 \Psi_5\Psi_8\Psi_7\Psi_2 \Psi_4\Psi_3(\underline{c}) m^\la.
\]
Then we can continue as follows. 
\begin{align*}
=& (\Psi_4\Psi_7)\Psi_8(\Psi_5\Psi_6\Psi_5)(\Psi_2\Psi_4\Psi_3)\Psi_7(\underline{c})m^\la\\
=&  (\Psi_4\Psi_7\Psi_8)\Psi_5\Psi_6\Psi_5(\underline{c}')e(\nu')\Psi_2\Psi_4\Psi_3\Psi_7(\underline{c})m^\la,
\end{align*}
where
$\nu'=\nu^1_l*\dot{\nu}^2_l*(i+1)*(i+1)*\dot{\nu}^1_r*\nu^2_m*\dot{\nu}^3_l*\nu^2_r*(i+1)*\nu^3_r$ and 
$
\underline{c}'=(a^1_l,a^2_l-1,1,1,a^1_r-1,a^2_m,a^3_l-1,a^2_r,1,a^3_r),
$
so that $\Psi_5\Psi_6\Psi_5(\underline{c}')e(\nu') = \Psi_6 \Psi_5 \Psi_6 (\underline{c}') e(\nu')$ and we have
\begin{align*}
\Psi_4 \Psi_7 \Psi_8 \Psi_5 \Psi_6 \Psi_5 \Psi_2 \Psi_4 \Psi_3 \Psi_7 (\underline{c})m^\la
=&  \Psi_4\Psi_7\Psi_8\Psi_6\Psi_5\Psi_6\Psi_2\Psi_4\Psi_3\Psi_7(\underline{c})m^\la \\
=& \Psi_7 \Psi_8 \Psi_6 \Psi_2 (\Psi_4 \Psi_5 \Psi_4) \Psi_3 \Psi_6 \Psi_7 (\underline{c})m^\la.
\end{align*}
Hence, $\Psi_4 \Psi_7 \Psi_8 \Psi_5 \Psi_6 \Psi_5 \Psi_2 \Psi_4 \Psi_3 \Psi_7 (\underline{c})m^\la$ is equal to
\[
(\Psi_7 \Psi_8 \Psi_6 \Psi_2) \Psi_5 \Psi_4 \Psi_3 \Psi_5 \Psi_6 \Psi_7 (\underline{c}) m^\la
+ (\Psi_7 \Psi_8 \Psi_6 \Psi_2) (\Psi_4 \Psi_5 \Psi_4 -\Psi_5 \Psi_4 \Psi_5) \Psi_3 \Psi_6 \Psi_7 (\underline{c}) m^\la.
\]
\item
Similarly, if we move $x_{a^1_l+a^1_r+a^2_l+a^2_m+a^3_l}$ to the right,
\begin{multline*}
\qquad x_{a^1_l+a^1_r+a^2_l+a^2_m+a^3_l}\Psi_7\Psi_6\Psi_2\Psi_3(\underline{b})m^\la \\
=\Psi_2\Psi_3(x_{a^1_l+a^1_r+a^2_l+a^2_m+a^3_l}\Psi_7-\Psi_7x_{a^1_l+a^1_r+a^2_l+a^2_m+a^2_r+a^3_l})\Psi_6(\underline{b})m^\la, \qquad\qquad
\end{multline*}
and we apply Lemma \ref{square of block transposition}(4) to $(x_{a^1_l+a^1_r+a^2_l+a^2_m+a^3_l}\Psi_7-\Psi_7x_{a^1_l+a^1_r+a^2_l+a^2_m+a^2_r+a^3_l})e(\mu)$, where
\[
\mu=\nu^1_l*(i+1)*\dot{\nu}^1_r*\nu^2_l*\nu^2_m*\dot{\nu}^3_l*\nu^2_r*(i+1)*\nu^3_r.
\]
If $a^2_r=0$ then $\Psi_6$ and $\Psi_7$ are the identity and 
\[
x_{a^1_l+a^1_r+a^2_l+a^2_m+a^3_l}\Psi_7\Psi_6\Psi_2\Psi_3(\underline{b})m^\la=\Psi_2\Psi_3(\underline{b})x_{a^1_l+a^1_r+a^2_l+a^2_m+a^3_l}m^\la=0.
\]
Now suppose that $a^2_r\geq 1$.
We write $\nu=\nu^1_l*(i+1)*\dot{\nu}^1_r*\nu^2_l*\nu^2_m*(i+1)*\dot{\nu}^2_r*\dot{\nu}^3_l*(i+1)*\nu^3_r$.
Then, after applying Lemma \ref{square of block transposition}(4), we obtain
\begin{align*}
\Psi_3 \Psi_6\Psi_4 \Psi_5 &\Psi_4 (\underline{b}') x_{a^1_l+a^1_r+a^2_l+a^2_m+a^3_l} \Psi_7 \Psi_6 \Psi_2 \Psi_3 (\underline{b}) m^\la\\
&=-\Psi_3\Psi_6\Psi_4\Psi_5\Psi_4\Psi_8\Psi_6\Psi_7\Psi_2\Psi_3(\underline{d})m^\la\\
&= - (\Psi_3 \Psi_2 \Psi_8 \Psi_4) (\Psi_6 \Psi_5 \Psi_6) \Psi_7 \Psi_4 \Psi_3 (\underline{d})m^\la.
\end{align*}

Hence,
$- (\Psi_3 \Psi_2 \Psi_8 \Psi_4) (\Psi_6 \Psi_5 \Psi_6) \Psi_7 \Psi_4 \Psi_3 (\underline{d})m^\la$ is equal to
\[
-(\Psi_3\Psi_2\Psi_8\Psi_4) \Psi_5\Psi_6\Psi_7\Psi_5\Psi_4\Psi_3(\underline{d})m^\la 
- (\Psi_3\Psi_2\Psi_8\Psi_4) (\Psi_6\Psi_5\Psi_6-\Psi_5\Psi_6\Psi_5) \Psi_7\Psi_4\Psi_3(\underline{d})m^\la.\qedhere
\]
\end{enumerate}
\end{proof}

\begin{lem}\label{lem: reduced exp of length 10}
Let $\underline{c}$ and $\underline{d}$ be as in \cref{lem: wAgA-wBgB 2}. Then
\[
S_7 S_8 S_6 S_2 S_5 S_4 S_3 S_5 S_6 S_7 (\underline{c}) \quad \text{and} \quad
S_3 S_2 S_8 S_4 S_5 S_6 S_7 S_5 S_4 S_3 (\underline{d})
\]
are reduced.
\end{lem}

\begin{proof}
By \cref{cor: reduced S_i}, it suffices to check that they have length 10 if $\underline{c} = \underline{d} = (1,\dots,1)$, which is easy to check.
\end{proof}

\begin{lem} \label{lem: wAgA-wBgB 3}
Let $A'=(2,c-1)$, $B'=(1,c'+1)$. Then we have the following.
\begin{enumerate}
\item
Let $a^2_l\geq 1$ and $a^2_r =0$. Then
\[
\Psi_3 \Psi_6 \Psi_4 \Psi_5\Psi_4(\underline{b}') x_{a^1_l+a^2_l+1} \Psi_7\Psi_6\Psi_2\Psi_3(\underline{b}) m^\la = \psi_w g^\la_{A'}
\]
for some reduced expression of $w \in \sg_n$ with
$\ell(w \ttg^{A'}) = \ell(w) + \ell(\ttg^{A'}) < \ell(\ttg^{A,B})$.
\item
Let $a^2_r\geq 1$ and $a^2_l=0$. Then
\[
\Psi_3 \Psi_6 \Psi_4 \Psi_5\Psi_4(\underline{b}') x_{a^1_l + a^1_r + a^2_l + a^2_m + a^3_l} \Psi_7\Psi_6\Psi_2\Psi_3(\underline{b}) m^\la = -\psi_w g^\la_{B'}
\]
for some reduced expression of $w \in \sg_n$ with
$\ell(w \ttg^{B'}) = \ell(w) + \ell(\ttg^{B'}) < \ell(\ttg^{A,B})$.
\end{enumerate}
\end{lem}

\begin{proof}
The result follows readily if we can show that the second term in each of \cref{lem: wAgA-wBgB 2}(1) and (2) are zero under the corresponding conditions, since $g^\la_{A'} = \Psi_5 \Psi_6 \Psi_7 (\underline{c})m^\la$ and $g^\la_{B'} = \Psi_5 \Psi_4\Psi_3(\underline{d}) m^\la$. (We note that under the corresponding hypotheses, $\Psi_7 (\underline{c})$ in (1) 
and $\Psi_3(\underline{d})$ in (2) are the identity.)
\begin{enumerate}
\item
For the second term from \cref{lem: wAgA-wBgB 2}(1), 
it suffices to consider
\[
(\Psi_4\Psi_5\Psi_4-\Psi_5\Psi_4\Psi_5)\Psi_3\Psi_6\Psi_7(\underline{c})m^\la
=(\Psi_4\Psi_5\Psi_4-\Psi_5\Psi_4\Psi_5)(\underline{c}')e(\nu')\Psi_3\Psi_6\Psi_7(\underline{c})m^\la,
\]
where $\nu' = \nu^1_l*(i+1)*\dot{\nu}^2_l*\dot{\nu}^1_r*(i+1)*\dot{\nu}^3_l*\nu^2_m*\nu^2_r*(i+1)*\nu^3_r$ and 
\[
\underline{c}'=(a^1_l,1,a^2_l-1,a^1_r-1,1,a^3_l-1,a^2_m,a^2_r,1,a^3_r).
\]
For this, we need to compute
\[
(\Psi_2\Psi_1\Psi_2-\Psi_1\Psi_2\Psi_1)(a^1_r-1,1,a^3_l-1)e(\dot{\nu}^1_r*(i+1)*\dot{\nu}^3_l).
\]
When $a^1_r=1$, this error term is zero, and we are done. So we may assume that $a^1_r\geq2$ and continue as follows.
The above expression is equal to 
\[
\Psi[2,a^3_l-2]\Psi_2(2,a^1_r-2,a^3_l-1)\Psi_2(1,a^1_r-2,1)e(\dot{\nu}^1_r*(i+1)*\dot{\nu}^3_l)
\]
by Lemma \ref{commutation formula1}, but we need a different formula here: first, we apply Lemma \ref{commutation formula1} to obtain
\begin{multline*}
(\Psi_2\Psi_1\Psi_2-\Psi_1\Psi_2\Psi_1)(a^3_l-1,1,a^1_r-1)e(\dot{\nu}^3_l*(i+1)*\dot{\nu}^1_r)\\
=\Psi_2(a^1_r-1,a^3_l-2,1)\Psi[a^3_l-2,a^1_r-1]\Psi_2(a^3_l-1,2,a^1_r-2)e(\dot{\nu}^3_l*(i+1)*\dot{\nu}^1_r).
\end{multline*}
Then, we apply the anti-involution of $R(\beta)$ to obtain
\begin{multline*}
(\Psi_2\Psi_1\Psi_2-\Psi_1\Psi_2\Psi_1)(a^1_r-1,1,a^3_l-1)e(\dot{\nu}^1_r*(i+1)*\dot{\nu}^3_l)\\
=\Psi_2(a^3_l-1,a^1_r-2,2)\Psi[a^1_r-1,a^3_l-2]\Psi_2(a^1_r-1,1,a^3_l-2)e(\dot{\nu}^1_r*(i+1)*\dot{\nu}^3_l).
\end{multline*}
We use this formula to compute the second term.
To state the result in this case, we change $\underline{c}$ to
\[
\underline{c} = (a^1_l, 1, 1, a^1_r-2, a^2_l-1, 1, a^2_m, a^3_l-2, 1, 1, a^3_r).
\]
The second term is then
\[
- (\Psi_8 \Psi_9) (\Psi_7 \Psi_8) (\Psi_2) [(\Psi_7 \Psi_6) (\Psi_4 \Psi_5)(\Psi_6)] (\Psi_8 \Psi_7) (\Psi_3 \Psi_4) (\underline{c}) m^\la.
\]
We continue as follows.
\begin{align*}
&= - \Psi_2 \Psi_8 \Psi_9 (\Psi_7 \Psi_8 \Psi_7 (\underline{c}') e(\nu')) \Psi_6 \Psi_4 \Psi_5 \Psi_6 \Psi_3 \Psi_4 \Psi_8 \Psi_7 (\underline{c}) m^\la\\
&= - \Psi_2 \Psi_8 \Psi_9 \Psi_8 \Psi_7 \Psi_6 \Psi_4 \Psi_5 \Psi_6 \Psi_3 \Psi_4 (\Psi_8^2)  \Psi_7 (\underline{c}) m^\la\\
&= - \Psi_2 (\Psi_8 \Psi_9 \Psi_8 (\underline{c}'') e(\nu'')) \Psi_7 \Psi_6 \Psi_4 \Psi_5 \Psi_6 \Psi_3 \Psi_4  \Psi_7 (\underline{c}) m^\la\\
&= - \Psi_2 \Psi_9 \Psi_8 \Psi_7 \Psi_6 \Psi_4 \Psi_5 \Psi_6 \Psi_3 \Psi_4  \Psi_7 \Psi_9 (\underline{c}) m^\la\\
&= 0  \text{ (because $\Psi_9 (\underline{c}) = \psi_{a^1_l + a^1_r + a^2_l + a^2_m + a^3_l - 1}$ annihilates $m^\la$, as $a^3_l = a^2_l +1\geq 2$),}
\end{align*}
where the third equality above follows from \cref{square of block transposition}(1) and
\begin{align*}
\underline{c}' &= (a^1_l, 1, a^2_l -1, a^3_l - 2, 1, 1, a^1_r -2, 1, a^2_m, 1, a^3_r),\\
\nu' &= \nu^1_l*(i+1)* \dot{\nu}^2_l* \ddot{\nu}^3_l* (i+2)* (i+1)* \ddot{\nu}^1_r* (i+2)* \nu^2_m* (i+1)* \nu^3_r,\\
\underline{c}'' &=  (a^1_l, 1, a^2_l -1, a^3_l - 2, 1, 1, a^2_m, a^1_r -2, 1, 1, a^3_r),\\
\nu'' &= \nu^1_l*(i+1)* \dot{\nu}^2_l* \ddot{\nu}^3_l* (i+2)* (i+1)* \nu^2_m* \ddot{\nu}^1_r* (i+2)* (i+1)* \nu^3_r.
\end{align*}
Thus, if we write
\[
\psi_w = \Psi_7 \Psi_8 \Psi_6 \Psi_2 \Psi_5 \Psi_4 \Psi_3 (a^1_l, 1, a^1_r-1, a^2_l-1, a^3_l-1, 1, a^2_m, a^2_r, 1, a^3_r)
\]
by abuse of notation, we have the result.
\item
For the second term from \cref{lem: wAgA-wBgB 2}(2), 
it suffices to consider
\[
(\Psi_6 \Psi_5\Psi_6 - \Psi_5 \Psi_6 \Psi_5) \Psi_7 \Psi_4 \Psi_3 (\underline{d}) m^\la
= (\Psi_6 \Psi_5\Psi_6- \Psi_5\Psi_6 \Psi_5) (\underline{d}') e(\nu')\Psi_7 \Psi_4\Psi_3 (\underline{d}) m^\la,
\]
where
$\nu'=\nu^1_l*(i+1)*\nu^2_l*\nu^2_m*\dot{\nu}^1_r*(i+1)*\dot{\nu}^3_l*\dot{\nu}^2_r*(i+1)*\nu^3_r$ and 
\[
\underline{d}' = (a^1_l, 1, a^2_l, a^2_m, a^1_r-1, 1, a^3_l-1, a^2_r-1, 1, a^3_r).
\]
For this, we need to compute
\[
(\Psi_2\Psi_1\Psi_2-\Psi_1\Psi_2\Psi_1)(a^1_r-1,1,a^3_l-1)e(\dot{\nu}^1_r*(i+1)*\dot{\nu}^3_l)
\]
again. But, our assumption that $a^2_l=0$ is equivalent to $a^3_l=1$, which implies that this error term is zero. Thus if we write
\[
\psi_w = \Psi_3 \Psi_2 \Psi_8 \Psi_4 \Psi_5 \Psi_6 \Psi_7 (a^1_l, 1, a^2_l, a^2_m, 1, a^1_r-1, a^2_r-1, a^3_l-1, 1, a^3_r)
\]
by abuse of notation, the result follows.
\qedhere
\end{enumerate}
\end{proof}

\begin{lem} \label{lem: wAgA-wBgB 4}
We have
\[
\psi_{w^A} g_A^\lambda \equiv \psi_{w^B} g_B^\lambda \quad \pmod{\Ga^{\lambda}_{< \ell(\ttg^{A,B})}}.
\]
\end{lem}

\begin{proof}
\cref{lem: wAgA-wBgB 2} implies that
\[
\psi_{w^A} g_A^\lambda - \psi_{w^B} g_B^\lambda = \psi_w g_{A'}^\lambda - \psi_{w'} g_{B'}^\lambda + \gamma + \delta,
\]
for some reduced expressions of $w, w' \in \sg_n$ such that $\ell(w)+\ell(\ttg^{A'})< \ell(\ttg^{A,B})$ and $\ell(w')+\ell(\ttg^{B'})< \ell(\ttg^{A,B})$.
Here, $\gamma$ and $\delta$ are the second terms appearing in \cref{lem: wAgA-wBgB 2}(1) and (2) respectively, and are both zero by \cref{lem: wAgA-wBgB 3} unless $a^2_l\ge1$ and $a^2_r\ge 1$.
\cref{lem: reduced exp of length 10} implies that $\psi_w g_{A'}^\la$ and $\psi_{w'} g_{B'}^\la$ belong to $\Ga^{\la}_{< \ell(\ttg^{A,B})}$;
we may also have that $\psi_w g_{A'}^\la$ or $\psi_{w'} g_{B'}^\la$ are zero, in the degenerate cases that $a^2_l=0$ or $a^2_r=0$, respectively, by \cref{lem: wAgA-wBgB 2}.

So it remains to show that $\gamma + \delta \equiv 0 \pmod{\Ga^{\lambda}_{< \ell(\ttg^{A,B})}}$ when $a^2_l\ge 1$ and $a^2_r\ge 1$. In fact, we will show that $\gamma + \delta=0$. We continue by further calculation with $\gamma$ and $\delta$.

As in \cref{lem: wAgA-wBgB 2}(1), we compute $\gamma$ by using
\begin{multline*}
(\Psi_2\Psi_1\Psi_2-\Psi_1\Psi_2\Psi_1)(a^1_r-1,1,a^3_l-1)e(\dot{\nu}^1_r*(i+1)*\dot{\nu}^3_l)\\
=\Psi_2(a^3_l-1,a^1_r-2,2)\Psi[a^1_r-1,a^3_l-2]\Psi_2(a^1_r-1,1,a^3_l-2)e(\dot{\nu}^1_r*(i+1)*\dot{\nu}^3_l).
\end{multline*}
To state the result, we change $\underline{c}$ to
\[
\underline{c}=(a^1_l,1,1,a^1_r-2,a^2_l-1,1,a^2_m,1,a^2_r-1,a^3_l-2,1,1,a^3_r),
\]
and write $\nu=\nu^1_l*(i+1)*(i+2)*\ddot{\nu}^1_r*\dot{\nu}^2_l*(i+1)*\nu^2_m*(i+1)*\dot{\nu}^2_r*\ddot{\nu}^3_l*(i+2)*(i+1)*\nu^3_r$.
Then
\begin{align*}
\gamma = &-(\Psi_8\Psi_9)(\Psi_{10}\Psi_{11})(\Psi_7\Psi_8)(\Psi_2)(\Psi_7\Psi_6\Psi_4\Psi_5\Psi_6)(\Psi_3\Psi_4)(\Psi_8\Psi_7)\Psi_9\Psi_{10}\Psi_8\Psi_9(\underline{c})m^\la\\
= &-\Psi_8 \Psi_9 \Psi_{10} \Psi_{11} \Psi_2 (\Psi_7 \Psi_8 \Psi_7 (\underline{c}')e(\nu')) \Psi_6 \Psi_4 \Psi_5 \Psi_6 \Psi_3 \Psi_4 \Psi_8 \Psi_7 \Psi_9 \Psi_{10} \Psi_8 \Psi_9 (\underline{c})m^\la\\
= &-\Psi_8 \Psi_9 \Psi_{10} \Psi_{11} \Psi_2 \Psi_8 \Psi_7 \Psi_6 \Psi_4 \Psi_5 \Psi_6 \Psi_3 \Psi_4 (\Psi_8^2 (\underline{c}'') e(\nu'')) \Psi_7 \Psi_9 \Psi_{10} \Psi_8 \Psi_9 (\underline{c})m^\la\\
= &-(\Psi_8 \Psi_9 \Psi_8 (\underline{c}''') e(\nu''')) \Psi_{10} \Psi_{11} \Psi_2 \Psi_7 \Psi_6 \Psi_4 \Psi_5 \Psi_6 \Psi_3 \Psi_4 \Psi_7 \Psi_9 \Psi_{10} \Psi_8 \Psi_9 (\underline{c})m^\la,\\
= &-\Psi_9 \Psi_8 (\Psi_9 \Psi_{10} \Psi_9 (\underline{c}''') e(\omega)) \Psi_{11} \Psi_2 \Psi_7 \Psi_6 \Psi_4 \Psi_5 \Psi_6 \Psi_3 \Psi_4 \Psi_7 \Psi_{10} \Psi_8 \Psi_9 (\underline{c})m^\la,
\end{align*}
where
\begin{align*}
\underline{c}' &= (a^1_l, 1, a^2_l -1, a^3_l - 2, 1, 1, a^1_r -2, 1, a^2_m, 1, a^2_r - 1, 1, a^3_r),\\
\nu' &= \nu^1_l*(i+1)* \dot{\nu}^2_l* \ddot{\nu}^3_l* (i+2)* (i+1)* \ddot{\nu}^1_r* (i+2)* \nu^2_m* (i+1)* \dot\nu^2_r* (i+1)*\nu^3_r,\\
\underline{c}'' &= (a^1_l, 1, 1, a^1_r-2, a^2_l-1, 1, a^3_l-2, a^2_m, 1, 1, a^2_r-1, 1, a^3_r),\\
\nu'' &= \nu^1_l* (i+1)* (i+2)* \ddot{\nu}^1_r* \dot{\nu}^2_l* (i+1)* \ddot{\nu}^3_l* \nu^2_m* (i+2)* (i+1)* \dot\nu^2_r* (i+1)*\nu^3_r,\\
\underline{c}''' &= (a^1_l, a^2_l -1, 1, a^3_l-2, 1, 1, a^2_m, a^1_r -2, 1, 1, 1, a^2_r-1, a^3_r),\\
\nu''' &= \nu^1_l* \dot{\nu}^2_l* (i+1)* \ddot{\nu}^3_l* (i+2)* (i+1)* \nu^2_m* \ddot{\nu}^1_r* (i+2)* (i+1)* (i+1)* \dot\nu^2_r* \nu^3_r,\\
\omega &= \nu^1_l* \dot{\nu}^2_l* (i+1)* \ddot{\nu}^3_l* (i+2)* (i+1)* \nu^2_m* \ddot{\nu}^1_r* (i+1)* (i+2)* (i+1)* \dot\nu^2_r* \nu^3_r.
\end{align*}
We use
\[
(\Psi_2 \Psi_1 \Psi_2 - \Psi_1 \Psi_2 \Psi_1) (1,1,1) e((i+1)*(i+2)*(i+1)) = e((i+1)*(i+2)*(i+1))
\]
to continue as follows.
\begin{align*}
= &-\Psi_9 \Psi_8 \Psi_{10} \Psi_9 (\Psi_{10} \Psi_{11} \Psi_{10} (\underline{e}) e(\omega')) \Psi_2 \Psi_7 \Psi_6 \Psi_4 \Psi_5 \Psi_6 \Psi_3 \Psi_4 \Psi_7  \Psi_8 \Psi_9 (\underline{c})m^\la\\
&+ \Psi_9 \Psi_8 \Psi_{11} \Psi_2 \Psi_7 \Psi_6 \Psi_4 \Psi_5 \Psi_6 \Psi_3 \Psi_4 \Psi_7 \Psi_{10} \Psi_8 \Psi_9 (\underline{c})m^\la\\
= &-\Psi_9 \Psi_8 \Psi_{10} \Psi_9 \Psi_{11} \Psi_{10} \Psi_2 \Psi_7 \Psi_6 \Psi_4 \Psi_5 \Psi_6 \Psi_3 \Psi_4 \Psi_7  \Psi_8 \Psi_9 \Psi_{11} (\underline{c})m^\la\\
&+ \Psi_4 \Psi_2 \Psi_3 \Psi_9 \Psi_8 \Psi_7 \Psi_6 \Psi_5 \Psi_6 \Psi_7 \Psi_8 \Psi_{11} \Psi_{10} \Psi_9 \Psi_4 (\underline{c})m^\la\\
= &+ \Psi_4 \Psi_2 \Psi_3 \Psi_9 (\Psi_8 \Psi_7 \Psi_6 \Psi_5 \Psi_6 \Psi_7 \Psi_8) \Psi_{11} \Psi_{10} \Psi_9 \Psi_4 (\underline{c})m^\la,
\hypertarget{error term1}{\tag{\textdagger}}
\end{align*}
where, in the final equality,  we have used that $\Psi_{11} (\underline{c}) m^\la = \psi_{a^1_l + a^1_r + a^2_l + a^2_m + a^2_r + a^3_l - 1} m^\la = 0$ since $a^3_l\geq 2$,
and $\underline{e} = (a^1_l, a^2_l -1, 1, a^3_l-2, 1, 1, a^2_m, a^1_r -2, 1, a^2_r-1, 1, 1, a^3_r)$,
\[
\omega' =  \nu^1_l* \dot{\nu}^2_l* (i+1)* \ddot{\nu}^3_l* (i+2)* (i+1)* \nu^2_m* \ddot{\nu}^1_r* (i+1)* \dot\nu^2_r* (i+2)* (i+1)* \nu^3_r.
\]

Similarly, we compute $\delta$ as in the proof of \cref{lem: wAgA-wBgB 2}(2) by using the same equality as above. We replace $\underline{d}$ with $\underline{c}$:
\[
\underline{c}=(a^1_l,1,1,a^1_r-2,a^2_l-1,1,a^2_m,1,a^2_r-1,a^3_l-2,1,1,a^3_r)
\]
and write $\nu=\nu^1_l*(i+1)*(i+2)*\ddot{\nu}^1_r*\dot{\nu}^2_l*(i+1)*\nu^2_m*(i+1)*\dot{\nu}^2_r*\ddot{\nu}^3_l*(i+2)*(i+1)*\nu^3_r$. Then
\begin{align*}
\delta = &-(\Psi_5\Psi_4)(\Psi_3\Psi_2)(\Psi_6\Psi_5)(\Psi_{11})(\Psi_9\Psi_8\Psi_6\Psi_7\Psi_8)(\Psi_{10}\Psi_9)(\Psi_5\Psi_6)\Psi_4\Psi_5\Psi_3\Psi_4(\underline{c})m^\la\\
= &-\Psi_5 \Psi_4 \Psi_3 \Psi_2 \Psi_9 \Psi_8 (\Psi_6 \Psi_5 \Psi_6 (\underline{c}') e(\nu')) \Psi_{11} \Psi_7 \Psi_8 \Psi_{10} \Psi_9 \Psi_5 \Psi_6 \Psi_4 \Psi_5 \Psi_3 \Psi_4 (\underline{c}) m^\la\\
= &-\Psi_5 \Psi_4 \Psi_5 \Psi_3 \Psi_2 \Psi_9 \Psi_8 \Psi_6 \Psi_7 \Psi_8 (\Psi_{11} \Psi_{10} \Psi_9) (\Psi_5^2 (\underline{c}'') e(\nu'')) (\Psi_4 \Psi_3) \Psi_6 \Psi_5 \Psi_4 (\underline{c}) m^\la\\
= &-\Psi_5 \Psi_4 \Psi_5 \Psi_3 \Psi_2 \Psi_9 \Psi_6 (\Psi_4 \Psi_3) (\Psi_8 \Psi_7 \Psi_8 (\underline{c}''') e(\nu''')) (\Psi_{11} \Psi_{10} \Psi_9) \Psi_6 \Psi_5 \Psi_4 (\underline{c}) m^\la\\
= &-(\Psi_5 \Psi_4 \Psi_5 (\underline{d}) e(\omega)) \Psi_3 \Psi_2 \Psi_4 \Psi_3 \Psi_9 \Psi_6 \Psi_7 \Psi_8 \Psi_7 \Psi_{11} \Psi_{10} \Psi_9 \Psi_6 \Psi_5 \Psi_4 (\underline{c}) m^\la\\
= &-\Psi_4 \Psi_5 (\Psi_4 \Psi_3 \Psi_4 (\underline{d}') e(\omega')) \Psi_2 \Psi_3 \Psi_9 \Psi_6 \Psi_7 \Psi_8 \Psi_7 \Psi_{11} \Psi_{10} \Psi_9 \Psi_6 \Psi_5 \Psi_4 (\underline{c}) m^\la\\
= &-\Psi_4 \Psi_3 \Psi_5 \Psi_4 (\Psi_3 \Psi_2 \Psi_3 (\underline{d}'') e(\omega'')) \Psi_9 \Psi_6 \Psi_7 \Psi_8 \Psi_7 \Psi_{11} \Psi_{10} \Psi_9 \Psi_6 \Psi_5 \Psi_4 (\underline{c}) m^\la\\
&-\Psi_4 \Psi_5 \Psi_2 \Psi_3 \Psi_9 \Psi_6 \Psi_7 \Psi_8 \Psi_7 (\Psi_{11} \Psi_{10} \Psi_9) (\Psi_6 \Psi_5) \Psi_4 (\underline{c}) m^\la\\
= &-\Psi_4 \Psi_3 \Psi_5 \Psi_4 \Psi_2 \Psi_3 \Psi_9 \Psi_6 \Psi_7 \Psi_8 \Psi_7 \Psi_{11} \Psi_{10} \Psi_9 \Psi_6 \Psi_5 \Psi_4 \Psi_2 (\underline{c}) m^\la\\
&-\Psi_4 \Psi_2 \Psi_3 \Psi_9 (\Psi_5 \Psi_6 \Psi_7 \Psi_8 \Psi_7 \Psi_6 \Psi_5) \Psi_{11} \Psi_{10} \Psi_9 \Psi_4 (\underline{c}) m^\la\\
= &-\Psi_4 \Psi_2 \Psi_3 \Psi_9 (\Psi_5 \Psi_6 \Psi_7 \Psi_8 \Psi_7 \Psi_6 \Psi_5) \Psi_{11} \Psi_{10} \Psi_9 \Psi_4 (\underline{c}) m^\la
\hypertarget{error term2}{\tag{\ddag}}
\end{align*}
where we have used \cref{square of block transposition}(2) for the fourth equality, $\Psi_2 (\underline{c}) m^\la = \psi_{a^1_l+1} m^\la = 0$ (since $a^1_r \geq a^2_r +1\geq 2$) in the final equality, and
\begin{align*}
\underline{c}' &= (a^1_l, 1, a^2_l -1, 1, a^2_m, 1, a^3_l - 2, a^1_r -2, 1, 1, 1, a^2_r - 1, a^3_r),\\
\nu' &= \nu^1_l* (i+1)* \dot{\nu}^2_l* (i+1)* \nu^2_m* (i+2)* \ddot{\nu}^3_l* \ddot{\nu}^1_r* (i+1)* (i+2)* (i+1)* \dot\nu^2_r* \nu^3_r,\\
\underline{c}'' &= (a^1_l, 1, a^2_l -1, 1, 1, a^2_m, a^1_r -2, 1, a^2_r - 1, a^3_l - 2, 1, 1, a^3_r),\\
\nu'' &= \nu^1_l*(i+1)* \dot{\nu}^2_l* (i+1)* (i+2)* \nu^2_m* \ddot{\nu}^1_r* (i+1)* \dot\nu^2_r* \ddot{\nu}^3_l*  (i+2)* (i+1)* \nu^3_r,\\
\underline{c}''' &= (a^1_l, 1, 1, a^2_l -1, 1, a^2_m, a^1_r -2, 1, a^3_l - 2, 1, 1, a^2_r - 1, a^3_r),\\
\nu''' &= \nu^1_l* (i+1)* (i+2)* \dot{\nu}^2_l* (i+1)* \nu^2_m* \ddot{\nu}^1_r* (i+1)* \ddot{\nu}^3_l* (i+2)* (i+1)* \dot\nu^2_r* \nu^3_r,\\
\underline{d} &= (a^1_l, a^2_l -1, 1, 1, 1, a^3_l - 2, a^2_m, 1, 1, a^1_r -2, 1, a^2_r - 1, a^3_r), \quad \underline{d}' = \underline{d}\\
\omega &= \nu^1_l* \dot{\nu}^2_l* (i+1)* (i+1)* (i+2)* \ddot{\nu}^3_l* \nu^2_m* (i+1)* (i+2)* \ddot{\nu}^1_r* (i+1)* \dot\nu^2_r* \nu^3_r,\\
\omega' &= \nu^1_l* \dot{\nu}^2_l* (i+1)* (i+2)* (i+1)* \ddot{\nu}^3_l* \nu^2_m* (i+1)* (i+2)* \ddot{\nu}^1_r* (i+1)* \dot\nu^2_r* \nu^3_r,\\
\underline{d}'' &= (a^1_l, 1, 1, a^2_l -1, 1, a^3_l - 2, a^2_m, 1, 1, a^1_r -2, 1, a^2_r - 1, a^3_r),\\
\omega'' &= \nu^1_l* (i+1)* (i+2)* \dot{\nu}^2_l* (i+1)* \ddot{\nu}^3_l* \nu^2_m* (i+1)* (i+2)* \ddot{\nu}^1_r* (i+1)* \dot\nu^2_r* \nu^3_r.
\end{align*}
It's easy to see, by applying three further braid relations which don't yield error terms, that \hyperlink{error term1}{(\textdagger)} and \hyperlink{error term2}{(\ddag)} 
are negations of each other, so that $\gamma + \delta = 0$, and the proof is complete.
\end{proof}

\subsection{Proof of \cref{Thm: Specht modules}}\

\begin{lem} \label{Lem: G^AB 3}
Let $A$ and $B$ be Garnir nodes of $[\la]$. Let $w^A$ and $w^B$ be the fully commutative elements given in \cref{lem: wA wB}.
Then we have
\begin{align*}
\psi_{w^A} g_A^\lambda \equiv \psi_{w^B} g_B^\lambda \quad \pmod{\Ga^{\lambda}_{ < \ell(\ttg^{A,B})}}
\end{align*}
\end{lem}
\begin{proof}
We may assume that $A = (r,c)$ is to the left of $B = (r',c')$ in $[\la]$.

Suppose $r = r'+1$ and $\res B  = \res {r+1,c}$. Without loss of generality, we may assume that $[\la]$ has 3 rows and $r'=1$. 
Then the assertion holds by \cref{lem: wAgA-wBgB 4}.

Otherwise, it follows from \cref{psi_w} and \cref{Cor: homo-r for C_infty} that
\[
\psi_{w^A} g_A^\lambda - \psi_{w^B} g_B^\lambda = \sum_{w } a_w \psi_w m^\lambda \quad \text{ for some $a_w \in \bR$},
\]
where $w$ runs over all elements such that (i) $w \prec w^{\ttg^{A,B}}$ and (ii) $e(\res {\ttg^{A,B}} )\psi_w m^\lambda = \psi_w m^\lambda$.
By \cref{lem: T and G AB}, we conclude that
\[
\psi_{w^A} g_A^\lambda - \psi_{w^B} g_B^\lambda = 0.\qedhere
\]
\end{proof}

\begin{lem} \label{lem: closedness}
Let $\ttt \in \rT \la$.
Suppose that $\ttt = w \ttg^A$ for $w \in \sg_n$ and a Garnir node $A$ of $[\la]$ with $\ell(\ttt) = \ell(w) + \ell(\ttg^A)$.
\begin{enumerate}
\item If $\ttt = u \ttg^B $ for $u \in \sg_n$ and a Garnir node $B$ of $[\la]$ with
$ \ell(\ttt)  = \ell(u) + \ell(\ttg^B)$, then
\begin{align*}
\psi_{w} g_A^\lambda \equiv \psi_{u} g_B^\lambda \quad \pmod{\Ga^{\lambda}_{ < \ell(\ttt)}}.
\end{align*}
\item   We have
\begin{align*}
\left\{
  \begin{array}{ll}
    x_i \psi_{w} g_A^\lambda \equiv 0 & \hbox{ for all } i\in \{1,\dots,n\}, \\
    \psi_j \psi_{w} g_A^\lambda \equiv 0 & \hbox{ unless $s_j \ttt \in \rT\la$ and $s_j\ttt \domsby \ttt$},
  \end{array}
\right.
 \quad \pmod{\Ga_{< \ell(\ttt)}^\lambda}.
\end{align*}
\item For $\sigma \in \sg_n$ with $\ell(\sigma) + \ell(\ttg^A) < \ell(\ttt)$,
\[
\psi_\sigma g^\la_A \equiv 0 \quad \pmod{\Ga^\la_{< \ell(\ttt)}}.
\]
\end{enumerate}

\end{lem}
\begin{proof}

First, we prove (1), (2), and (3) for Garnir tableaux. Let $\ttt = \ttg^A$.
If $\ttg^A  = u \ttg^B$ for some $u\in \sg_n$ and some Garnir node $B$ of $[\la]$ with $ \ell(\ttt) = \ell(u) + \ell(\ttg^B)$, 
then it implies $\ttg^A \ge_L \ttg^B$, so that $\ttg^{A,B} = \ttg^A$ follows. Thus $B=A$ and $u = \mathrm{id}$ by Lemma \ref{construction of G^{A,B}}, proving (1).
Assertion (3) also holds obviously since there is no $\sigma\in \sg_n$ such that $\ell(\sigma) + \ell(\ttg^A) < \ell(\ttg^A)$. Assertion (2) follows from \cref{lem: base}.

Now we prove (1), (2), and (3) by induction on $l := \ell(\ttt)$. If $l = \min\{\ell(\ttt)\mid \ttt \in \rT\la\}$, then $\ttt$ is a Garnir tableau and there is nothing to prove.
We assume that (1), (2), and (3) hold for all $\ttt' = w' \ttg^{A'} \in \rT\la$ with 
\[
\ell(\ttt') = \ell(w') + \ell(\ttg^{A'}) < l.
\]

(1) We consider $\ttt = w \ttg^A = u \ttg^B$ for some $w, u \in \sg_n$ and Garnir nodes $A,B$ of $[\la]$ with $\ell(\ttt) = l = \ell(w) + \ell(\ttg^A) =\ell(u) + \ell(\ttg^B)$. By \cref{garjoin}, there is $v \in \sg_n$ such that
\[
\ttt = v \ttg^{A,B} \text{ with } \ell(T) = \ell(v) + \ell(\ttg^{A,B}),
\]
which tells us that
\begin{align*}
w &= v w^A \ \text{ with } \ell(w) = \ell(v) + \ell(w^A), \\
u &= v w^B\ \text{ with } \ell(u) = \ell(v) + \ell(w^B),
\end{align*}
where $w^A$ and $w^B$ are given in \cref{lem: wA wB}. Then, by \cref{psi_w} and the induction hypothesis on (3), we know that if $\psi_\sigma$ appears as an error term in 
$(\psi_w-\psi_v\psi_{w^A}) e(\res {\ttg^A} )$ then $\psi_\sigma g^\la_A\in \mathfrak \ttg^\la_{< \ell(\ttt)}$. Similarly, if $\psi_\sigma$ appears as an error term in 
$(\psi_u-\psi_v\psi_{w^B}) e(\res {\ttg^B})$ then $\psi_\sigma g^\la_B\in \mathfrak \ttg^\la_{< \ell(\ttt)}$.

By \cref{Lem: G^AB 3} and the induction hypothesis, we have 
\begin{equation}
\begin{aligned}
\psi_{w} g_A^\lambda - \psi_{u} g_B^\lambda &\equiv \psi_{v} \psi_{w^A} g_A^\lambda - \psi_{v} \psi_{w^B}  g_B^\lambda \\
&\equiv \psi_{v} (\psi_{w^A} g_A^\lambda -  \psi_{w^B}  g_B^\lambda) \\
&\equiv 0 \qquad \qquad  \qquad \qquad \qquad \pmod{\Ga^{\lambda}_{ < \ell(\ttt)}}.
\end{aligned}
\end{equation}
Thus, (1) holds for $\ell(\ttt)=l$.

(2) For $i = 1, \ldots, n$, it follows from
\[
x_i \psi_w e( \res{\ttg^A}) = \psi_w x_{w^{-1}(i)} e(\res{\ttg^A}) + \sum_{ w' \prec w } \psi_{w'} f_{w'} e(\res{\ttg^A}) \quad \text{ for $f_{w'} \in \bR[x_1, \dots, x_n]$}
\]
that
\[
x_i \psi_w g_A^\la = \psi_w x_{w^{-1}(i)} g_A^\la + \sum_{ w' \prec w } \psi_{w'} f_{w'} g_A^\la \equiv 0 \quad \pmod{\Ga^{\lambda}_{ < \ell(\ttt)}}
\]
by the induction hypothesis. It remains to prove that $\psi_j \psi_w g_A^\lambda \equiv 0$ unless $s_j \ttt \in \rT\la$ and $s_j \ttt \domsby \ttt$. There are two cases: 
\begin{itemize}
\item[(i)]
$s_j \ttt \notin \rT \la$ (i.e.~$s_j \ttt \notin \rST \la$ or $s_j \ttt \in \ST \la$),
\item[(ii)]
$s_j \ttt \in \rT \la$ and $s_j \ttt \vartriangleright \ttt $.
\end{itemize}

\bigskip
\begin{enumerate}
\item[(i)] If $s_j \ttt \notin \rST \la$, then there is a node $(r,c) \in [\la]$ such that
\[
\ttt(r,c) = j, \quad \ttt(r,c+1) = j+1.
\]
By \cref{lem: garnir} (2), we can take $B \in [\la]$ and $u \in \sg_n$ such that
\begin{enumerate}
\item[(a)] $\ttg^B(r,c+1) = \ttg^B(r,c)+1$,
\item[(b)] $\ttt = u \ttg^{B}$,
\item[(c)] $s_j u = u s_p$ where $p = \ttg^{B} (r,c)$.
\end{enumerate}
Thus, we have
\[
\psi_j \psi_u e(\nu) = \psi_u \psi_p e(\nu) + \sum_{u' \prec u } \psi_{u'} f_{u'} e(\nu) \quad \text{ for $f_{u'} \in \bR[x_1, \dots, x_n]$},
\]
where $\nu =  e(\res{\ttg^B})$.
Since (1) holds for the length $l$, 
\cref{psi_w} implies that $\psi_w g_A^\la - \psi_u g_B^\la \in \Ga^\la_{\leq l -3}$, so that $\psi_j \psi_w g_A^\la \equiv \psi_j \psi_u g_B^\la \pmod{\Ga^\la_{<\ell(\ttt)}}$ by the induction hypothesis on (3). Similarly, \cref{psi_w,lem: base} imply that $\psi_j \psi_u g_B^\la \equiv \psi_u \psi_p g_B^\la \pmod{\Ga^\la_{<\ell(\ttt)}}$, and $\psi_p g_B^\la = 0$, so that
\begin{align*}
\psi_j \psi_w g_A^\la \equiv \psi_j \psi_u g_B^\la
\equiv  \psi_u \psi_p g_B^\la  = 0 \quad \pmod{\Ga^{\la}_{< \ell(\ttt)}}.
\end{align*}

Suppose that $s_j \ttt \in \ST \la$. Then there is a node $ C= (r,c) \in [\la]$ such that
\[
\ttt(r,c) = j+1, \quad \ttt(r+1,c) = j.
\]
By \cref{lem: garnir}(1), there is a permutation $v \in \sg_n$ such that
\begin{enumerate}
\item[(a)] $\ttt = v \ttg^{C}$,
\item[(b)] $s_j v = v s_q$ where $q = \ttg^{C} (r+1,c)$.
\end{enumerate}
In a similar manner as above, we have
\begin{align*}
\psi_j \psi_w g_A^\lambda \equiv \psi_j \psi_v g_C^\lambda
\equiv  \psi_v \psi_q g_C^\lambda = 0 \quad \pmod{\Ga^{\lambda}_{ < \ell(\ttt)}}.
\end{align*}

\item[(ii)] We assume that $ \tts := s_j \ttt \in \rT \la$ and $ \tts \vartriangleright \ttt $.
Note that $\ell(\ttt) = \ell(\tts) + 1$. Then
\[
\tts = u \ttg^B \ \text{ and } \ell(\tts) = \ell(u) + \ell(\ttg^B)
\]
for some Garnir node $B \in [\la]$ and $u \in \sg_n$.
As (1) holds for the length $l$, we have
\begin{align*}
\psi_j \psi_w g_A^\la \equiv \psi_j \psi_j \psi_u g_B^\la
\equiv  Q_{\nu_j, \nu_{j+1}}(x_j, x_{j+1}) \psi_u g_B^\la \equiv 0 \quad \pmod{\Ga^{\la}_{ < \ell(\ttt)}},
\end{align*}
where $\nu = (\nu_1, \dots, \nu_n) \in I^n$ is the residue sequence of $u \ttg^{B}$.
\end{enumerate}
By (i) and (ii) the assertion (2) holds for $\ell(\ttt) = l$.

(3) Suppose that $\ttt = w \ttg^A$ for $w \in \sg_n$ and a Garnir node $A$ of $[\la]$ with $\ell(\ttt) = \ell(w) + \ell(\ttg^A)$.
Let $\sigma = s_{i_1}\dots s_{i_r}$ be a reduced expression for $\sigma$, so that $\psi_\sigma = \psi_{i_1}\dots \psi_{i_r}$ and $\ell(\sigma) + \ell(\ttg^A)  < l$. 
Note that we do not assume that $\ell(\sigma \ttg^A) = \ell(\sigma) + \ell(\ttg^A)$.

If $\ell(\sigma \ttg^A) = \ell(\sigma) + \ell(\ttg^A)$, then  $\sigma \ttg^A \in \rT\la$ by \cref{straighten}. Thus, we are done.

If $\ell(\sigma \ttg^A) < \ell(\sigma) + \ell(\ttg^A)$, then there is some $k$ such that 
$\ell(s_{i_k} \dots s_{i_r} \ttg^A) = \ell(s_{i_k} \dots s_{i_r}) + \ell(\ttg^A)$
and
$\ell(s_{i_{k-1}} \dots s_{i_r} \ttg^A) < \ell(s_{i_k} \dots s_{i_r} \ttg^A)$.
Once again, by the induction hypothesis on (2), we have that $\psi_{i_{k-1}} \dots \psi_{i_r} g_A^\la \in \Ga^\la_{< \ell(s_{i_k} \dots s_{i_r} \ttg^A)}$ and may conclude that $\psi_{i_1} \dots \psi_{i_r} g_A^\la \in \Ga^\la_{< l}$ by induction.

Thus, assertion (3) holds for $\ell(\ttt) = l$, which completes the proof.
\end{proof}

\begin{cor}\label{reasoning for the next theorem}
For each $\ttt \in \rT\la$, we fix a Garnir node $A$ and $w \in \sg_n$ such that $\ttt = w \ttg^A$, and set 
$g_\ttt^\lambda = \psi_w g_A^\la+\Ga^{\lambda}_{< \ell(\ttt)}$. Then
\begin{enumerate}
\item
The element $g_\ttt^\lambda \in \Ga^{\lambda} / \Ga^{\lambda}_{< \ell(\ttt)}$ does not depend on the choice of $A$ or the choice of reduced expression for $w$.
\item
The $\bR$-submodule $\sum_{t>0} \Ga^\lambda_{< t}$ is an $R(\beta)$-submodule of $\Pe^\lambda$.
\end{enumerate}
\end{cor}
\begin{proof}
Parts (1) and (2) follow from \cref{lem: closedness} (1) and \cref{lem: closedness} (2), respectively.
\end{proof}

Recall that $\Ga^\la = \im \  \Gh^{\la}$ and $\Sp^{\la} = q^{\deg(\ttt^\la)} \coker\ \Gh^\la$ for $\la \vdash n$.

\begin{thm} \label{thm: basis of g^la}
\begin{enumerate}
\item The $\bR$-submodules $\{ \Ga^\la_{< t}  \}_{t \in \Z_{>0}}$ give a filtration of $\Ga^\la$.
\item For $t\in \Z_{>0}$, $\Ga^{\la}_{< t+1} / \Ga^{\la}_{< t} $ is a free $\bR$-module with basis
\[
\{  g_\ttt^\lambda \mid  \ttt \in \rT\la, \ell(T)=t \}.
\]
\end{enumerate}
\end{thm}
\begin{proof}
Corollary \ref{reasoning for the next theorem}(2) implies that 
$\Ga^\lambda = \sum_{t > 0} \Ga^\lambda_{< t}$, which is (1). Then (2) is clear.
\end{proof}

We are now ready to prove \cref{Thm: Specht modules}.

\begin{proof}[Proof of \cref{Thm: Specht modules}] \
Let us consider
\[
v = a_{\ttt^1} \psi_{w^{\ttt_1}}  m^{\lambda} + a_{\ttt^2}\psi_{w^{\ttt_2}} m^{\lambda} + \dots + a_{\ttt^t} \psi_{w^{\ttt_t}} m^{\lambda} \in \Pe^\la
\]
for some $a_{\ttt^1}, \dots, a_{\ttt^t} \in \bR$ and  $\ttt_1, \dots, \ttt_t \in \ST \la$. By $\eqref{basis of M}$ and \cref{thm: basis of g^la}, we have
\[
v \in  \Ga^\lambda \ \text{ if and only if }  a_{\ttt^1}= \dots = a_{\ttt^t}=0,
\]
which implies that $\{ \psi_{w^\ttt}\overline{m}^{\lambda} \mid \ttt \in \ST\la  \}$ is linearly independent in $\Sp^\la$. 
Thus the assertion follows from \cref{Thm for Specht}(1).
\end{proof}

\section{A conjecture in type $C^{(1)}_\ell$} \label{Sec: conjecture}

We end with a conjecture giving the elements $g^\la_A$ explicitly in type $C^{(1)}_\ell$, as well as a basis of $\Sp^{\la}$ in this type. In Remark \ref{garnirformremark}, we noted the similarity between the Garnir elements in type $A_\infty$ and $C_\infty$. Similarly, we expect that the affine type $C$ case resembles that of affine type $A$. Our constructions in this section closely follow \cite[Section 5]{kmr}. We fix $\la \in \Par_n$ and $\kappa\in \Z^l$ throughout, as well as a Garnir node $A = (r,c,t)\in[\la]$. Recall the definition of the Garnir tableau $\ttg^A$ from earlier, as well as the residue pattern in type $C^{(1)}_\ell$ -- in particular the natural projection $p:\Z \rightarrow \Z/ 2\ell\Z$.

\begin{defn}
A \emph{brick} is a set of $2\ell$ adjacent nodes in the same row of the Garnir belt $\belt^A$, $\{(a,b,t), (a,b+1,t),\dots, (a,b+2\ell-1,t)\}$, such that $p(\kappa_t +b - a) = p(\kappa_t +c - r)$.
\end{defn}

We denote by $k$ the number of bricks contained in $\belt^A$, and label the bricks $B_1,B_2,\dots,B_k$ from left-to-right along row $r+1$ and then from left-to-right along row $r$. We now introduce permutations which transpose adjacent bricks. Let $d$ be the smallest entry of $B_1$ in $\ttg^A$. Then for $1\leq r< k$, the permutation
\[
w_r = w_r^A := \prod_{\mathclap{a=d+2(r-1)\ell}}^{d+2r\ell-1} (a,a+2\ell) \in \sg_n
\]
may be thought of as transposing $B_r$ and $B_{r+1}$. We have the corresponding elements $\sigma_r = \sigma_r^A := (-1)^\ell \psi_{w_r} \in R(\beta)$. We further define $\tau_r = \tau_r^A := \sigma_r + 1$.
We should emphasise here that we have a $(-1)^\ell$ in our definition of $\sigma_r$, which differs from the definition of \emph{row} Specht modules in \cite[Section 5.4]{kmr}. However a similar minus sign occurs in their definition of corresponding elements in \emph{column} Specht modules \cite[Section 7.1]{kmr}. We suspect that this minus sign is merely an artefact of our choice of the polynomials $Q_{i,j}(u,v)$.

Note that any permutation $u$ of bricks may be written as a reduced expression $u = w_{r_1} \dots w_{r_i}$, and if $u$ is fully commutative we have a well-defined element $\tau_u := \tau_{r_1} \dots \tau_{r_i}$.

We define $\ttt^A$ to be the tableau obtained from $\ttg^A$ by rearranging the entries in the bricks $B_1,\dots,B_k$ so that they are in order along row $r$ and then row $r+1$. This is the most dominant row-strict tableau which may be obtained from $\ttg^A$ by acting by our brick permutations $w_r$.

\begin{ex}
Let $\ell=2, \la=(15,7,3)\in \mathscr{P}^1_{25}$, and $A=(1,5)$. We depict the Garnir tableau $\ttg^A$ below with Garnir belt shaded and bricks $B_1,B_2,B_3$ labelled, as well as the tableau $\ttt^A$.
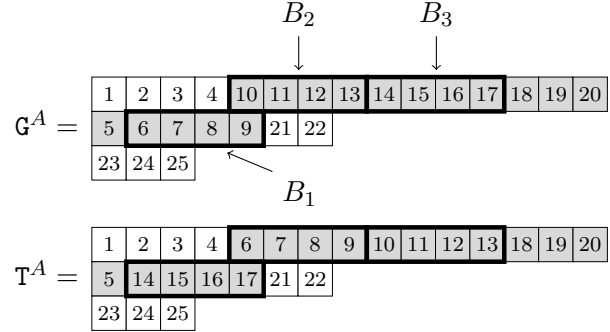
\[ 
\begin{tikzpicture}

\draw[->, thin] (2.75,1) node[above] {$B_2$} -- (2.75,.55);
\draw[->, thin] (4.58,1) node[above] {$B_3$} -- (4.58,.55);
\draw[->, thin] (2.4,-.8) node[below right] {$B_1$} -- (1.8,-.55);

\draw (-.6,-.15) node{$\ttg^A = $};
\draw (-.6,-2.15) node{$\ttt^A = $};

\scriptsize{\tyoung(0cm,0cm,1234!\gr<10><11><12><13><14><15><16><17><18><19><20>,56789!\wh<21><22>,<23><24><25>)}
{\Ylinethick{1.7pt}\Yfillopacity{0}
\tgyoung(0cm,0cm,^4_4_4,:_4)}

\scriptsize{\tyoung(0cm,-2cm,1234!\gr6789<10><11><12><13><18><19><20>,5<14><15><16><17>!\wh<21><22>,<23><24><25>)}
{\Ylinethick{1.7pt}\Yfillopacity{0}
\tgyoung(0cm,-2cm,^4_4_4,:_4)}
\end{tikzpicture}
\]
Note that $\ttg^A = w_1 w_2 \ttt^A$.
\end{ex}

\begin{conj}
Let $\la\in \Par_n$, and let $A\in [\la]$ be a Garnir node. Suppose $\belt^A$ contains $a$ bricks in the first row, and $b$ in the second. Then in type $C^{(1)}_\ell$, the Garnir element $g^\la_A$ is
\[
g^\la_A = \sum_u \tau^A_u \psi_{w^{\ttt^A}} m^\la,
\]
where the sum is over all $u \in \sg_{a+b}/\sg_a\times\sg_b$.
Furthermore, \cref{Thm: Specht modules} and \cref{Cor: Specht modules} hold in type $C^{(1)}_\ell$, giving a homogeneous basis of $\Sp^\la$ indexed by standard $\la$-tableaux.
\end{conj}

\begin{ex}\label{affgarex}
Continuing from the previous example,
\begin{align*}
g^\la_A &= \psi_{w^{\ttt^A}} m^\la + \tau_2 \psi_{w^{\ttt^A}} m^\la + \tau_1 \tau_2 \psi_{w^{\ttt^A}} m^\la\\
&= 3\psi_{w^{\ttt^A}} m^\la + 2 \sigma_2 \psi_{w^{\ttt^A}} m^\la + \sigma_1 \sigma_2 \psi_{w^{\ttt^A}} m^\la + \sigma_1 \psi_{w^{\ttt^A}} m^\la.
\end{align*}
\end{ex}

Evidence for our conjecture is provided by many examples we have computed in GAP~\cite{GAP4}.

Finally, we note that the above form for $g^\la_A$ does not instantly yield a clean expression in terms of basis elements of $\Pe^\la$ -- this can be seen in Example \ref{affgarex} where $w_1 \ttt^A m^\la$ is not row-strict. Fayers~\cite{fdyck} has addressed this problem in type $A$, and in fact if our conjecture holds, then his work automatically applies to our $g^\la_A$ too.

\vskip 1em


\bibliographystyle{amsplain}

\end{document}